\documentclass[12pt,reqno]{amsart}
\usepackage{txfonts}
\usepackage{a4wide}
\allowdisplaybreaks \numberwithin{equation}{section}
\usepackage{color}
\usepackage{exscale}
\usepackage{relsize}
\usepackage{graphicx}
\usepackage[pagewise]{lineno}
\usepackage[titletoc, title]{appendix}

 \usepackage{verbatim}
\numberwithin{equation}{section}

\newtheorem{theorem}{Theorem}[section]
\newtheorem{proposition}[theorem]{Proposition}
\newtheorem{corollary}[theorem]{Corollary}
\newtheorem{lemma}[theorem]{Lemma}
\newtheorem*{theoremA}{Theorem A}
\newtheorem*{theoremB}{Theorem B}
\newtheorem*{theoremC}{Theorem C}

\theoremstyle{definition}

\theoremstyle{remark}
\newtheorem{remark}[theorem]{Remark}

\begin{document}
\title[nontrivial  normalized solutions   for   Gross-Pitaevskii systems  ]
{Topological degree and existence of
nontrivial  normalized solutions for  mass-critical    Gross-Pitaevskii systems}

 \author{Fengshuang Gao}

\address{  School of Mathematics, Nanjing University of Aeronautics and Astronautics, Nanjing, P.R.China}
\email{gfs@nuaa.edu.cn}
 \author{Yuxia  Guo}
\address{  Department of Mathematical Science, Tsinghua University, Beijing, P.R.China}
\email{yguo@mail.tsinghua.edu.cn}
 \author{Shusen Yan}
\address{ School of Mathematics and Statistics, Key Laboratory of Nonlinear Analysis and Applications
		(Ministry of Education), Central China Normal University, Wuhan 430079, China.}
\email{syan@ccnu.ac.cn}

 \author{Weilin Yu}
\address{  School of Mathematics and Statistics, Jiangxi Normal University, Nanchang, P. R. China.}
\email{weilinyu@amss.ac.cn}

\thanks{F. Gao is partially supported by NSFC(No.12201293). Y. Guo is  supported
by National Key R\&D Program (No. 2023YFA1010002) and
 NSFC(No. 12271283).  S.Yan is supported by National Key R\&D Program (No. 2023YFA1010002) and NSFC (No.
12571118).
W. Yu is supported by NSFC (No.12501262) and Jiangxi Provincial Natural Science Foundation (No. 20252BAC200148) }

\begin{abstract}
In this paper, we study the existence of nontrivial solutions for the following   Gross-Pitaevskii system involving  mass-critical exponent:
\[
\left\{
\begin{array}{ll}
-\Delta u_{1}+V_1(x)u_{1}=a_{1}u_{1}^3+\beta u_{1}u_{2}^2+\mu u_{1}& \hbox{ in }\Omega,\\
-\Delta u_{2}+V_2(x)u_{2}=a_{2}u_{2}^3+\beta u_{2}u_{1}^2+\mu u_{2}&\hbox{ in }\Omega,\\
 u_{1},u_{2}\ge 0 &\hbox{ in }\Omega,
 \\
 u_1=u_2=0 &\hbox{ on }\partial\Omega,
\end{array}\right.
\]
with the  constraint
\[
\int_\Omega (u_1^2+u_2^2)=1,
\]
where $\Omega$ is an unbounded smooth domain in $\mathbb{R}^2$,  $a_1, a_2, \beta$ are positive parameters, $V_i$  are  trapping potentials, and $\mu\in\mathbb{R}$ is an unknown  Lagrange multiplier.  We derive the existence of solutions
by computing the Leray-Schauder degree for the parameters  $a_1,a_2, \beta$, which
are away from some critical values. The system may have semi-trivial solutions of the form $(u_1, 0)$
or $(0, u_2)$. Our novelty is that we provide  mechanisms  ensuring that the solutions we find
are nontrivial, i.e., $u_1>0$ and $u_2>0$.

\end{abstract}

\maketitle
\section{introduction}\label{1}
Let $\Omega$ be an unbounded smooth domain in $\mathbb{R}^2$. We consider  the following nonlinear Gross-Pitaevskii system
\begin{equation}\label{1-1}
\left\{
\begin{array}{ll}
-\Delta u_{1}+V_1(x)u_{1}=a_{1}u_{1}^3+\beta u_{1}u_{2}^2+\mu u_{1}& \hbox{ in }\Omega,\\
-\Delta u_{2}+V_2(x)u_{2}=a_{2}u_{2}^3+\beta u_{2}u_{1}^2+\mu u_{2}&\hbox{ in }\Omega,\\
  u_{1},u_{2}\ge 0 &\hbox{ in }\Omega,
 \\
 u_1=u_2=0 &\hbox{ on }\partial\Omega,\end{array}\right.
\end{equation}
with the  constraint
\begin{equation}\label{1-2}
\int_{\Omega} (u_{1}^2+u_{2}^2)=1,
\end{equation}
where  $V_i$ are trapping potentials, $\mu\in\mathbb{R}$ is the Lagrange multiplier, and $a_1, a_2, \beta$ are positive parameters.  The pair $(u_1,u_2)$ belongs to the space $\mathcal{H}=H_1(\Omega)\times H_2(\Omega)$  equipped with the norm
\[
\|(u_1,u_2)\|_{\mathcal{H}}=\|u_1\|_{H_1}+\|u_2\|_{H_2}=\sum\limits_{i=1}^2\left(\int_{\Omega}|\nabla u_i|^2+V_i(x)|u_i|^2\right)^{\frac{1}{2}},
\]
where
\[
H_i(\Omega)=\left\{u\in H_0^1(\Omega):\int_\Omega V_i(x)|u|^2<+\infty\right\}.
\]

 System \eqref{1-1} arises as a relevant model for various physical phenomena, such as two-component Bose-Einstein condensates, nonlinear optics and fluid mechanics.  Its solutions exhibit different qualitative behaviors depending on the  parameters $a_1$, $a_2$ and $\beta$. In particular, a positive (resp. negative) sign of $a_1$, $a_2$ represents a self-focusing (resp. defocusing) intraspecies interaction within each component, while $\beta>0$ (resp. $\beta<0$) implies a cooperative (resp. competitive) interspecies interaction. In this paper, we restrict our investigation to  the case where $a_1, a_2, \beta>0$.

Solutions of \eqref{1-1}--\eqref{1-2} can be found by looking for critical points of the following  functional
\[
\begin{aligned}
E_{a_1,a_2,\beta}(u_1,u_2)=&\int_{\Omega} (|\nabla u_1|^2+|\nabla u_2|^2)+\int_\Omega (V_1(x)|u_1|^2+V_2(x)u_2^2)\\
&-\int_{\Omega}\left(\frac{a_1}{2}|u_1|^4+\frac{a_2}{2}|u_2|^4+\beta |u_1|^2|u_2|^2\right),
\end{aligned}
\]
on the manifold
$
\mathcal{M}=\left\{(u_1,u_2)\in \mathcal{H}:\int_\Omega (u_1^2+ u_2^2)=1\right\}.
$
In this paper, we aim to find solutions to \eqref{1-1}--\eqref{1-2} with $u_1\neq 0$ and $u_2\neq 0$.
Hence, we need certain mechanisms to ensure that we do not end up
with semi-trivial solutions, i.e., solutions of the form
 $(u_1, 0)$ with $u_1>0$ or $(0,u_2)$ with $u_2>0$.

We call a solution $(u_1,u_2)$ of \eqref{1-1}-\eqref{1-2} a {\bf semi-trivial solution} if
one of $u_i$ is zero; a {\bf nontrivial solution} if $u_1\neq 0$ and $u_2\neq 0$.

To obtain a solution for \eqref{1-1}--\eqref{1-2}, the first attempt  is to study the following minimization problem
\begin{equation}\label{1-30-7}
\inf_{(u_1,u_2)\in  \mathcal{M} } E_{a_1,a_2,\beta}(u_1,u_2).
\end{equation}
It is proved in \cite{BC,glwz}  that $E_{a_1,a_2,\beta}$ admits minimizers  on $\mathcal{M}$ if
\[
0<a_1<a^*, \quad 0<a_2<a^*, \quad 0<\beta<\beta_1^*,
\]
where
\begin{equation}\label{1-3}
\beta_1^*=a^*+\sqrt{(a^*-a_1)(a^*-a_2)},\quad a^*=\int_{\mathbb R^2} Q^2,
\end{equation}
and   $Q$ is the unique radial  solution of
\begin{equation}\label{1-4}
-\Delta Q+Q=Q^3\quad \hbox{in }\mathbb{R}^2.
\end{equation}
Moreover, in \cite{glwz},  it is shown that for
fixed $a_i\in (0, a^*)$,
$i=1, 2$, if $V_1$ and $V_2$ have at least one common
minimum point,  then as $\beta\to \beta_1^*-0$, any minimizer $(u_1, u_2)$ of \eqref{1-30-7}
must blow up. In addition,  from   the asymptotic behavior of the minimizer as
$\beta\to \beta_1^*-0$, one can derive that
any minimizer is nontrivial  if  $\beta$ is sufficiently close to $\beta_1^*$.
However, it is not clear whether any  minimizer of \eqref{1-30-7} is nontrivial
for all $0<a_1,a_2<a^*$  and $0<\beta<\beta_1^*$.

Another effective way to prove the existence of nontrivial solutions for
$a_i$ or $\beta$ near certain critical values is to construct blow-up
solutions via a reduction argument as in \cite{lpwy}. With this idea,
in \cite{gx,gy}, it shows that if $V_1$ and $V_2$ have $k$ common non-degenerate critical points,
then \eqref{1-1}--\eqref{1-2} has a nontrivial solution provided
  $0<a_1, a_2<a^*$, and  $\beta$ is close to
$
ka^*+\sqrt{(ka^*-a_1)(ka^*-a_2)}.
$

The existence of nontrivial solutions (even semi-trivial ones) for  \eqref{1-1}--\eqref{1-2}
is largely
unknown for most parameter values $a_i$ and $\beta$, due to the lack of effective methods
to handle normalized solutions with higher energy for mass-critical problems. Indeed, \eqref{1-1}--\eqref{1-2}
is very sensitive to the potentials $V_i$ and the domain $\Omega$.  For example,
if $V_1=V_2=V$ and $\beta>\max(a_1, a_2)$, then from \cite{wy}, we know that
either $(u_1, u_2)$ is semi-trivial, or

 \[
 (u_1,u_2)=\left(\sqrt{\frac{\beta-a_2}{2\beta-(a_1+a_2)}} u, \sqrt{\frac{\beta-a_1}{2\beta-(a_1+a_2)}}u\right),
 \]
where $u$  satisfies
\begin{equation}\label{1-13}
\begin{cases}
-\Delta u+V(x)u=a u^3+\mu u\;\;\;\text{in}\;\; \Omega,\\
0\leq u\in H_0^1(\Omega),
\\
\int_\Omega u^2=1,
\end{cases}
\end{equation}
with $a=\frac{\beta^2-a_1a_2}{2\beta-(a_1+a_2)}$.
Moreover, it is proved in \cite{cyy} that for $V=|x|^2$ and $\Omega=\mathbb R^2$,
\eqref{1-13}  has no solution if $a\ge a^*$.  This result immediately implies
for $V_1=V_2=|x|^2$ and $\Omega=\mathbb R^2$,
\eqref{1-1}--\eqref{1-2} has no nontrivial solution if $\beta>\max\{a_1, a_2,\beta_1^*\}$
(and no semi-trivial solution if $a_i\ge a^*$, $i=1, 2$).

 In this paper, we aim
 to investigate   the effect of the  topology of $\Omega$
  and the properties of $V_i$ on  the existence of nontrivial solutions to \eqref{1-1}-\eqref{1-2} for  $a_1,\, a_2,\,\beta$ , which
are away from some critical values.  To this end, we assume  that
\begin{equation}\label{1-6}
\Omega=\mathbb{R}^2\setminus \cup_{i=1}^k\bar{O}_i,
\end{equation}
where $O_i$, $i=1,\cdots, k$,  are $k$ bounded, simply connected open subsets of $\mathbb{R}^2$ with $\bar{O}_i\cap\bar{O}_j=\emptyset$ for $i\neq j$.  For the potential $V$, we assume it
satisfies the following conditions:

\begin{enumerate}
\item[$(V_1)$] $V\in C^2(\bar{\Omega})$, $\lim\limits_{|x|\to+\infty} V(x)=+\infty$ and $0\leq V(x)\leq Ce^{\alpha|x|}$ for some $\alpha>0$ small.
\item[$(V_2)$] $\nabla V(x)\cdot\vec{n}_x\geq c_0$ as $|x|\to+\infty$, where $c_0>0$ is a constant and $\vec{n}_x=\frac{x}{|x|}$.
\item[$(V_3)$] As $|x|\to+\infty$, we have
\[
\frac{V(x+y)}{V(x)}\to 1, \quad \frac{\nabla V(x+y)\cdot \vec{n}_x}{\nabla V(x)\cdot\vec{n}_x}\to 1,
\]
uniformly for $y$ with $|y|=o(|x|)$.
\item[$(V_4)$] $\frac{\partial V}{\partial \nu}>0$ on $\partial\Omega$, where $\nu$ denotes the outward unit normal to $\partial\Omega$.
\end{enumerate}
Note that $(V_1)$--$(V_4)$ are local conditions on $V$.

 First, we recall the following
result.

\begin{theoremA}[\cite{cyy}]
Suppose that $V$ satisfies $(V_1)$--$(V_4)$. Given any small $\delta>0$, there exists $C_\delta>0$,
such that
 for any $a\in [m a^*+\delta, (m+1) a^*-\delta]$, and any solution $(u,\mu)$ of \eqref{1-13}, it holds 
\[
\|u\|_{L^\infty(\Omega)},\; |\mu|\le C_\delta.
\]

\end{theoremA}
On the other hand,
it follows from Theorem~1.2 in \cite{cyy} that if $a\in (a^*, 2 a^*)$ and $k>0$, then \eqref{1-13} has
a solution.  Also, \eqref{1-13} has
a solution if $a\in (0, a^*)$.

To state our results, we first define $\mathcal{E}$ as follows.

Let $V_1$ be the function in the first equation of \eqref{1-1},
satisfying $(V_1)$--$(V_4)$. We define

\[
\begin{split}
\mathcal{E}:=\Big\{V: & V\in C^2(\bar{\Omega}), V\geq 0, \;\;
\|V\|_{C^2(\bar \Omega\cap \{|x|\le M\})}\le C_M,\;\forall\;M>0,
\\
&\Bigl|\frac{V(x)}{ V_1(x)}-1\Bigr|+\Bigl|\frac{|\nabla V(x)|}{ |\nabla V_1(x)|}-1\Bigr|\le g(x)
\Big\},
\end{split}
\]
where  $g(x)\ge 0$ is a given function satisfying  $g(x)\to 0$ as $|x|\to  +\infty$ and as
$d(x,\partial\Omega)\to 0$.
 It is worth pointing out that for any
$V_n\in \mathcal E$, up to a subsequence, there exists $V\in\mathcal E$,
such that $V_n\to V$ in $C^1(\bar \Omega\cap\{|x|\le M\})$
for any $M>0$.

  For $a_2<k a^*$, we denote
\begin{equation}\label{1-5}
\beta_k^*=ka^*+\sqrt{(ka^*-a_1)(ka^*-a_2)}.
\end{equation}

Without loss of generality, we assume that $a_2\ge a_1$.
The first result on the existence of nontrivial solutions  is as follows.

\begin{theorem}\label{thm1-3}  Suppose that $V_2\in \mathcal{E}$.  If
 one of the following conditions holds:
\begin{enumerate}
\item[(i)] $a_1,a_2\in (0,a^*)$,
$\beta\in (a_2,\beta_1^*)$ and  $k\geq 0$;

\item[(ii)] $a_1,a_2\in (0,a^*)$,  $\beta\in  (\beta_1^*,\beta_2^*)$ and $k> 0$;

\item[(iii)] $a_1\in (0,a^*), a_{2}\in (a^*,2a^*)$,  $\beta\in   (a_2, \beta_2^*)$ and $k> 0$.
\end{enumerate}
 Then \eqref{1-1}-\eqref{1-2} has a nontrivial solution
 provided 
 
 \begin{equation}\label{30-17-8}
 \|V_2-V_1\|_{L^\infty(\Omega)}<\hat c_0,
 \end{equation}
  where $\hat c_0$ is the constant defined in \eqref{1-31-7}.
 
\end{theorem}

 Theorem \ref{thm1-3} provides existence results for  $\beta>a_2= \max\{a_1,a_2\}$. Next we consider the case $0<\beta<a_1$.

\begin{theorem}\label{thm1-4}  Suppose that  $V_2\in\mathcal{E}$,
and $k\ge 0$. 
Assume  that  $0\le \beta<a_1$,
$a_1\in (0,a^*)$ and $a_2\in ((m-1)a^*,ma^*)$,  where $m\in \mathbb{Z}_+$. Then
 \eqref{1-1}-\eqref{1-2} has a nontrivial solution provided
\begin{equation}\label{1-13-8}
-c_2(a_2)<\inf \{V_1-V_2\} \leq \sup \{V_1-V_2\} <c_1(a_1),
\end{equation}
 where $c_1(a_1)$ and $c_2(a_2)$ are positive constants defined in
 \eqref{n6-22} and  \eqref{6-22} respectively.
\end{theorem}

 It is worth pointing out that if $\beta=0$, the two equations in \eqref{1-1} decouple.
 But  \eqref{1-1}-\eqref{1-2} still  has a nontrivial solution. This phenomenon
 cannot always occur if the constraint \eqref{1-2} is replaced by

 \[
 \int_{\Omega} u_1^2=1,\quad \int_{\Omega} u_2^2=1,
 \]
 since for $V_2(x)=|x|^2$, $k=0$ and $a_2\ge a^*$, \eqref{1-13} has no solution.

We remark that the positive constants $\hat c_0$ and $c_1(a_1)$ and $c_2(a_2)$
in Theorems~\ref{thm1-3} and \ref{thm1-4} respectively are defined explicitly.
Hence Theorems~\ref{thm1-3} and \ref{thm1-4} are not perturbation results
of those for the case $V_1=V_2$.

In the rest of this section, we outline the proof of Theorems~\ref{thm1-3} and \ref{thm1-4}.

We first convert the existence of nontrivial solutions into a fixed point problem
for a compact operator.
To be more precise,
for any   $f_1,f_2\in L^2_+(\Omega)$, we consider the following  problem
\begin{equation}\label{1-11}
\left\{\begin{array}{ll}
-\Delta u_1+V_1(x)u_1=f_1+\mu u_1&\hbox{ in }\Omega,\\
-\Delta u_2+V_2(x)u_2=f_2+\mu u_2&\hbox{ in }\Omega,\\
u_1,u_2\geq 0 &\hbox{ in }\Omega,
\\
u_1=u_2=0&\hbox{ on }\partial\Omega,
\\
\int_\Omega (u_1^2+u_2^2)=1,
\end{array}\right.
\end{equation}
where
\begin{equation}\label{1-10}
L^2_+(\Omega)=\left\{ u\in L^2(\Omega): \;  u\geq 0,\;\; u\not\equiv 0\right\}.
\end{equation}
As shown in Lemma \ref{lem3-3},    \eqref{1-11} has  a unique solution $(u_1,u_2,\mu)$. Thus the operator
\[
T: L_+^2(\Omega)\times L_+^2(\Omega) \to \mathcal{H};\;\; T(f_1,f_2)=(u_1,u_2)
\]
is well-defined.
Consequently, the operator

\[
\mathbb{T}_{a_1,a_2,\beta}(u_1,u_2)=T\left(a_1|u_1|^3+\beta |u_1|u_2^2, a_2|u_2|^3+\beta u_1^2|u_2|\right),
\]
is well-defined for $u_1\ne 0$ and $u_2\ne 0$. Thus, if $\mathbb{T}_{a_1,a_2,\beta}$ has
a fixed point $(u_1,u_2)$ with $u_1\ne 0$ and $u_2\ne 0$, then \eqref{1-1}-\eqref{1-2} has a nontrivial solution.

Note that it is difficult to prove that $\mathbb{T}_{a_1,a_2,\beta}$ has a fixed point with $u_1\neq 0$ and $u_2\neq 0$, since $\mathbb{T}_{a_1,a_2,\beta}$ may not be well-defined for all $(u_1, u_2)\ne (0,0)$. See Remark~\ref{re1-4-8}. To overcome
this difficulty, we instead consider the following perturbed operator

\begin{equation}\label{1-6-8}
\mathbb{T}_{a_1,a_2,\beta}^{\varepsilon}(u_1,u_2)=T\left(a_1|u_1|^3+\beta |u_1|u_2^2+\varepsilon |u_2|, a_2|u_2|^3+\beta u_1^2|u_2|+\varepsilon |u_1|\right).
\end{equation}
 For each small $\varepsilon >0$,
 $\mathbb{T}_{a_1,a_2,\beta}^{\varepsilon}$ is well-defined
 for all $(u_1, u_2)\ne (0,0)$. Moreover,    $(u_1,u_2)$ is a fixed point of $\mathbb{T}_{a_1,a_2,\beta}^{\varepsilon}$ if and only if
$(u_1,u_2,\mu)$ is a solution to the following perturbed problem

\begin{equation}\label{1-9}
\left\{
\begin{array}{ll}
-\Delta u_{1}+V_1(x)u_{1}=a_{1}u_{1}^3+\beta u_{1}u_{2}^2+\varepsilon u_2+\mu u_{1}& \hbox{ in }\Omega,\\
-\Delta u_{2}+V_2(x)u_{2}=a_{2}u_{2}^3+\beta u_{2}u_{1}^2+\varepsilon u_1+\mu u_{2}&\hbox{ in }\Omega,\\
  u_{1},u_{2}\geq 0 &\hbox{ in }\Omega,
 \\
 u_1=u_2=0 &\hbox{ on }\partial\Omega,
 \\
 \int_{\Omega}(u_{1}^2+u_{2}^2)=1.
 \end{array}\right.
\end{equation}

Note that any solution $(u_{1,\varepsilon},u_{2,\varepsilon})$ of \eqref{1-9}
is nontrivial.
The general idea is to prove the existence of a fixed point $(u_{1,\varepsilon},u_{2,\varepsilon})$ for
$\mathbb{T}_{a_1,a_2,\beta}^{\varepsilon}$. Then, with
suitable estimates on $(u_{1,\varepsilon},u_{2,\varepsilon})$, we prove
that it converges to a nontrivial solution of
\eqref{1-1}-\eqref{1-2}.

To this end, we will compute the Leray-Schauder degree
\begin{equation}\label{1-12}
d_{a_1,a_2,\beta}^{\varepsilon}:=\text{deg}(I-\mathbb{T}_{a_1,a_2,\beta}^{\varepsilon}, \bar{B}_{R}\setminus B_\delta, 0)
\end{equation}
for some large $R>0$ and small $\delta>0$, where
\[
B_{r}=\Big\{(u_1,u_2)\in\mathcal{H}:\|(u_1,u_2)\|_{\mathcal{H}}< r \Big\}.
\]

In view of the compact embedding $H_i(\Omega)\hookrightarrow L^2(\Omega)$, there exists $\delta>0$ such that for any $(u_1,u_2)\in \mathcal{H}$ with
\[
\int_{\Omega} (u_1^2+u_2^2)=1,
\]
we have
\[
\|(u_1,u_2)\|_{\mathcal{H}}> \delta.
\]
Hence, for the Leray-Schauder degree to be well-defined, we need to identify
  the values of $a_1$, $a_2$ and $\beta$ at which solutions of \eqref{1-9} are bounded
  in $\mathcal{H}$ for all  $\varepsilon\ge 0$ sufficiently small. Or equivalently, we  need to identify
  the values of $a_1, a_2$ and $\beta$ at which solutions of \eqref{1-9} blow up. However, due to the interaction of different types of peak solutions for \eqref{1-9},
  the blow-up parameter values are
  not as simple as in Theorem~A for the single equation.  For example, if $V_i$
  has $k_i$ different non-degenerate critical points $y^{(i)}_j$, then for given
  $b_1>0$ and $b_2>0$ and large $-\mu>0$, the following problem

  \[
  \left\{
\begin{array}{ll}
-\Delta v_{1}+V_1(x)v_{1}=b_{1}v_{1}^3+\beta v_{1}v_{2}^2+\mu v_{1}& \hbox{ in }\Omega,\\
-\Delta v_{2}+V_2(x)v_{2}=b_{2}v_{2}^3+\beta v_{2}v_{1}^2+\mu v_{2}&\hbox{ in }\Omega,\\
  v_{1},v_{2}> 0 &\hbox{ in }\Omega,
 \\
 v_1=v_2=0 &\hbox{ on }\partial\Omega,\end{array}\right.
 \]
  has a solution of the form

  \[
  (v_1(x), v_2(x))= \Bigl( \sum_{j=1}^{k_1}\sqrt{\frac{-\mu}{b_1}}Q\left(\sqrt{-\mu}(x-y^{(1)}_{j})\right)(1+o(1)),
  \sum_{j=1}^{k_2}\sqrt{\frac{-\mu}{b_2}}Q\left(\sqrt{-\mu}(x-y^{(2)}_{j})\right)(1+o(1)) \Bigr).
  \]
  Then,

  \[
  \int_{\Omega} (v_1^2+v^2_2)=\left(\frac{k_1}{b_1}+\frac{k_2}{b_2}\right)a^*+o(1).
  \]
   Let

  \[
  (u_1(x), u_2(x))=\frac{(v_1(x), v_2(x))}{\Bigl(\int_{\Omega} (v_1^2+v^2_2)\Bigr)^{\frac12}}.
  \]
We see that  $(u_1, u_2)$ satisfies  \eqref{1-1}--\eqref{1-2} with

  \[
  a_i=b_i \left(\frac{k_1}{b_1}+\frac{k_2}{b_2}\right)a^*+o(1), \quad  i=1, 2.
  \]
 This shows that blow-up for \eqref{1-1}--\eqref{1-2} can occur for $(a_1, a_2, \beta)$ with any $\beta>0$. This large number of blow-up solutions makes computing
 the degree for all of them very difficult. To simplify the computation,
 we identify a range of parameters $a_1$, $a_2$ and $\beta$ for which
  all blow-up solutions can be classified.

Though by Theorem~\ref{lemB-1} in Appendix~B, it is possible to classify all the values at which solutions of \eqref{1-9} blow up,
 in this paper, we just
prove the following relatively simple result, which is sufficient for us
to prove the existence result.

\begin{theorem}\label{thm1-1}
Let $\delta>0$ be a fixed small constant. Suppose  that $\bar{V}_1,\bar{V}_2,\; \hat{V}_1,\hat{V}_2\in \mathcal{E}$. If one
of the following conditions holds:

\begin{enumerate}
\item[(a)] $a_1,a_2\in (0,a^*-\delta]$,  and either $\beta\in (0, \beta_1^*-\delta]$ or
$\beta\in [\beta_m^*+\delta, \beta_{m+1}^*-\delta]$ for some integer $m>0$;

\item [(b)]  $a_1\in (0, a^*-\delta]$ and $a_{2}\in [a^*+\delta,2a^*-\delta]$,
 $\beta\in  [a_1, \beta_2^*-\delta]$;

\item[(c)] $a_1\in (0,a^*-\delta]$ and $a_{2}\in [ma^*+\delta,(m+1)a^*-\delta]$  for some integer $m\ge 0$,
$\beta\in(0,a_1)$,
\end{enumerate}
 then there exist a positive constant $C_{\delta}$ and a small constant
 $\varepsilon_0>0$ such that for any
 $\varepsilon\in [0,\varepsilon_0)$ and
\[
V_1=t\bar{V}_1+(1-t)\hat{V}_1,\;\; V_2=t\bar{V}_2+(1-t)\hat{V}_2, \;\; 0\leq t\leq 1,
\]
it holds 
\[
\|u_1\|_{L^\infty(\Omega)}+\|u_2\|_{L^\infty(\Omega)}, \|(u_1,u_2)\|_{\mathcal{H}}, |\mu|\leq C_{\delta}.
\]
for any solution $\left(u_1,u_2,\mu\right)\in \mathcal{H}\times \mathbb{R}$ of \eqref{1-9}.

\end{theorem}

 It follows from  Theorem~\ref{thm1-1} that the Leray-Schauder degree
\eqref{1-12} is well-defined for $a_1$, $a_2$ and $\beta$ satisfying the conditions
in Theorem~\ref{thm1-1}. One of our main results
is the following degree formula.

\begin{theorem}\label{thm1-2}
Let  $0<\varepsilon<\varepsilon_0$ and $V_2\in \mathcal{E}$.
\begin{enumerate}

\item[(i)] If $a_1,a_2\in (0,a^*)$ and $\beta \in (0,\beta_1^*)$, then
$d_{a_1,a_2,\beta}^{\varepsilon}=1$.

\item[(ii)] If $a_1,a_2\in (0,a^*)$ and $\beta \in (\beta_1^*,\beta_2^*)$, then
$d_{a_1,a_2,\beta}^{\varepsilon}=k$, where $k$ is the number of holes in $\Omega$.

\item[(iii)] Suppose either $a_1\in (0,a^*)$ and $a_{2}\in (a^*,2a^*)$.
If  $\beta\in   [a_1, \beta_2^*)$, then  $d_{a_1,a_2,\beta}^{\varepsilon}=k$.
\end{enumerate}

\end{theorem}

A consequence of Theorem~\ref{thm1-2} is that for small $\varepsilon>0$, \eqref{1-9}
has  a family of solutions $\{(u_{1,\varepsilon},u_{2,\varepsilon})\}$ with
both components positive,
provided $a_1$, $a_2$ and $\beta$ satisfy the conditions in Theorem~\ref{thm1-2}. By Theorem~\ref{thm1-1}, they are uniformly bounded in $\mathcal{H}$  with respect to $\varepsilon$.
Then, up to a subsequence, we may assume that
$(u_{1,\varepsilon},u_{2,\varepsilon})\to (u_{1},u_{2})$ strongly in $L^2(\mathbb R^2)\times L^2(\mathbb{R}^2)$ as
 $\varepsilon\to 0$. The next crucial step is to provide mechanisms to ensure that 
 $u_1>0$ and $u_2>0$.
 We will prove that if  $\|V_2-V_1\|_{L^\infty(\Omega)}<\hat c_0$, where
 $\hat c_0$ is given in
 \eqref{1-31-7}, then $u_1>0$ and $u_2>0$. Thus, $(u_{1},u_{2})$ is a nontrivial solution
 of \eqref{1-1}--\eqref{1-2}. The condition $\beta>a_2$ in (iii) of Theorem~\ref{thm1-3} is 
 to make $\hat c_0>0$, though  the degree computation in (iii) of Theorem~\ref{thm1-2}
 just requires that $\beta>a_1$.

\smallskip

We now briefly discuss the procedure in  the proof of Theorem \ref{thm1-2},
focusing on the case $a_1,\, a_2\in (0, a^*)$ and $\beta \in (\beta_1^*,\beta_2^*)$.

\bigskip

{\bf Step 1.} By  Theorem~\ref{thm1-1} and the homotopy invariance of degree,  to  compute the degree $d_{a_1,a_2,\beta}$ for any $a_1,a_2\in (0,a^*)$ and $\beta \in (\beta_1^*,\beta_2^*)$,
we may take $V_1=V_2=V$.

{\bf Step 2.} Using the  homotopy invariance of degree, we prove that
 for $\varepsilon>0$ small enugh, if  $a_1,a_2\in (0,a^*)$ and $0<\beta<\beta_1^*$,
 then
\begin{equation}\label{1-14}
d_{a_1,a_2,\beta}^{\varepsilon}=1.
\end{equation}

{\bf Step 3.} The value of $d_{a_1,a_2,\beta}$ changes when $\beta$ crosses $\beta_1^*$, since
\eqref{1-1}--\eqref{1-2} has solutions blowing up as $\beta\to \beta_1^*$. Thus,
it is crucial to find all such blow-up solutions and compute their contributions
to the degree.  To simplify these computations, using the  homotopy invariance of degree,
we may further deform $V$ so that it satisfies the following conditions:

\begin{enumerate}
\item [$(H_1)$] $V$ has finitely many critical points, all of which are non-degenerate.

\item[$(H_2)$] For each critical point $x$ of $V$, $\Delta V(x)\neq 0$.
\end{enumerate}
With these extra assumptions, we can classify all the blow-up solutions
for \eqref{1-1}--\eqref{1-2}  as $\beta\to \beta_1^*$ and compute their degrees.
Combining Step~2 with the contributions to the degree
of all the blow-up solutions  as $\beta\to \beta_1^*$, we obtain
the value of $d_{a_1,a_2,\beta}$ for   $a_1,a_2\in (0,a^*)$ and $\beta \in (\beta_1^*,\beta_2^*)$.

 \begin{remark}
 To compute $d_{a_1,a_2,\beta}$ for   $a_1,a_2\in (0,a^*)$ and $\beta \in (\beta_2^*,\beta_3^*)$,
 we need to find all the blow-up solutions for \eqref{1-1}--\eqref{1-2}  as
 $\beta$ crosses $\beta^*_2$.  This is more complex than the previous case, because
 as $\beta\to \beta^*_2$, \eqref{1-1}--\eqref{1-2} has double-peaked solutions concentrating
 at the same point.

 \end{remark}

\bigskip

We now turn to discuss the proof of Theorem~\ref{thm1-4}.

Since $a_1\in (0,a^*)$, \eqref{1-1}--\eqref{1-2} always has a semi-trivial solution $(u_1, 0)$.
To prove the existence of a nontrivial solution, we will show that the following degree for
all nontrivial solutions is well-defined and nonzero:

\begin{equation}\label{1-15}
\begin{split}
d_{a_1,a_2,\beta}^{non-trivial}:=\text{deg}(I-\mathbb{T}_{a_1,a_2,\beta}, \bar{B}_{R}\setminus (B_\delta \cup W_1\cup W_2), 0),
\end{split}
\end{equation}
where
\begin{equation}\label{1-15b}
W_1=\{  (u_1,u_2)\in  \bar{B}_{R}\setminus B_\delta :\; \|u_2\|_{L^2(\Omega)}<\sigma  \},\;\;\; W_2=\{  (u_1,u_2)\in  \bar{B}_{R}\setminus B_\delta :\; \|u_1\|_{L^2(\Omega)}<\sigma  \},
\end{equation}
and $\sigma>0$ is a suitable small constant. This consists of two steps.

{\bf Step~1}.  We  prove that there exists a constant
$\sigma>0$, such that for any $V_2\in \mathcal E$, if \eqref{1-13-8} holds, then 
 any nontrivial solution $(u_1,u_2;\mu)$ of \eqref{1-1}-\eqref{1-2} satisfies
\begin{equation}\label{1-14b}
\|u_1\|_{L^2(\Omega)}>\sigma,\quad \|u_2\|_{L^2(\Omega)}>\sigma.
\end{equation}
Thus, $d_{a_1,a_2,\beta}^{non-trivial}$ is well-defined. By the homotopy
invariance, we can choose $V_1=V_2$ to calculate $d_{a_1,a_2,\beta}^{non-trivial}$.  As a by-product of \eqref{1-14b},
 the following topological degree for all semi-trivial solutions is also well-defined
 if $V_1=V_2$:
\begin{equation}\label{nn1-15}
\begin{split}
d_{a_1,a_2,\beta}^{semi-trivial}:=\text{deg}(I-\mathbb{T}_{a_1,a_2,\beta}, W_1\cup W_2, 0).
\end{split}
\end{equation}

{\bf Step~2}. Now for $V_1=V_2$, we have

\begin{equation}\label{1-11-8}
\begin{split}
d_{a_1,a_2,\beta}=&d_{a_1,a_2,\beta}^{non-trivial}
+d_{a_1,a_2,\beta}^{semi-trivial}\\
=&d_{a_1,a_2,\beta}^{non-trivial}
+d_{a_1,a_2,\beta}^{1,semi-trivial}+d_{a_1,a_2,\beta}^{2,semi-trivial},
\end{split}
\end{equation}
where

\[
d_{a_1,a_2,\beta}^{i,semi-trivial}=\text{deg}(I-\mathbb{T}_{a_1,a_2,\beta}, W_i, 0).
\]
Using the homotopy invariance of the degree and $a_1\in (0, a^*)$, we can show that
$d_{a_1,a_2,\beta}^{1,semi-trivial}=1$. Though the exact values of 
$d_{a_1,a_2,\beta}$ and $d_{a_1,a_2,\beta}^{2,semi-trivial}$
may not be computed, we use the
Leray-Schauder degree reduction theorem (see for
example Theorem~8.17 in \cite{D}), together with a suitable homotopy, to derive

\[
d_{a_1,a_2,\beta}^{2,semi-trivial}=d_{a_1,a_2,\beta},
\]
and hence $d_{a_1,a_2,\beta}^{non-trivial}=-1$.

Let us point out that it is possible to replace \eqref{30-17-8} and \eqref{1-13-8} by other conditions to a obtain
similar result as in Theorems~\ref{thm1-3} and \ref{thm1-4}. For example, if the first eigenvalue $\lambda_1(V_i)$
for $-\Delta +V_i$ satisfy $\lambda_1(V_1)=\lambda_1(V_2)$, then by
Lemma~\ref{l1-4-8},  $\mathbb{T}_{a_1,a_2,\beta}$ is well defined in $B_R\setminus B_\delta $.
We can compute the
Leray-Schauder degree directly without deforming $V_2$ to $V_1$ and 
prove that $d_{a_1,a_2,\beta}^{non-trivial}=-1$ provided  $\beta\in (0,\beta_0)$ for some $\beta_0>0$.
See the discussion in Section~\ref{8}.

Results on the existence of solutions 
 to the mean field equations 
 by computing the Leray-Schauder degree can be found in \cite{li,l,cl-1,cl}.

Before closing this section, we mention that  results on the Gross-Pitaevskii system with  constraints $\int_{\Omega}u_1^2=b_1$ and $\int_{\Omega}u_2^2=b_2$ can be found in \cite{bs,bjs,bzz,gzz,ntv2,ntv1} and the references therein.  Most of these
papers  study the super-mass-critical cases using variational methods.

 \bigskip
 This paper is organized as follows. In Section \ref{3}, we make some preparations so that
the existence of solutions can be reformulated as a fixed point problem for a compact
operator.
 In Section \ref{2},
 we  use blow-up arguments, together with the classification of positive solutions of the
  nonlinear Schr\"odinger system in Appendix~B, to prove Theorem \ref{thm1-1}.
  We compute  the Morse index of each blow-up solution as $\beta\to\beta_1^*$
  in  Section~\ref{4}, and  the total Leray-Schauder degree for all
  the blow-up solutions as $\beta\to\beta_1^*$  in Section \ref{5}.
   Sections \ref{6} and \ref{100} are devoted to the proofs of Theorem~\ref{thm1-3}
   and Theorem~\ref{thm1-4}, respectively. We also discuss briefly other conditions
   ensuing the existence of nontrivial solutions for \eqref{1-1}-\eqref{1-2} in
   Section~8.

\section{preliminary }\label{3}

This section collects several results, from which we can change  the existence
of solutions for  \eqref{1-1}--\eqref{1-2} to  a fixed point problem
for a compact operator. As some of the proofs are entirely similar to those in \cite{cyy}, we omit them here.
Suppose that   $V_1,\,V_2\in \mathcal{E}$.   Let $\lambda_1(V_i)$ be the first eigenvalue of $-\Delta+V_i(x)$ on $\Omega$ with zero Dirichlet boundary condition, and $e_1(V_i)$ be the unique positive eigenfunction corresponding to $\lambda_1(V_i)$ with $\int_\Omega e_1^2(V_i)=1$.
Without loss of generality, we may assume that
\begin{equation}\label{3-1}
\lambda_1(V_1)\leq \lambda_1(V_2).
\end{equation}

\begin{lemma}[\cite{cyy}, Lemma A.3]\label{le6-3-1}
Let $\lambda\leq \lambda_1(V_i)$ and $w\in\mathcal{H}$ satisfy
\[
-\Delta w+V_i(x)w-\lambda w\geq 0\quad \hbox{in }\Omega.
\]

If $\lambda=\lambda_1(V_i)$, then either $w>0$, $w\equiv 0$, or $w<0$.

If $\lambda<\lambda_1(V_i)$, then either $w>0$ or $w\equiv0$.
\end{lemma}

Let
\[
L^2_+(\Omega)=\left\{ u\in L^2(\Omega): \;  u\geq 0,\;\; u\not\equiv 0\right\}.
\]
For any  $f_1,f_2\in L^2_+(\Omega)$, we consider the following problem
\begin{equation}\label{3-2}
\left\{\begin{array}{l}
-\Delta u_1+V_1(x) u_1=f_1+\mu u_1,\\
-\Delta u_2+V_2(x)u_2=f_2+\mu u_2,\\
u_1,u_2\geq 0,
 \int_\Omega (u_1^2+u_2^2)=1.
\end{array}\right.
\end{equation}
To find a solution of  \eqref{3-2}, it  is natural to consider the following minimization problem:
\[
\lambda:=\inf\limits_{(u_1,u_2)\neq (0,0), \|u_1\|^2_2+\|u_2\|^2_2=1}\int_{\Omega}\left(|\nabla u_1|^2+|\nabla u_2|^2+V_1(x) u_1^2+ V_2(x) u_2^2\right)-2\int_\Omega\left(f_1u_1+f_2u_2\right).
\]
By a standard minimization argument, $\lambda$ is attained.
Then the existence of a solution to \eqref{3-2} follows.

\begin{lemma}\label{lem3-2}  If $f_1,f_2\in L_+^2(\Omega)$, then $\mu <\lambda_1(V_1)$.
\end{lemma}

\begin{proof} Let $f_1,f_2\in L_+^2(\Omega)$. Since  $f_1\not\equiv 0$, from the first equation of \eqref{3-2} we have $u_1\not\equiv 0$.
Let $\tilde{u}_1= u_1/\|u_1\|_2$, $\tilde{f}_1= f_1/\|u_1\|_2$. Then
\[
\begin{cases}
-\Delta \tilde{u}_1+V_1(x) \tilde{u}_1=\tilde{f}_1+\mu \tilde{u}_1\quad \text{in}\;\; \Omega,
\\
0\leq \tilde{u}_1\in H_0^1(\Omega),
\\
\int_\Omega \tilde{u}_1^2=1.
\end{cases}
\]
Then by \cite[Lemma A.5]{cyy}, we have $\mu<\lambda_1(V_1)$.
\end{proof}

\begin{lemma}\label{lem3-3}
If $f_1,f_2\in L_+^2(\Omega)$, then the solution $(u_1,u_2,\mu)$ of \eqref{3-2} is unique.

\end{lemma}

\begin{proof}
Suppose there exists another solution $(\tilde{u}_1,\tilde{u}_2,\tilde{\mu})$ such that
\[
\left\{\begin{array}{l}
-\Delta \tilde{u}_1+V_1(x) \tilde{u}_1=f_1+\tilde{\mu} \tilde{u}_1,\hbox{ in }\Omega,\\
-\Delta \tilde{u}_2+V_2(x)\tilde{u}_2=f_2+\tilde{\mu} \tilde{u}_2,\hbox{ in }\Omega,\\
\tilde{u}_1,\tilde{u}_2\geq 0, \int_\Omega (\tilde{u}_1^2+\tilde{u}_2^2)=1.
\end{array}\right.
\]
Without loss of generality, we may assume $\mu\geq \tilde{\mu}$. Let  $w_i=u_i-\tilde{u}_i$. Then $(w_1,w_2)$ satisfies
\begin{equation}\label{3-3}
\left\{\begin{array}{l}
-\Delta w_1+V_1(x) w_1-\mu w_1=(\mu-\tilde{\mu})\tilde{u}_1\geq 0,\\
-\Delta w_2+V_2(x)w_2-\mu w_2=(\mu-\tilde{\mu})\tilde{u}_2\geq 0.\\
\end{array}\right.
\end{equation}
By Lemma~\ref{lem3-2}, we have $\mu<\lambda_1$.
On the other hand,  from Lemma \ref{le6-3-1}, we get $w_1,w_2\geq 0$. If $(w_1,w_2)\neq (0,0)$, then
\[
1=\int_\Omega (u_1^2+u_2^2)> \int_\Omega (\tilde{u}_1^2+\tilde{u}_2^2)=1,
\]
which is impossible. Thus, $w_1\equiv w_2\equiv 0$. This gives  $u_1=\tilde{u}_1$ and $u_2=\tilde{u}_2$,
and hence  $\mu=\tilde{\mu}$.
\end{proof}

By Lemma~\ref{lem3-3}, we can define an operator $T:L_+^2(\Omega)\times L_+^2(\Omega)\to \mathcal{H}$ by
\[
T(f_1,f_2)=(u_1,u_2),
\]
where $(u_1,u_2)$ is the unique solution of \eqref{3-2}.

\begin{remark}\label{re1-4-8}
Note that if $\lambda_1(V_1)<\lambda_1(V_2)$, $T$ is not well-defined at $(0, f_2)$ for any $f_2\in L_+^2(\Omega)$  .  Indeed, for $f_2= \alpha e_1(V_2)$, $\alpha\in (0, \lambda_1(V_2)-\lambda_1(V_1))$,
\eqref{3-2}  has a solution given  by

\[
(u_1, u_2,\mu)=\Bigl(\sqrt{ 1-\frac{\alpha^2}{(\lambda_1(V_2)-\lambda_1(V_1))^2} }e_1(V_1),  \frac{\alpha}{\lambda_1(V_2)-\lambda_1(V_1)}e_1(V_2),\lambda_1(V_1) \Bigr).
\]
But \eqref{3-2}  also has another solution given  by

\[
(u_1, u_2,\mu)=(0, w, \mu^*),
\]
where $(w, \mu^*)$ is the solution of

\[
\begin{cases}
-\Delta w +V_2(x) w= \alpha e_2 +\mu^* w,\;\; w>0, &\text{in}\; \Omega,\\
\int_{\Omega}w^2=1, w=0,&\text{on}\; \partial\Omega.
\end{cases}
\]
Thus, the solution to \eqref{3-2} is not unique.
\end{remark}

On the other hand, from the proof of Lemma~\ref{lem3-3}, the following lemma holds.

\begin{lemma}\label{l1-4-8}
If $\lambda_1(V_1)=\lambda_1(V_2)$ (in particular, if $V_1=V_2$), then for any nonnegative $f_1,f_2\in L^2(\Omega)$ with $(f_1,f_2)\neq (0,0)$,
the solution $(u_1,u_2,\mu)$ of \eqref{3-2} is unique.

\end{lemma}

 By Lemma~\ref{lem3-3}, for any $a_1,a_2,\beta, \varepsilon>0$ and $(u_1,u_2)\in \mathcal{H}\setminus\{(0,0)\}$, the  operator 
\begin{equation}\label{3-4}
\mathbb{T}_{a_1,a_2,\beta}^\varepsilon (u_1,u_2):=T(a_1|u_1|^3+\beta |u_1|u_2^2+\varepsilon |u_2|,a_2|u_2|^3+\beta |u_2|u_1^2+\varepsilon |u_1|)
\end{equation}
is well-defined. Moreover, it follows from Lemma~\ref{l1-4-8} that
$\mathbb{T}_{a_1,a_2,\beta}^0$ is also well-defined in $\mathcal{H}\setminus\{(0,0)\}$, if $V_1=V_2$.

\begin{proposition}\label{pro3-4}
Suppose either $V_1\neq V_2$ and $a_1,a_2,\beta,\varepsilon>0$, or $V_1=V_2$ and $a_1,a_2,\beta>0$, $\varepsilon\geq 0$. Then for any small $\delta>0$, $\mathbb{T}_{a_1,a_2,\beta}^\varepsilon$ is a compact operator from $\mathcal{H}\setminus B_\delta$ into  itself, where
\[
B_{\delta}=\Big\{(u_1,u_2)\in\mathcal{H}:\|(u_1,u_2)\|_{\mathcal{H}}< \delta \Big\}.
\]
\end{proposition}

\begin{proof} For any $(u_1,u_2)\in \mathcal{H}\setminus B_\delta $, we have
\[
\mathbb{T}_{a_1,a_2,\beta}^\varepsilon (u_1,u_2)\in \mathcal{H}\cap \Big\{\int_\Omega (u_1^2+u_2^2)=1\Big\}.
\]
Since the embedding $H_i(\Omega)\hookrightarrow L^2(\Omega)$ is compact, there exists $\delta_0>0$ such that for any $(u_1,u_2)\in \mathcal{H}$ with
\[
\int_{\Omega} (u_1^2+u_2^2)=1,
\]
we have
\[
\|(u_1,u_2)\|_{\mathcal{H}}> \delta_0.
\]
Thus for small $\delta >0$,  $\mathbb{T}_{a_1,a_2,\beta}^\varepsilon$ maps $\mathcal{H}\setminus B_\delta$ into  itself.

Using a similar argument as in \cite[Proposition A.6]{cyy}, we can prove that $\mathbb{T}_{a_1,a_2,\beta}^\varepsilon$ is compact.
\end{proof}

 From the discussion in this section, we conclude that the existence of a solution
 for \eqref{1-9} is equivalent to the existence of a fixed point for the compact
 operator $\mathbb{T}_{a_1,a_2,\beta}^\varepsilon$.

At the end of  this section, assuming Theorem \ref{thm1-1}
 at the moment, we prove that if $\|V_2-V_1\|_{L^\infty(\Omega)}< \hat c_0$
 for some $\hat c_0>0$, then as $\varepsilon\to 0$,
  any solution $(u_{1\varepsilon}, u_{2\varepsilon})$
  of \eqref{1-9} converges to a 
  nontrivial solution
 to \eqref{1-1}--\eqref{1-2}.

 To determine the constant $\hat c_0>0$, for any $V\in \mathcal E$, we consider
 the following problem
 
\begin{equation}\label{6-6}
\begin{cases}
-\Delta u+V(x)u=a u^3+\mu u,\; u>0,\;\;\;\text{in}\;\;\Omega,
\\
u=0\;\;\; \text{on}\;\;\partial\Omega,
\\
\int_\Omega u^2=1,
\end{cases}
\end{equation}
where $a\in (ma^*, (m+1)a^*)$.

Set

\begin{equation}\label{n6-7}
 c(a, V)=\inf \Big\{ \int_\Omega u^4:\; \text{$u$ is a positive solution of \eqref{6-6}
} \Big\}.
\end{equation}

Define
\begin{equation}\label{6-7}
\hat c(a)=\inf \Big\{c(a, V):\;
\forall\; V\in \mathcal E,\; \|V-V_1\|_{L^\infty(\Omega)}\le c(a_1, V_1)(\beta-a_1) \Big\}.
\end{equation}
By Theorem~A,  if \eqref{6-6} has solution for  $a\in (ma^*, (m+1)a^*)$,
we can prove that  $  c(a, V)$ is achieved by some $u>0$ and hence $ c(a, V)>0$. Otherwise,
we regard $  c(a, V)=+\infty$.  Similarly, if \eqref{6-6} has solution for some $V\in \mathcal E$
with $\|V-V_1\|_{L^\infty(\Omega)}\le c(a, V_1)(\beta-a)$ and $a\in (ma^*, (m+1)a^*)$,
then  $  \hat c(a)$ is achieved by some $V$ and hence $ \hat c(a)>0$. Otherwise,
we regard $  \hat c(a)=+\infty$.

We define $ \hat c_0>0$ as follows.  We assume that $a_1\in (0, a^*)$.
Let

 \begin{equation}\label{1-31-7}
\hat c_0=\min\{ c(a_1, V_1)(\beta-a_1), \hat c(a_2)(\beta-a_2)\}.
 \end{equation}

 \begin{proposition}\label{p1-5-8}
 Assume that $a_1$, $a_2$ and $\beta$ satisfy the conditions in Theorem \ref{thm1-3}.
  If
   $\|V_2-V_1\|_{L^\infty(\Omega)}< \hat c_0$,
   then there exists $\sigma>0$ such that for any sufficiently small $\varepsilon\in (0,\varepsilon_0)$ and any solution $(u_{1\varepsilon}, u_{2\varepsilon})$
  of \eqref{1-9}, it holds
\begin{equation}\label{6-8}
\|u_{1\varepsilon}\|_{L^2(\Omega)}> \sigma,\quad \|u_{2\varepsilon}\|_{L^2(\Omega)}> \sigma.
\end{equation}

\end{proposition}

\begin{proof}
We suppose, for contradiction, that up to a subsequence,
\[
\int_\Omega u_{1\varepsilon}^2\to 1, \;\;\; \int_\Omega u_{2\varepsilon}^2\to 0, \;\;\; \varepsilon \to 0.
\]
By Theorem \ref{thm1-1}, $u_{i\varepsilon}$ is bounded in $H_i(\Omega)$ and $L^\infty(\Omega)$ for $i=1,2$,  and $\mu_\varepsilon$ is bounded.  So up to a subsequence, we can assume that $u_{1\varepsilon}\rightharpoonup u_1$ in $H_1(\Omega)$,  and $u_{1\varepsilon}\to u_1$ in $L^p(\Omega)$  for $p\geq 2$, and $\mu_\varepsilon \to \mu$.  Then by the first equation of \eqref{1-9},  $u_1$ is a positive solution of
\begin{equation}\label{6-9}
\begin{cases}
-\Delta u_1+V_1(x)u_1=a_1u^3_1+\mu u_1\;\;\;\text{in}\;\;\Omega,
\\
u_1=0\;\;\; \text{on}\;\;\partial\Omega,
\\
\int_\Omega u_1^2=1,
\end{cases}
\end{equation}
which gives
\begin{equation}\label{6-9a}
\mu=\int_\Omega \big(|\nabla u_1|^2+V_1(x) u_1^2- a_1u_1^4\big).
\end{equation}

We now consider the asymptotic behavior of $u_{2\varepsilon}$.

First, we assume that $\varepsilon/ \|u_{2\varepsilon}\|_{L^2(\Omega)}\to +\infty$. Setting $v_\varepsilon= u_{2\varepsilon}/\varepsilon$, then we have  $ \|v_{\varepsilon}\|_{L^2(\Omega)}\to 0$.
From the second equation of \eqref{1-9},  we get
\begin{equation}\label{6-10}
-\Delta v_{\varepsilon}+V_2(x)v_{\varepsilon}=a_2 \varepsilon^2 v_{\varepsilon}^3+\beta v_{\varepsilon}u_{1\varepsilon}^2+u_{1\varepsilon}+\mu_\varepsilon v_{\varepsilon}\;\;\;\text{in}\;\;\Omega,\quad v_\varepsilon=0 \;\;\;\text{on}\;\;\partial\Omega.
\end{equation}
Then using the Gagliardo-Nirenberg-Sobolev inequality, we obtain
\[
\begin{split}
\int_\Omega (|\nabla v_\varepsilon|^2+V_2(x)v_{\varepsilon}^2)=&a_2 \varepsilon^2 \int_\Omega v_{\varepsilon}^4+\beta \int_\Omega v_{\varepsilon}^2u_{1\varepsilon}^2+\int_\Omega v_{\varepsilon} u_{1\varepsilon}+\mu_\varepsilon\int_\Omega v_{\varepsilon}^2
\\
\leq & o(1) \int_\Omega |\nabla v_\varepsilon|^2 +o(1),
\end{split}
\]
which implies
\[
\int_\Omega (|\nabla v_\varepsilon|^2+V_2(x)v_{\varepsilon}^2)\to 0.
\]
Then multiplying \eqref{6-10} by $\varphi\in C_0^\infty(\Omega)$, integrating by parts, and letting $\varepsilon \to 0$, we get
 \[
 \int_\Omega u_1 \varphi =0.
 \]
Since $\varphi$ is arbitrary, we have $u_1=0$, which
contradicts to  $\int_\Omega u_1^2=1$. Thus the case $\varepsilon/ \|u_{2\varepsilon}\|_{L^2(\Omega)}\to +\infty$ cannot occur.

Now we assume that $\varepsilon/ \|u_{2\varepsilon}\|_{L^2(\Omega)}\to \sigma_0\geq 0$. Setting $v_{\varepsilon}=u_{2\varepsilon}/ \|u_{2\varepsilon}\|_{L^2(\Omega)}$,  we have
$ \|v_{\varepsilon}\|_{L^2(\Omega)}=1$. From the second equation of \eqref{1-9},  we get
\begin{equation}\label{6-11}
-\Delta v_{\varepsilon}+V_2(x)v_{\varepsilon}=a_2 \|u_{2\varepsilon}\|_{L^2(\Omega)}^2 v_{\varepsilon}^3+\beta v_{\varepsilon}u_{1\varepsilon}^2+\frac{\varepsilon}{\|u_{2\varepsilon}\|_{L^2(\Omega)}}u_{1\varepsilon}+\mu_\varepsilon v_{\varepsilon}\;\;\;\text{in}\;\;\Omega,\quad v_\varepsilon=0 \;\;\;\text{on}\;\;\partial\Omega.
\end{equation}
Then we have
\[
\begin{split}
\int_\Omega (|\nabla v_\varepsilon|^2+V_2(x)v_{\varepsilon}^2)=&a_2 \|u_{2\varepsilon}\|_{L^2(\Omega)}^2 \int_\Omega v_{\varepsilon}^4+\beta \int_\Omega v_{\varepsilon}^2u_{1\varepsilon}^2+\frac{\varepsilon}{\|u_{2\varepsilon}\|_{L^2(\Omega)}}\int_\Omega v_{\varepsilon} u_{1\varepsilon}+\mu_\varepsilon\int_\Omega v_{\varepsilon}^2
\\
\leq & o(1) \int_\Omega |\nabla v_\varepsilon|^2 +C,
\end{split}
\]
which shows that $v_{\varepsilon}$ is bounded in $H_2(\Omega)$.
So up to a subsequence, we can assume that $v_{\varepsilon}\rightharpoonup v$ in $H_2(\Omega)$,  and $v_{\varepsilon}\to v$ in $L^p(\Omega)$  for $p\geq 2$. Moreover, $v$ is a positive solution of
\begin{equation}\label{6-12}
\begin{cases}
-\Delta v+V_2(x)v=\beta v u_1^2+\sigma_0 u_1+\mu v\;\;\;\text{in}\;\;\Omega,
\\
v=0\;\;\; \text{on}\;\;\Omega,
\\
\int_\Omega v^2=1.
\end{cases}
\end{equation}
Since $v$ is a positive solution of \eqref{6-12}, it is a positive minimizer of the following problem
\[
\inf_{\|w\|_{L^2(\Omega)}=1}\int_\Omega \big( |\nabla w|^2+V_2(x)w^2-\beta u_1^2 w^2-2\sigma_0 u_1 w  \big).
\]
 Then by \eqref{6-9a} and \eqref{6-12}, we have
\[
\begin{split}
\mu-\sigma_0 \int_\Omega u_1 v=&\int_\Omega \big( |\nabla v|^2+V_2(x)v^2-\beta u_1^2 v^2-2\sigma_0 u_1 v  \big)\\
=&\inf_{\|w\|_{L^2(\Omega)}=1}\int_\Omega \big( |\nabla w|^2+V_2(x)w^2-\beta u_1^2 w^2-2\sigma_0 u_1 w  \big)
\\
\leq& \int_\Omega \big( |\nabla u_1|^2+V_2(x)u_1^2-\beta u_1^4-2\sigma_0 u_1^2 \big)
\\
=&\mu -2\sigma_0 +\int_\Omega (V_2-V_1)u_1^2+(a_1-\beta)\int_\Omega u_1^4,
\end{split}
\]
which yields
\[
\begin{split}
(\beta-a_1)\int_\Omega u_1^4\leq& \int_\Omega (V_2-V_1)u_1^2+\sigma_0 \int_\Omega u_1 v-2\sigma_0
\\
\leq& \|V_2-V_1\|_{L^\infty(\Omega)} \|u_1\|_{L^2(\Omega)}^2+\sigma_0 \|u_1\|_{L^2(\Omega)}\|v\|_{L^2(\Omega)}-2\sigma_0
\\
=&\|V_2-V_1\|_{L^\infty(\Omega)}-\sigma_0
\leq \|V_2-V_1\|_{L^\infty(\Omega)}.
\end{split}
\]
By \eqref{6-7}, we have $\int_\Omega u_1^4\ge c(a_1, V_1)$. So we  obtain that
\begin{equation}\label{6-14}
\|V_2-V_1\|_{L^\infty(\Omega)}\geq c(a_1,V_1)(\beta -a_1)\geq\hat c_0.
\end{equation}
Therefore, if $\|V_2-V_1\|_{L^\infty(\Omega)}<\hat c_0$, the case $\varepsilon/ \|u_{2\varepsilon}\|_{L^2(\Omega)}\to \sigma_1\geq 0$  cannot occur either.
Hence, we have proved that $\|u_{2\varepsilon}\|_{L^2(\Omega)}> \sigma$.
We can also prove $\|u_{1\varepsilon}\|_{L^2(\Omega)}> \sigma$ in a similar way.
\end{proof}


\section{Profile of blow-up solutions}\label{2}

In this section, we prove Theorem~\ref{thm1-1}  using the uniqueness of
positive solution to \eqref{B-1}.
Let  $\{(u_{1n},u_{2n},\mu_n)\}$ be a blow-up sequence satisfying
\begin{equation}\label{2-1}
\left\{
\begin{array}{ll}
-\Delta u_{1n}+V_{1n}(x)u_{1n}=a_{1n}u_{1n}^3+\beta_n u_{1n}u_{2n}^2+\varepsilon_nu_{2n}+\mu_n u_{1n}& \hbox{ in }\Omega,\\
-\Delta u_{2n}+V_{2n}(x)u_{2n}=a_{2n}u_{2n}^3+\beta_n u_{2n}u_{1n}^2+\varepsilon_nu_{1n}+\mu_n u_{2n}&\hbox{ in }\Omega,\\
u_{1n},u_{2n}\geq 0,  \hbox{ and }\int_{\Omega}u_{1n}^2+u_{2n}^2=1,
\end{array}\right.
\end{equation}
where  $V_{in}=t_n\bar{V}_i+(1-t_n)\hat{V}_i$, $t_n\in [0,1]$, $\varepsilon_n\in [0,\varepsilon_0)$ and as $n\to+\infty$,
\begin{equation}\label{2-2}
\begin{split}
a_{in}\to a_i, \;\;\; \beta_n\to \beta,\;\;\; \max u_{1n}+\max u_{2n}\to +\infty.
\end{split}
\end{equation}

\subsection{The case $a_i\in (0,a^*)$ and  $\beta\in (0, +\infty)$.} {$\\$}

First, we point out that Lemma~\ref{lemA-1} implies that blow up
does not occur if $a_i\in (0,a^*)$ and $\beta\in (0, \beta_1^*)$.
In this subsection, assuming $a_i\in (0,a^*)$ and $\beta\in (0, +\infty)$,
we  prove that blow-up occurs only at $\beta= \beta_m^*$,
and after scaling, $\{(u_{1n},u_{2n})\}$ converges locally to the
nontrivial solution of  \eqref{B-1}.

Let $z_{in}\in \Omega$ be such that $$u_{in}(z_{in})=M_{in}:=\max\limits_{x\in\Omega}u_{in}(x), \quad\rho_n=\frac{M_{1n}}{M_{2n}}.
$$
We have the following result.

\begin{theorem}\label{th2-1}
There exist $\rho_0>0,c>0$, and $m$ points $x_{1n},\cdots, x_{mn}\in \Omega$ with $|x_{jn}|\leq C$ such that, up to a subsequence,
\begin{equation}\label{ad2-3}
\begin{split}
&\rho_n\to \rho_0,\;\;\; -\mu_n M_{1n}^{-2}\to c,\;\;\;
\sqrt{-\mu_n}d(x_{in},\partial \Omega)\to+\infty,\;\;\; \sqrt{-\mu_n} |x_{in}-x_{jn}|\to+\infty \;\;\text{for}\;\; i\neq j,
\\
&\left\|
\frac{u_{1n}}{\sqrt{-\mu_n}}-\sqrt{\frac{\beta-a_2}{\beta^2-a_1a_2}}
\sum\limits_{j=1}^{m}Q\left(\sqrt{-\mu_n}(\cdot-x_{jn})\right)\right\|_{L^\infty(\Omega)}\to 0,\\
&\left\|\frac{u_{2n}}{\sqrt{-\mu_n}}-\sqrt{\frac{\beta-a_1}{\beta^2-a_1a_2}}\sum\limits_{j=1}^{m}
Q\left(\sqrt{-\mu_n}(\cdot-x_{jn})\right)\right\|_{L^\infty(\Omega)}\to 0.
\end{split}
\end{equation}
as $n\to +\infty$.
Moreover, we have $\beta= \beta_m^*$, where
\begin{equation}\label{2-3}
 \beta_m^*:=ma^*+\sqrt{(ma^*-a_1)(ma^*-a_2)}.
\end{equation}
\end{theorem}

\begin{proof}We proceed  in several steps.

{\bf Step 1.} We  have  $\rho_n\to \rho_0>0$.

Suppose, for contradiction, that   $\rho_n\to 0$ or $\rho_n\to+ \infty$.
We first consider the case $\rho_n\to +\infty$.
Set $M_n:=M_{1n}$, $x_{1n}=z_{1n}$ and
\[
\left\{\begin{array}{l}
\tilde{u}_{1n}=M_n^{-1}u_{1n}(M_n^{-1}y+x_{1n}),\\
\tilde{u}_{2n}=M_n^{-1}u_{2n}(M_n^{-1}y+x_{1n}),
\end{array}\right.
\]
then $(\tilde{u}_{1n},\tilde{u}_{2n})$ satisfies
\begin{equation}\label{2-4}
\left\{\begin{array}{l}
-\Delta\tilde{u}_{1n}+M_n^{-2}\Big(V_{1n}\left(M_n^{-1}y+x_{1n})-\mu_n\right)\Big)\tilde{u}_{1n}=a_{1n}\tilde{u}_{1n}^3+\beta_n\tilde{u}_{1n}\tilde{u}_{2n}^2+M_n^{-2}\varepsilon_n \tilde{u}_{2n},\\
-\Delta\tilde{u}_{2n}+M_n^{-2}\Big(V_{2n}\left(M_n^{-1}y+x_{1n})-\mu_n\right)\Big)\tilde{u}_{2n}=a_{2n}\tilde{u}_{2n}^3+\beta_n\tilde{u}_{2n}\tilde{u}_{1n}^2+M_n^{-2}\varepsilon_n \tilde{u}_{1n}.
\end{array}\right.
\end{equation}
Since $\tilde{u}_{1n}(0)=1$  is the maximum value of $\tilde{u}_{1n}$, we have $-\Delta\tilde{u}_{1n}(0)\geq 0$. Taking $y=0$ in the first equation of \eqref{2-4},  we get
\[
 M_n^{-2}(V_{1n}(x_{1n})-\mu_n)\leq a_{1n}+\beta_n\rho_n^{-2}+M_n^{-2}\varepsilon_n \rho_n^{-1}\leq C.
\]
In addition, since the right-hand side of \eqref{2-4} is nonnegative, we 
deduce  that $\mu_n\leq \min\{\lambda_{1n},\lambda_{2n}\}\leq C$, where $\lambda_{in}$ is the first eigenvalue of the operator
$
-\Delta +V_{in}(x).
$
 Thus there exists   $c_1\geq 0$ such that
\[
M_n^{-2}(V_{1n}(x_{1n})-\mu_n)\to c_1.
\]
As $\rho_n\to 0$,  up to a subsequence, we have that either  $\tilde{u}_{1n}\to u_1$ in $C_{loc}^2(\mathbb{R}^2)$ with
\[
-\Delta u_1+c_1u_1=a_1u_1^3 \hbox{ in }\mathbb{R}^2, ~~~~0\leq u_1\leq u_1(0)=1,
\]
or $\tilde{u}_{1n}\to u_1$ in $C_{loc}^2(\Pi)$ with
\[
-\Delta u_1+c_1u_1=a_1u_1^3 \hbox{ in }\Pi, \;\; u=0\hbox{ on }\partial \Pi,\;~~~~0\leq u_1\leq u_1(0)=1,
\]
where $\Pi$ is a half-plane. By Liouville's theorem, the latter case cannot occur.
Moreover, in the first case, we have $c_1>0$.
 Thus,   $u_1(y)=a_2^{-\frac{1}{2}}c_1^{\frac{1}{2}}Q(c_1^{\frac{1}{2}}y)$. Then  for any $R>0$ large,
\[
M_n^{-1}\left\|u_{1n}-M_n a_1^{-\frac{1}{2}}c_1^{\frac{1}{2}}Q\left(c_1^{\frac{1}{2}}M_n(\cdot-x_{1n})\right)
\right\|_{L^\infty\left(B_{RM_n^{-1}}(x_{1n})\right)}\to 0.
\]
As a result, since $0<a_1<a^*$, for sufficiently large $R$ and $n$,
\begin{equation}\label{2-5}
\begin{aligned}
1=&\int_{\Omega}u_{1n}^2+u_{2n}^2\geq \int_{B_{RM_n^{-1}}(x_{1n})}u_{1n}^2
=\int_{B_{RM_n^{-1}}(x_{1n})}\left(M_n a_1^{-\frac{1}{2}}c_1^{\frac{1}{2}}Q\left(c_1^{\frac{1}{2}}M_n(\cdot-x_{1n})\right)\right)^2+o_n(1)\\
=&\frac{1}{a_1}\int_{B_{c_1^{\frac{1}{2}}R}(0)}Q^2+o_n(1)
=\frac{a^*}{a_1}+o_R(1)+o_n(1)>1,
\end{aligned}
\end{equation}
a contradiction. Thus, $\rho_n\to +\infty$ cannot occur.

Similarly, we can prove that $\rho_n\to 0$  cannot occur by swapping   $u_{2n}$ and $u_{1n}$. So,  there exists $\rho_0>0$ such that  $\rho_n\to \rho_0>0$.

\textbf{Step 2. } Denote  $x_{1n}:=z_{1n}$ and
\[
\left\{\begin{array}{l}
\tilde{u}_{1n}=M_n^{-1}u_{1n}(M_n^{-1}y+x_{1n}),\\
\tilde{u}_{2n}=M_n^{-1}u_{2n}(M_n^{-1}y+x_{1n}).
\end{array}\right.
\]
Then $(\tilde{u}_{1n},\tilde{u}_{2n})$ satisfies
\begin{equation}\label{2-6}
\left\{\begin{array}{l}
-\Delta\tilde{u}_{1n}+M_n^{-2}\Big(V_{1n}\left(M_n^{-1}y+x_{1n})-\mu_n\right)\Big)\tilde{u}_{1n}=a_{1n}\tilde{u}_{1n}^3+\beta_n\tilde{u}_{1n}\tilde{u}_{2n}^2+M_n^{-2}\varepsilon_n \tilde{u}_{2n},\\
-\Delta\tilde{u}_{2n}+M_n^{-2}\Big(V_{2n}\left(M_n^{-1}y+x_{1n})-\mu_n\right)\Big)\tilde{u}_{2n}=a_{2n}\tilde{u}_{2n}^3+\beta_n\tilde{u}_{2n}\tilde{u}_{1n}^2+M_n^{-2}\varepsilon_n \tilde{u}_{1n}.
\end{array}\right.
\end{equation}
As in Step 1, we have that
\begin{equation}\label{10-4-8}
M_n^{-2}(V_{1n}(x_{1n})-\mu_n)\to c_1\geq 0.
\end{equation}

If for any $R>0$, $M_n^{-1}\|u_{2n}\|_{L^\infty\left(B_{RM_n^{-1}}(x_{1n})\right)}\to 0$, then  we can use a similar argument as in Step 1 to obtain that
\[
M_n^{-1}\|u_{1n}-M_na_2^{-\frac{1}{2}}c_1^{\frac{1}{2}}Q(c^{\frac{1}{2}}_1M_n(\cdot-x_{1n}))\|_{L^\infty\left(B_{RM_n^{-1}}(x_{1n})\right)}\to 0,
\]
which gives  a contradiction as in \eqref{2-5}.  So, there exist $R_0,\sigma_0>0$ such that for large $n$,
\[
M_n^{-1}\|u_{2n}\|_{L^\infty\left(B_{R_0M_n^{-1}}(x_{1n})\right)}\geq \sigma_0>0.
\]
 We
claim that it also holds

\[
M_n^{-2}(V_{2n}(x_{1n})-\mu_n)\to c_1\geq 0.
\]
Indeed, if $|x_{1n}|\le C$, then this follows directly from \eqref{10-4-8}, while if
$|x_{1n}|\to +\infty$,  it can be derived  from \eqref{10-4-8} and $\bigl|
\frac{V_1(x)}{V_2(x)}-1\bigr|\le g(x)\to 0$ as $|x|\to +\infty$.
So we see that
there exist  $c_1\geq 0$ and positive functions $u_1^{(1)},u_2^{(1)}\in C^2_{loc}(\Pi)$ such that
\begin{equation}\label{2-7}
\left\{\begin{array}{ll}
-\Delta u_1^{(1)}+c_1 u_1^{(1)} =a_1 (u_1^{(1)})^3+\beta u_1^{(1)} (u_2^{(1)})^2& \hbox{ in } \Pi,\\
-\Delta u_2^{(1)}+c_1 u_2^{(1)}=a_2(u_2^{(1)})^3+\beta u_2^{(1)} (u_1^{(1)})^2& \hbox{ in } \Pi,
\end{array}\right.
\end{equation}
where $\Pi=\mathbb{R}^2_+$ or $\mathbb{R}^2$, and $(u_1^{(1)},u_2^{(1)})=(0,0)$ on $\partial \Pi$ if $\Pi=\mathbb{R}_+^2$.
By the Pohozaev identity in \cite{el},  \eqref{2-7} has no solution for $\Pi=\mathbb{R}^2_+$. So,  $\Pi=\mathbb{R}^2$.  Moreover, it follows from the Liouville theorem in \cite{ttvw} that $c_1>0$.
Consequently,  $M_n d(x_{1n},\partial \Omega)\to +\infty$ and
for any $R>0$,
\begin{equation}\label{2-8}
M_n^{-1}\left\|u_{1n}-M_n u_1^{(1)}\left(M_n(\cdot-x_{1n})\right)\right\|_{L^\infty\left(B_{RM_n^{-1}}(x_{1n})\right)}
+M_n^{-1}\left\|u_{2n}-M_n u_2^{(1)}\left(M_n(\cdot-x_{1n})\right)\right\|_{L^\infty\left(B_{RM_n^{-1}}(x_{1n})\right)}\to 0.
\end{equation}

\textbf{Step 3.} If
\[
M_n^{-1}\left\|u_{1n}-M_n u_1^{(1)}\left(M_n(\cdot-x_{1n})\right)\right\|_{L^\infty\left(\Omega\right)}
+M_n^{-1}\left\|u_{2n}-M_n u_2^{(1)}\left(M_n(\cdot-x_{1n})\right)\right\|_{L^\infty\left(\Omega\right)}\to 0,
\]
as $n\to +\infty$, then the proof is completed.

We now consider the case
\[
M_n^{-1}\left\|u_{1n}-M_n u_1^{(1)}\left(M_n(\cdot-x_{1n})\right)\right\|_{L^\infty\left(\Omega\right)}\geq  c_0,\;\; \text{or}\;\; M_n^{-1}\left\|u_{2n}-M_n u_2^{(1)}\left(M_n(\cdot-x_{1n})\right)\right\|_{L^\infty\left(\Omega\right)}\geq c_0,
\]
where $c_0>0$ is a constant.  We first suppose that   $M_n^{-1}\left\|u_{1n}-M_n u_1^{(1)}\left(M_n(\cdot-x_{1n})\right)\right\|_{L^\infty\left(\Omega\right)}\geq c_0$.  Then there exists a local maximum point $x_{2n}$ of $u_{1n}$ such that
\[
M_n|x_{1n}-x_{2n}|\to +\infty,\;\;\; M_n^{-1}u_{1n}(x_{2n})\geq c_0/2>0.
\]
Denote
\[
\left\{\begin{array}{l}
\tilde{u}_{1n}=M_n^{-1}u_{1n}(M_n^{-1}y+x_{2n}),\\
\tilde{u}_{2n}=M_n^{-1}u_{2n}(M_n^{-1}y+x_{2n}).
\end{array}\right.
\]
Then $(\tilde{u}_{1n},\tilde{u}_{2n})$ satisfies
\[
\left\{\begin{array}{l}
-\Delta\tilde{u}_{1n}+M_n^{-2}\Big(V_{1n}\left(M_n^{-1}y+x_{2n})-\mu_n\right)\Big)\tilde{u}_{1n}=a_{1n}\tilde{u}_{1n}^3+\beta_n\tilde{u}_{1n}\tilde{u}_{2n}^2+M_n^{-2}\varepsilon_n \tilde{u}_{2n},\\
-\Delta\tilde{u}_{2n}+M_n^{-2}\Big(V_{2n}\left(M_n^{-1}y+x_{2n})-\mu_n\right)\Big)\tilde{u}_{2n}=a_{2n}\tilde{u}_{2n}^3+\beta_n\tilde{u}_{2n}\tilde{u}_{1n}^2+M_n^{-2}\varepsilon_n \tilde{u}_{1n}.
\end{array}\right.
\]
As in Steps~1 and 2, we can prove that there exists $c_2>0$ such that
\[
M_n^{-2}\big(V_{in}\left(x_{2n})-\mu_n\right)\big)\to c_2,
\]
  $M_n d(x_{2n},\partial \Omega)\to +\infty$, and
for any $R>0$,
\begin{equation}\label{2-9}
M_n^{-1}\left\|u_{1n}-M_n u_1^{(2)}\left(M_n(\cdot-x_{2n})\right)\right\|_{L^\infty\left(B_{RM_n^{-1}}(x_{2n})\right)}
+M_n^{-1}\left\|u_{2n}-M_n u_2^{(2)}\left(M_n(\cdot-x_{2n})\right)\right\|_{L^\infty\left(B_{RM_n^{-1}}(x_{2n})\right)}\to 0,
\end{equation}
where $u_1^{(2)}, u_2^{(2)}\in C^2_{loc}(\mathbb{R}^2)$ satisfy
\begin{equation}\label{2-10}
\left\{\begin{array}{ll}
-\Delta u_1^{(2)}+c_2 u_1^{(2)}=a_1 (u_1^{(2)})^3+\beta u_1^{(2)} (u_2^{(2)})^2& \hbox{ in } \mathbb{R}^2,\\
-\Delta u_2^{(2)}+c_2 u_2^{(2)}=a_2 (u_2^{(2)})^3+\beta u_2^{(2)} (u_1^{(2)})^2& \hbox{ in } \mathbb{R}^2.
\end{array}\right.
\end{equation}

In the case $M_n^{-1}\left\|u_{2n}-M_n u_2^{(1)}\left(M_n(\cdot-x_{1n})\right)\right\|_{L^\infty\left(\Omega\right)}\geq  c_0$, we similarly obtain \eqref{2-9}-\eqref{2-10}.

\textbf{Step 4. }Repeating the procedure in Step 3, we obtain $m$ points $x_{1n},\cdots, x_{mn}$ such that
\[
M_n|x_{in}-x_{jn}|\to +\infty\;\; \text{for}\;\; i\neq j, \quad M_nd(x_{in},\partial\Omega)\to +\infty,
\]
\[
M_n^{-2}\left(V_{in}(x_{jn})-\mu_n\right)\to c_j>0,
\]
and for any $R>0$,
\begin{equation}\label{2-11}
M_n^{-1}\left\|u_{1n}-M_n u_1^{(j)}\left(M_n(\cdot-x_{jn})\right)\right\|_{L^\infty\left(B_{RM_n^{-1}}(x_{jn})\right)}
+M_n^{-1}\left\|u_{2n}-M_n u_2^{(j)}\left(M_n(\cdot-x_{jn})\right)\right\|_{L^\infty\left(B_{RM_n^{-1}}(x_{jn})\right)}\to 0,
\end{equation}
where  $\left(u_1^{(j)},u_2^{(j)}\right)$ satisfies
 \begin{equation}\label{2-12}
\left\{\begin{array}{ll}
-\Delta u_1^{(j)}+c_j u_1^{(j)}=a_1(u_1^{(j)})^3+\beta u_1^{(j)} (u_2^{(j)})^2& \hbox{ in } \mathbb{R}^2,\\
-\Delta u_2^{(j)}+c_j u_2^{(j)}=a_2 (u_2^{(j)})^3+\beta u_2^{(j)} (u_1^{(j)})^2& \hbox{ in } \mathbb{R}^2.
\end{array}\right.
\end{equation}

By Lemma~\ref{lemA-1} in Appendix~A, we have
 \begin{equation}\label{2-13}
 \int_{\mathbb{R}^2}((u_1^{(j)})^2+ (u_2^{(j)})^2)\geq \frac{a^*}{\max\{a_1,a_2,\beta\}}.
\end{equation}
Then by  \eqref{2-11} and \eqref{2-13}, we obtain
\begin{equation}\label{2-14b}
\begin{aligned}
1=&\int_{\Omega}(u_{1n}^2+u_{2n}^2)\geq\sum\limits_{j=1}^{m}\int_{B_{RM_n^{-1}}(x_{jn})}(u_{1n}^2+u_{2n}^2)\\
=&\sum\limits_{j=1}^{m}\int_{B_{RM_n^{-1}}(x_{jn})} \Big[\left(M_n u_1^{(j)}(M_n(x-x_{jn}))\right)^2+    \left(M_n u_2^{(j)}(M_n(x-x_{jn}))\right)^2\Big] +o_n(1)\\
=&\sum\limits_{j=1}^m\int_{B_{R}(0)} \left((u_1^{(j)})^2+(u_2^{(j)})^2\right)+o_n(1)=\sum\limits_{j=1}^m\int_{\mathbb{R}^2} \left((u_1^{(j)})^2+(u_2^{(j)})^2\right)+o_n(1)+o_R(1)\\
\geq &\frac{ma^*}{\max\{a_1,a_2,\beta\}}+o_R(1)+o_n(1).
\end{aligned}
\end{equation}
Thus,   this procedure terminates after finitely many steps, say $m$.  And
\begin{equation}\label{2-14}
\begin{split}
M_n^{-1}\left\|u_{1n}-\sum\limits_{j=1}^{m}M_n u_1^{(j)}\left(M_n(\cdot-x_{jn})\right)\right\|_{L^\infty(\Omega)}
+M_n^{-1}\left\|u_{2n}-\sum\limits_{j=1}^{m}M_n u_2^{(j)}\left(M_n(\cdot-x_{jn})\right)\right\|_{L^\infty(\Omega)}\to 0.
\end{split}
\end{equation}
Moreover, from  $a_1,a_2\in (0,a^*)$  and \eqref{2-14b}, we obtain
\begin{equation}\label{2-14a}
\beta\geq a^*.
\end{equation}

\textbf{Step 5.}  We claim that  $|x_{in}|\leq C$ for $i=1,\dots, m$.
The proof is similar to that of  Lemma 2.3 in \cite{cyy}. For completeness, we  sketch it here.

Suppose, for contradiction, that $|x_{1n}|\to +\infty$. For the sequence $\{x_{in}\}$, we have  $|x_{in}-x_{1n}|\to 0$, $i=1,\cdots, i_0$,  $|x_{in}-x_{1n}|\geq \sigma_{i}>0$, $i=i_0+1,\cdots, m$.
Let $\theta>0$ be such that $2\theta<\min_{i\ge i_0+1}\{\sigma_{i}\}$.
Then by Step 4, we have $a_{1n}u_{1n}^2+\beta_n u_{2n}^2\leq \sigma_0^2M_n^2$ and $a_{2n}u_{2n}^2+\beta_n u_{1n}^2\leq \sigma_0^2M_n^2$ \hbox{ in } $B_{\theta}(x_{1n})\setminus B_{\frac{\theta}{4}}(x_{1n})$.
Moreover  by $|x_{1n}|\to +\infty$,  we obtain $V_{jn}(x_{1n})=(1+o(1))V_1(x_{1n})$, which
together with   $(V_3)$  yields
\[
M_n^{-2}\left(V_{jn}(x)-\mu_n\right)=M_n^{-2}\left((1+o_\theta(1))V_1(x_{1n})-\mu_n\right)\geq 3\sigma_0^2\hbox{ in }B_{\theta}(x_{1n}),
\]
for some $\sigma_0>0$. Thus,
\begin{align*}
-\Delta u_{1n}+2\sigma_0^2M_n^2 u_{1n}\leq \varepsilon_n u_{2n} \hbox{ and }
-\Delta u_{2n}+2\sigma_0^2 M_n^2u_{2n}\leq  \varepsilon_n u_{1n}\hbox{ in }B_{\theta}(x_{1n})\setminus B_{\frac{\theta}{4}}(x_{1n}).
\end{align*}
So we get
\begin{equation}\label{2-15}
-\Delta (u_{1n}+u_{2n})+ (2\sigma_0^2 M_n^2-\varepsilon_n\big)(u_{1n}+u_{2n})\leq 0  \hbox{ in }B_{\theta}(x_{1n})\setminus B_{\frac{\theta}{4}}(x_{1n}).
\end{equation}
Since $2\sigma_0^2 M_n^2-\varepsilon_n> \sigma_0^2 M_n^2$ and $u_{1n},u_{2n}\geq 0$,  it follows from the maximum principle and $L^p$ estimates that
\begin{equation}\label{2-16}
0\leq u_{1n}(x),u_{2n}(x)\leq CM_n\left(e^{\frac{1}{8}\sigma_0\theta M_n-\frac{1}{2}\sigma_0M_n|x-x_{1n}|}+e^{\frac{1}{2}\sigma_0M_n|x-x_{1n}|-\frac{1}{2}\sigma_0\theta M_n}\right)\hbox{ in }B_{\theta}(x_{1n})\setminus  B_{\frac{\theta}{4}}(x_{1n}),
\end{equation}
and
\begin{equation}\label{2-17}
|\nabla u_{1n}|, |\nabla u_{2n}|\leq CM_n^2e^{-\frac{1}{16}\sigma_0\theta M_n}\hbox{ on }\partial B_{\frac{\theta}{2}}(x_{1n}).
\end{equation}

Now  we set $\vec{n}=\frac{x_{1n}}{|x_{1n}|}$ and recall the following Pohozaev identity
\begin{equation}\label{2-18}
\begin{aligned}
&-\int_{\partial B_{\frac{\theta}{2}}(x_{1n})}\frac{\partial u_{1n}}{\partial \vec{n}}\frac{\partial u_{1n}}{\partial \vec{\nu}}+\frac{\partial u_{2n}}{\partial \vec{n}}\frac{\partial u_{2n}}{\partial \vec{\nu}}+\frac{1}{2}\int_{\partial B_{\frac{\theta}{2}}(x_{1n})}(|\nabla u_{1n}|^2+|\nabla u_{2n}|^2)\vec{n}\cdot\vec{\nu}\\
&-\int_{\partial B_{\frac{\theta}{2}}(x_{1n})}\frac{\beta_n}{2}u_{1n}^2u_{2n}^2\vec{n}\cdot\vec{\nu}-\int_{\partial B_{\frac{\theta}{2}}(x_{1n})}\left(\frac{a_{1n}}{4}u_{1n}^4+\frac{a_{2n}}{4}u_{2n}^4\right)\vec{n}\cdot\vec{\nu}\\
&+\frac{1}{2}\int_{\partial B_{\frac{\theta}{2}}(x_{1n})}\left(V_1(x)-\mu_n\right)u_{1n}^2\vec{n}\cdot\vec{\nu}
+\frac{1}{2}\int_{\partial B_{\frac{\theta}{2}}(x_{1n})}\left(V_2(x)-\mu_n\right)u_{2n}^2\vec{n}\cdot\vec{\nu}
\\&-\varepsilon_n \int_{\partial B_{\frac{\theta}{2}}(x_{1n})} \vec{n}\cdot\vec{\nu} (u_{1n} u_{2n})\\
=&\frac{1}{2}\int_{B_{\frac{\theta}{2}}(x_{1n})}\Bigl(\frac{\partial V_1}{\partial \vec{n}}u_{1n}^2+\frac{\partial V_2}{\partial \vec{n}}u_{2n}^2\Bigr).
\end{aligned}
\end{equation}
Combining \eqref{2-16}-\eqref{2-17}, we deduce
\begin{equation}\label{2-19}
\begin{aligned}
\text{LHS of \eqref{2-18}}= O\left( M_n^4e^{-\frac{1}{8}\sigma_0\theta M_n} \right).
\end{aligned}
\end{equation}
On the other hand,  
\[
\begin{aligned}
\frac{1}{2}\int_{B_{\frac{\theta}{2}}(x_{1n})}\Bigl(\frac{\partial V_1}{\partial \vec{n}}u_{1n}^2+\frac{\partial V_2}{\partial \vec{n}}u_{2n}^2\Bigr)
=\frac{1}{2}(1+o_\theta(1))(\nabla V_1(x_{1n})\cdot \vec{n})\int_{B_{\frac{\theta}{2}}(x_{1n})}(u_{1n}^2+u_{2n}^2).
\end{aligned}
\]
Noting that
\[
\int_{B_{\frac{\theta}{2}}(x_{1n})}(u_{1n}^2+u_{2n}^2)\geq a_0>0,
\]
we obtain for large $n$ and small $\theta$,
\begin{equation}\label{2-20}
\frac{1}{2}\int_{B_{\frac{\theta}{2}}(x_{1n})}\Bigl(\frac{\partial V_1}{\partial \vec{n}}u_{1n}^2+\frac{\partial V_2}{\partial \vec{n}}u_{2n}^2\Bigr)\geq \frac{a_0\sigma_0}{4},
\end{equation}
which, combining  \eqref{2-18} and  \eqref{2-19}, leads to
\[
 CM_n^4e^{-\frac{1}{8}\sigma_0\theta M_n}\geq \frac{a_0\sigma_0}{4}.
\]
This is a contradiction.
Thus  $|x_{in}|\leq C$ for $i=1,\dots, m$.

Since $\{x_{jn}\}$ is bounded, we see that
\[
c:=\lim\limits_{n\to\infty}M_n^{-2}(V_{in}(x_{jn})-\mu_n)=\lim\limits_{n\to\infty}-M_n^{-2}\mu_n,
\]
namely, $c_1=c_2=\cdots=c_m=c>0$. Moreover, for $j=1,\cdots,m$,  $(u_1^{(j)},u_2^{(j)})$ is a positive solution of
\begin{equation}\label{2-21}
\left\{\begin{array}{ll}
-\Delta u_1+cu_1=a_1u_1^3+\beta u_1 u_2^2& \hbox{ in } \mathbb{R}^2,\\
-\Delta u_2+cu_2=a_2u_2^3+\beta u_2u_1^2& \hbox{ in } \mathbb{R}^2.
\end{array}\right.
\end{equation}
Since $a_1,a_2\in (0,a^*)$, by \eqref{2-14a} we have
\begin{equation}\label{2-22}
 \beta>\max\{a_1,a_2\}.
 \end{equation}
 Then, it follows from Theorem~\ref{lemB-1} that
\[
(u_1,u_2)=\left(\sqrt{\frac{\beta-a_2}{\beta^2-a_1a_2}c}Q(\sqrt{c}x),\sqrt{\frac{\beta-a_1}{\beta^2-a_1a_2}c}Q(\sqrt{c}x)\right).
\]
Hence from \eqref{2-14} and  $\sqrt{-\mu_n}=\sqrt{c}M_n(1+o(1))$,
  \eqref{ad2-3} follows.

\textbf{Step 6.}
Using the comparison principle, we can prove that
\[
0\leq u_{1n}(x)+u_{2n}(x)\leq CM_n\sum_{j=1}^me^{-\sigma_0M_n|x-x_{jn}|} \hbox{ for }x\in\Omega\setminus \cup_{j=1}^m B_{RM_n^{-1}}(x_{jn}).
\]
As a result,
\[
\begin{aligned}
1=&\int_{\Omega}u_{1n}^2+u_{2n}^2=\sum_{j=1}^m \int_{B_{RM_n^{-1}}(x_{jn})}(u_{1n}^2+u_{2n}^2) +o_R(1)\\
=&\sum_{j=1}^m \int_{B_{RM_n^{-1}}(x_{jn})}\left(M_n\sqrt{\frac{\beta-a_2}{\beta^2-a_1a_2}c}Q(\sqrt{c}M_n(\cdot-x_{jn}))\right)^2
\\
&+\sum_{j=1}^m \int_{B_{RM_n^{-1}}(x_{jn})}\left(M_n\sqrt{\frac{\beta-a_1}{\beta^2-a_1a_2}c}Q(\sqrt{c}M_n(\cdot-x_{jn}))\right)^2
+o_n(1)+o_R(1)
\\
=&\frac{ma^*\big(2\beta-(a_1+a_2)\big)}{\beta^2-a_1a_2}+o_R(1)+o_n(1).
\end{aligned}
\]
Letting $n\to +\infty$ and then $R\to +\infty$, we obtain
\begin{equation}\label{1-5-8}
\frac{ma^*\big(2\beta-(a_1+a_2)\big)}{\beta^2-a_1a_2}=1\;\;\;\text{or}\;\; \beta^2-2ma^*\beta+ma^*(a_1+a_2)-a_1a_2=0.
\end{equation}
Solving \eqref{1-5-8} yields
\[
\beta=ma^*\pm \sqrt{(ma^*-a_1)(ma^*-a_2)}.
\]
By \eqref{2-22}, we have
\[
\beta=\beta_m^*:=ma^*+ \sqrt{(ma^*-a_1)(ma^*-a_2)}.
\]
Hence, we complete the proof.
\end{proof}

\smallskip
\subsection{The case $a_1 \in (0,a^*)$, $a_{2}\in [a^*,2a^*)$ and  $\beta\in [a_1, \beta_2^*)$. } {$\\$}\label{sec2-2}

In this subsection, we consider the case $a_1 \in (0,a^*)$, $a_{2}\in [a^*,2a^*)$,
 and  $\beta\in [a_1, \beta_2^*)$.   We prove that blow-up occurs only at $a_2= a^*$,
and after scaling, $\{(u_{1n},u_{2n})\}$ converges locally to the
semi-trivial solution of  \eqref{B-1}.

Let $z_{in}\in \Omega$ be such that $$
u_{in}(z_{in})=M_{in}:=\max\limits_{x\in\Omega}u_{in}(x), \quad\rho_n=\frac{M_{1n}}{M_{2n}}.
$$
We have the following result.

\begin{theorem}\label{th2-2}
There exist $c>0$ and   $x_{n}\in \Omega$ with $|x_{n}|\leq C$ such that, up to a subsequence,
\begin{equation}\label{2-24}
\begin{split}
& -\mu_n M_n^{-2}\to c,\;\;\;
\sqrt{-\mu_n}d(x_{n},\partial \Omega)\to+\infty,
\\
&\left\|\frac{u_{2n}}{\sqrt{-\mu_n}}-a_2^{-\frac{1}{2}}Q\left(\sqrt{-\mu_n }(\cdot-x_{n})\right)\right\|_{L^\infty(\Omega)}+\frac{1}{\sqrt{-\mu_n}}\left\|u_{1n}\right\|_{L^\infty(\Omega)}\to 0,
\end{split}
\end{equation}
as $n\to \infty$. Moreover, we have $a_2= a^*$.
\end{theorem}

To prove Theorem~\ref{th2-2}, we need the following technical lemma.

\begin{lemma}\label{lem2-3}
Suppose that $a_1\in (0,a^*)$, $a_2\in (a^*,2a^*)$.  Suppose further that $m\geq 1$.
Then   the following equation for   $\beta\in [a_2,+\infty)$
\begin{equation}\label{2-25}
\frac{a^*}{a_2}+\frac{ma^* \big(2\beta -(a_1+a_2)\big)}{\beta^2-a_1a_2}=1
\end{equation}
has a unique solution $\beta=\beta_m^+$, where
\[
\beta_m^+=c_m+\sqrt{(c_m-a_1)(c_m-a_2)}, \;\;\; c_m= \frac{m a_2a^*}{a_2-a^*}.
\]
Moreover, we have $\beta_m^+\geq \beta_2^*$.
\end{lemma}

\begin{proof} Equation \eqref{2-25} is equivalent to
\[
\beta^2-2c_m\beta +(a_1+a_2)c_m-a_1a_2=0,\;\;\text{or}\;\; (\beta-c_m)^2=(c_m-a_1)(c_m-a_2).
\]
Since $a_1\in (0,a^*)$ and $a_2\in (a^*,2a^*)$, we have
\[
c_m= \frac{m a_2a^*}{a_2-a^*}\geq  \frac{a_2 }{a_2-a^*}a^*> \frac{ 2a^*}{2a^*-a^*}a^*=2a^*>a_2>a_1.
\]
Solving \eqref{2-25} yields
\[
c_m\pm \sqrt{(c_m-a_1)(c_m-a_2)}.
\]
On the other hand, since $a_1<c_m-\sqrt{(c_m-a_1)(c_m-a_2)}< a_2$, we have
\[
\beta=\beta_m^+:=c_m+\sqrt{(c_m-a_1)(c_m-a_2)}
\]
if $\beta\in [a_2,+\infty)$. Since $c_m\geq 2a^*$, we have
\[
\beta_m^+=c_m+\sqrt{(c_m-a_1)(c_m-a_2)}\geq 2a^*+\sqrt{(2a^*-a_1)(2a^*-a_2)}=\beta_2^*.
\]
\end{proof}

We now  prove Theorem~\ref{th2-2}.
Up to a subsequence, we may assume that
\begin{equation}\label{2-26}
\rho_n\to 0,\;\;\;\text{or}\;\;\; \rho_n\to \rho_0>0,\;\;\;\text{or}\;\;\; \rho_n\to +\infty.
\end{equation}

\begin{lemma} \label{lem2-4}
The case $\rho_n\to+\infty$ cannot occur.
\end{lemma}
\begin{proof} The proof is given in  Step 1 of the proof of Theorem~\ref{th2-1}.
\end{proof}

\begin{lemma}\label{lem2-5}
The case $\rho_n\to \rho_0>0$ cannot occur.
\end{lemma}

\begin{proof} Suppose that $\rho_n\to \rho_0>0$.  We shall follow
the proof of Theorem~\ref{th2-1} closely. The only difference is that
the arguments there  apply to $u_1$, but not to $u_2$, because
here we assume that $a_2\ge a^*$.

\textbf{Step 1. } Denote  $x_{1n}=z_{1n}$, $M_n=M_{1n}$. Then,
in view of  $a_1\in (0,a^*)$,  similar to \eqref{2-8}, we can prove that  $M_n^{-1}d(x_{1n},\partial \Omega)\to +\infty$, and for any $R>0$,
\[
M_n^{-1}\left\|u_{1n}-M_n u_1^{(1)}\left(M_n(\cdot-x_{1n})\right)\right\|_{L^\infty\left(B_{RM_n^{-1}}(x_{1n})\right)}
+M_n^{-1}\left\|u_{2n}-M_n u_1^{(2)}\left(M_n(\cdot-x_{1n})\right)\right\|_{L^\infty\left(B_{RM_n^{-1}}(x_{1n})\right)}\to 0,
\]
where $(u_1^{(1)},u_1^{(2)})$ is a positive solution of
\[
\left\{\begin{array}{ll}
-\Delta u_1^{(1)}+c_1 u_1^{(1)}=a_1(u_1^{(1)})^3+\beta u_1^{(1)} (u_1^{(2)})^2& \hbox{ in } \mathbb{R}^2,\\
-\Delta u_1^{(2)}+c_1 u_1^{(2)}=a_2(u_1^{(2)})^3+\beta u_1^{(2)} (u_1^{(1)})^2& \hbox{ in } \mathbb{R}^2,
\end{array}\right.
\]
for some $c_1>0$. Note that for $\beta\in [\min \{a_1,a_2\}, \max\{a_1,a_2\}]$, the above equation does not have any positive solution. So we have
\begin{equation}\label{2-28}
\beta>\max\{a_1,a_2\}.
\end{equation}

\textbf{Step 2.} If $M_n^{-1}\left\|u_{1n}-M_n u_1\left(M_n(\cdot-x_{1n})\right)\right\|_{L^\infty(\Omega)}\geq \sigma_0>0$, then we can find a local maximum point $x_{2n}$ of $u_{1n}$,
 such that $M_n^{-1}u_{1n}(x_{2n})>\sigma_0/2>0$ and $M_n|x_{1n}-x_{2n}|\to +\infty$.  Similar to Step 1, there exists $c_2>0$ such that
 for any $R>0$,
\[
M_n^{-1}\left\|u_{1n}-M_n u_1^{(2)}\left(M_n(\cdot-x_{2n})\right)\right\|_{L^\infty\left(B_{RM_n^{-1}}(x_{2n})\right)}
+M_n^{-1}\left\|u_{2n}-M_nu_2^{(2)}\left(M_n(\cdot-x_{2n})\right)\right\|_{L^\infty\left(B_{RM_n^{-1}}(x_{2n})\right)}\to 0,
\]
where $(u_1^{(2)}, u_2^{(2)})$ is a positive solution of
\[
\left\{\begin{array}{ll}
-\Delta u_1^{(2)}+c_2 u_1^{(2)}=a_1(u_1^{(2)})^3+\beta u_1^{(2)} (u_2^{(2)})^2& \hbox{ in } \mathbb{R}^2,\\
-\Delta u_2^{(2)}+c_2 u_2^{(2)}=a_2(u_2^{(2)})^3+\beta u_2^{(2)}(u_1^{(2)})^2& \hbox{ in } \mathbb{R}^2.
\end{array}\right.
\]
Moreover, we have $M_n^{-1}d(x_{2n},\partial \Omega)\to +\infty$.

Repeating  this procedure, we find that there are $m\geq 1$ and  points $x_{1n},\cdots,x_{mn}$ such that
\[
M_n|x_{in}-x_{jn}|\to +\infty\;\; \text{for}\;\; i\neq j, \quad M_nd(x_{in},\partial\Omega)\to +\infty,
\]
\[
M_n^{-2}\left(V_{1n}(x_{jn})-\mu_n\right)\to c_j>0,
\]
and for any $R>0$,
\[
M_n^{-1}\left\|u_{1n}-M_n u_1^{(j)}\left(M_n(\cdot-x_{jn})\right)\right\|_{L^\infty\left(B_{RM_n^{-1}}(x_{jn})\right)}
+M_n^{-1}\left\|u_{2n}-M_n u_2^{(j)}\left(M_n(\cdot-x_{jn})\right)\right\|_{L^\infty\left(B_{RM_n^{-1}}(x_{jn})\right)}\to 0,
\]
where  $(u_1^{(j)}, u_2^{(j)})$ satisfies
\begin{equation}\label{2-29}
\left\{\begin{array}{ll}
-\Delta u_1^{(j)}+c_j u_1^{(j)}=a_1(u_1^{(j)})^3+\beta u_1^{(j)} (u_2^{(j)})^2& \hbox{ in } \mathbb{R}^2,\\
-\Delta u_2^{(j)}+c_j u_2^{(j)}=a_2(u_2^{(j)})^3+\beta u_2^{(j)} (u_1^{(j)})^2& \hbox{ in } \mathbb{R}^2.
\end{array}\right.
\end{equation}
for some $c_j>0$. Similar to Step 4 in the proof of Theorem \ref{th2-1},
this procedure terminates after finitely many steps, say $m$. And for any $R>0$,
\begin{equation}\label{2-30}
\begin{split}
&M_n^{-1}\left\|u_{1n}-\sum\limits_{j=1}^{m}M_n
u_1^{(j)}\left(M_n(\cdot-x_{jn})\right)\right\|_{L^\infty(\Omega)}\to 0,
\\
&M_n^{-1}\left\|u_{2n}-M_nu_2^{(j)}\left(M_n(\cdot-x_{jn})\right)\right\|_{L^\infty\left(B_{RM_n^{-1}}(x_{jn})\right)}\to 0,\;\;\; j=1,\cdots,m.
\end{split}
\end{equation}
It is worth pointing out that we cannot claim that the second
relation in \eqref{2-30} holds in $L^\infty(\Omega)$ as in \eqref{2-14}, because in view
of $a_2\ge a^*$, the argument cannot apply to $u_{2n}$.

\textbf{Step 3.}
Suppose that
\[
M_n^{-1}\left\|u_{2n}-\sum\limits_{j=1}^{m}M_n u_2^{(j)}\left(M_n(\cdot-x_{jn})\right)\right\|_{L^\infty(\Omega)}\to 0.
\]
Then we are in the same situation as in \eqref{2-14}. As a result,
 Steps~5 and 6 in the proof of Theorem \ref{th2-1} yield
\begin{equation}\label{2-31}
 (\beta-ma^*)^2=(ma^*-a_1)(ma^*-a_2).
\end{equation}
See \eqref{1-5-8}.

Note that $a_1\in (0,a^*)$ and $a_2\in [a^*,2a^*)$. Then for $m=1$,  \eqref{2-31} has a solution if and only if $a_2=a^*$, in which case $\beta=a^*$.  This contradicts \eqref{2-28}. Thus, $m\geq 2$.   As a result,
\[
\beta=ma^*+ \sqrt{(ma^*-a_1)(ma^*-a_2)}\geq 2a^*+ \sqrt{(2a^*-a_1)(2a^*-a_2)} =\beta_2^*,
\]
which contradicts $\beta<\beta_2^*$.  Hence,  we have proved that
\begin{equation}\label{2-32}
M_n^{-1}\left\|u_{2n}-\sum\limits_{j=1}^{m}M_n
u_2^{(j)}\left(M_n(\cdot-x_{jn})\right)\right\|_{L^\infty(\Omega)}\geq \sigma_0>0.
\end{equation}

\textbf{Step 4.}
From  \eqref{2-32},  there exists a local maximum point $y_{1n}$ of $u_{2n}$ such that
\[
M_n|y_{1n}-x_{in}|\to +\infty,~~i=1,2,\dots,m,~~~~\hbox{ and }~~~~ M_n^{-1}u_{2n}(y_{1n})\geq \sigma_0/2>0.
\]
Denote
\[
\left\{\begin{array}{l}
\tilde{u}_{1n}=M_n^{-1}u_{1n}(M_n^{-1}y+y_{1n}),\\
\tilde{u}_{2n}=M_n^{-1}u_{2n}(M_n^{-1}y+y_{1n}).
\end{array}\right.
\]
Then $(\tilde{u}_{1n},\tilde{u}_{2n})$ satisfies
\begin{equation}\label{2-27}
\left\{\begin{array}{l}
-\Delta\tilde{u}_{1n}+M_n^{-2}\Big(V_{1n}\left(M_n^{-1}y+y_{1n})-\mu_n\right)\Big)\tilde{u}_{1n}=a_{1n}\tilde{u}_{1n}^3+\beta_n\tilde{u}_{1n}\tilde{u}_{2n}^2+M_n^{-2}\varepsilon_n \tilde{u}_{2n},\\
-\Delta\tilde{u}_{2n}+M_n^{-2}\Big(V_{2n}\left(M_n^{-1}y+y_{1n})-\mu_n\right)\Big)\tilde{u}_{2n}=a_{2n}\tilde{u}_{2n}^3+\beta_n\tilde{u}_{2n}\tilde{u}_{1n}^2+M_n^{-2}\varepsilon_n \tilde{u}_{1n}.
\end{array}\right.
\end{equation}
As in Step 1 in the proof of Theorem \ref{th2-1}, it can be verified that
\[
M_n^{-2}(V_{2n}(y_{1n})-\mu_n)\to b_1\geq 0.
\]
Moreover, it follows from \eqref{2-30} that for any $R>0$,  $M_n^{-1}\|u_{1n}\|_{L^\infty\left(B_{RM_n^{-1}}(y_{1n})\right)}\to 0$. Thus,
by the Liouville theorem, we can prove
\[
b_1>0,\hbox{ and } M_nd(y_{1n},\partial \Omega)\to +\infty,
\]
\[
M_n^{-1}\|u_{2n}-M_na_2^{-\frac{1}{2}}b_1^{\frac{1}{2}}Q(b^{\frac{1}{2}}_1M_n(\cdot-y_{1n}))\|_{L^\infty\left(B_{RM_n^{-1}}(y_{1n})\right)}\to 0.
\]

\textbf{Step 5. } Repeat the procedure in Step 4, we obtain  $k$ points $y_{1n},\cdots, y_{kn}$ such that
\[
M_n|y_{in}-y_{jn}|\to +\infty, \hbox{ for }i\neq j,\hspace{8pt} M_nd(y_{in},\partial\Omega)\to +\infty,
\]
\[
M_n^{-2}\left(V_{2n}(y_{in})-\mu_n\right)\to b_i>0,\hspace{8pt}M_n|x_{jn}-y_{in}|\to +\infty, \hbox{ for all }i,j,
\]
and
\[
M_n^{-1}\|u_{1n}\|_{L^\infty\left(B_{RM_n^{-1}}(y_{in})\right)}+M_n^{-1}\left\|u_{2n}-M_na_2^{-\frac{1}{2}}b_i^{\frac{1}{2}}Q\left(b^{\frac{1}{2}}_iM_n(\cdot-y_{in})\right)\right\|_{L^\infty\left(B_{RM_n^{-1}}(y_{in})\right)}\to 0.
\]
Then combining with \eqref{2-13}, we have
\begin{equation}\label{2-34}
\begin{aligned}
1=&\int_{\Omega}(u_{1n}^2+u_{2n}^2)\geq\sum\limits_{j=1}^{m}\int_{B_{RM_n^{-1}}(x_{jn})}(u_{1n}^2+u_{2n}^2)+\sum\limits_{i=1}^{k}\int_{B_{RM_n^{-1}}(y_{in})}u_{2n}^2\\
=&\sum\limits_{j=1}^{m}\int_{B_{RM_n^{-1}}(x_{jn})} \Big[\left(M_n u_1^{(j)}(M_n(x-x_{jn}))\right)^2+    \left(M_n u_2^{(j)}(M_n(x-x_{jn}))\right)^2\Big]\\
&+\sum\limits_{i=1}^k\int_{B_{RM_n^{-1}}(y_{in})}\left( M_na_2^{-\frac{1}{2}}b_i^{\frac{1}{2}}Q\left(b^{\frac{1}{2}}_iM_n(\cdot-y_{in})\right)\right)^2 +o_n(1)\\
=&\sum\limits_{j=1}^m\int_{B_{R}(0)} \left((u_1^{(j)})^2+(u_2^{(j)})^2\right)+\sum\limits_{i=1}^k\frac{1}{a_2}\int_{B_{b_i^{1/2}R}(0)}Q^2+o_n(1)\\
=&\sum\limits_{j=1}^m\int_{\mathbb{R}^2}\left((u_1^{(j)})^2+(u_2^{(j)})^2\right) +\frac{ka^*}{a_2}+o_n(1)+o_R(1)\\
\geq &\frac{ma^*}{\max\{a_1,a_2,\beta\}}+\frac{ka^*}{a_2}+o_R(1)+o_n(1).
\end{aligned}
\end{equation}
Thus,   this procedure terminates after finitely many  $k$ steps. Moreover,  since $a^*\leq a_2<2a^*$,
we have $k=1$. Thus as $n\to +\infty$
\begin{equation}\label{2-35}
\begin{split}
&M_n^{-1}\left\|u_{1n}-\sum\limits_{j=1}^{m}M_n u_1^{(j)}\left(M_n(\cdot-x_{jn})\right)\right\|_{L^\infty(\Omega)}\to 0,
\\
&M_n^{-1}\left\|u_{2n}-\sum\limits_{j=1}^{m}M_n u_2^{(j)}\left(M_n(\cdot-x_{jn})\right)-M_na_2^{-\frac{1}{2}}b_1^{\frac{1}{2}}Q\left(b^{\frac{1}{2}}_1M_n(\cdot-y_{1n})\right)\right\|_{L^\infty(\Omega)}\to 0.
\end{split}
\end{equation}

\textbf{Step 6.} Now  as in Step 5 of Theorem \ref{th2-1}, we can prove
that $|x_{jn}|$, $|y_{1n}|\leq C$, for $i=1,\cdots, m$. Thus
\[
c:=\lim\limits_{n\to\infty}M_n^{-2}(V_{in}(x_{jn})-\mu_n)
=\lim\limits_{n\to\infty}M_n^{-2}(V_{in}(y_{1n})-\mu_n)=\lim\limits_{n\to\infty}(-M_n^{-2}\mu_n).
\]
That is $b_1=c_1=c_2=\cdots=c_m=c>0$.
Moreover, we have  $\beta>\max\{a_1,a_2\}=a_2$, and for $j=1,\cdots,m$,
\[
(u_1^{(j)},u_2^{(j)})=\left(\sqrt{\frac{\beta-a_2}{\beta^2-a_1a_2}c}Q(\sqrt{c}x),\sqrt{\frac{\beta-a_1}{\beta^2-a_1a_2}c}Q(\sqrt{c}x)\right).
\]
Furthermore, \hbox{ for }$x\notin\left(\cup_{j=1}^m B_{RM_n^{-1}}(x_{jn})\cup B_{RM_n^{-1}}(y_{1n})\right)$, we can use the comparison principle to prove that
\[
0\leq u_{1n}(x)+u_{2n}(x)\leq  CM_ne^{-\sigma_0M_n|x-y_{1n}|}+CM_n\sum_{j=1}^me^{-\sigma_0M_n|x-x_{jn}|}.
\]
As a result
\[
\begin{aligned}
1=&\int_{\Omega}(u_{1n}^2+u_{2n}^2)=\sum\limits_{j=1}^{m}\int_{B_{RM_n^{-1}}(x_{jn})}(u_{1n}^2+u_{2n}^2)+\int_{B_{RM_n^{-1}}(y_{1n})}u_{2n}^2+o_R(1)\\
=&\sum_{j=1}^m \int_{B_{RM_n^{-1}}(x_{jn})}\left(M_n\sqrt{\frac{\beta-a_2}{\beta^2-a_1a_2}c}Q(\sqrt{c}M_n(\cdot-x_{jn}))\right)^2\\
&+\sum_{j=1}^m \int_{B_{RM_n^{-1}}(x_{jn})}\left(M_n\sqrt{\frac{\beta-a_1}{\beta^2-a_1a_2}c}Q(\sqrt{c}M_n(\cdot-x_{jn}))\right)^2\\
&+\int_{B_{RM_n^{-1}}(y_{1n})}\left( M_na_2^{-\frac{1}{2}}c^{\frac{1}{2}}Q\left(c^{\frac{1}{2}}M_n(\cdot-y_{1n})\right)\right)^2 +o_R(1)+o_n(1)\\
=&\frac{ma^*\big(2\beta-(a_1+a_2)\big)}{\beta^2-a_1a_2}+\frac{a^*}{a_2}+o_R(1)+o_n(1).
\end{aligned}
\]
Since $\beta_2^*>\beta>a_2$ and $m\geq 1$, by Lemma \ref{lem2-3}, we obtain a contradiction.
\end{proof}

\begin{proof}[Proof of Theorem~\ref{th2-2}]

Combining  Lemma \ref{lem2-4} and \ref{lem2-5}, we have $\rho_n\to  0$.  Thus, similar to Step 1 in the proof of Theorem~\ref{th2-1} again, there exists $c>0$ such that for any $R>0$,
\[
M_{n}^{-1}\left\|u_{2n}-M_{n} a_2^{-\frac{1}{2}}c^{\frac{1}{2}}Q\left(c^{\frac{1}{2}}M_{n}(\cdot-x_{n})\right)\right\|_{L^\infty\left(B_{RM_n^{-1}}(x_{n})\right)}\to 0,
\]
where $x_n=z_{2n}$,
\[
M_n^{-2}(V_{2n}(x_{n})-\mu_n)\to c>0,\hspace{8pt}M_nd(x_n,\partial\Omega)\to +\infty.
\]
We claim that
\begin{equation}\label{2-36}
M_{n}^{-1}\left\|u_{2n}-M_{n} a_2^{-\frac{1}{2}}c^{\frac{1}{2}}Q\left(c^{\frac{1}{2}}M_{n}(\cdot-x_n)\right)\right\|_{L^\infty(\Omega)}\to 0.
\end{equation}
If not, then we can prove  that there exists another local maximum point $y_{n}$ and a positive constant $\tilde{c}$ such that $M_n|x_n-y_n|\to +\infty$ and for any $R>0$
\[
M_{n}^{-1}\left\|u_{2n}-M_{n} a_2^{-\frac{1}{2}}\tilde{c}^{\frac{1}{2}}Q\left(\tilde{c}^{\frac{1}{2}}M_{n}(\cdot-y_{n})\right)\right\|_{L^\infty\left(B_{RM_n^{-1}}(y_{n})\right)}\to 0.
\]
As a result,
\[
\begin{split}
1=\int_\Omega (u_{1n}^2+u_{2n}^2)\geq \int_\Omega u_{2n}^2\geq \int_{B_{RM_n^{-1}}(x_{n})}u_{2n}^2+\int_{B_{RM_n^{-1}}(y_{n})}u_{2n}^2=\frac{2a^*}{a_2}+o_n(1)+o_R(1)>1
\end{split}
\]
 which   is impossible since $a_2\in [a^*,2a^*)$. Thus \eqref{2-36} holds. Moreover, as in Steps~ 5 and 6  of Theorem \ref{th2-1}, we can show $|x_{n}|\leq C$ and
\[
0\leq u_{1n}(x)+u_{2n}(x)\leq  CM_ne^{-\sigma_0M_n|x-x_{n}|},\hbox{ for }x\in \Omega\setminus B_{RM_n^{-1}}(x_n).
\]
 Thus
\[
c:=\lim\limits_{n\to\infty}M_n^{-2}(V_{2n}(x_n)-\mu_n)=\lim\limits_{n\to\infty}M_n^{-2}(-\mu_n),
\]
and
\[
\begin{split}
1=\int_\Omega (u_{1n}^2+u_{2n}^2)=\int_{B_{RM_n^{-1}}(x_{n})}(u_{1n}^2+u_{2n}^2)+o_R(1)=\frac{a^*}{a_2}+o_n(1)+o_R(1).
\end{split}
\]
We complete the proof.
\end{proof}

\smallskip
\subsection{The case $a_1 \in (0,a^*)$, $a_{2}\in (0,+\infty)$ and  $\beta\in (0, a_1)$} {$\\$}

In this subsection, we consider the case $a_1 \in (0,a^*)$, $a_{2}\in (0,+\infty)$,
 and  $\beta\in (0, a_1)$. 
  We will prove that blow-up occurs only at $a_2= ka^*$,
and after scaling, $\{(u_{1n},u_{2n})\}$ converges locally to the
semi-trivial solution of  \eqref{B-1}.

We have the following result.
\begin{theorem}\label{adth2-2}
There exist $c>0$ and   $x_{1n}, x_{2n},\dots,x_{kn}\in \Omega$ with $|x_{in}|\leq C$
 such that, up to a subsequence,
\begin{equation}\label{ad2-24}
\begin{split}
& -\mu_n M_n^{-2}\to c,\;\;\;
\sqrt{-\mu_n}d(x_{in},\partial \Omega)\to+\infty,
\\
&\left\|\frac{u_{2n}}{\sqrt{-\mu_n}}-\sum\limits_{i=1}^ka_2^{-\frac{1}{2}}Q\left(\sqrt{-\mu_n }(\cdot-x_{in})\right)\right\|_{L^\infty(\Omega)}+\frac{1}{\sqrt{-\mu_n}}\left\|u_{1n}\right\|_{L^\infty(\Omega)}\to 0,
\end{split}
\end{equation}
as $n\to +\infty$. Moreover, we have $a_2=ka^*$ for some $k\geq 1$.
\end{theorem}
To prove Theorem \ref{adth2-2}, we need the following lemma.
\begin{lemma}\label{adlem2-3}
Suppose that $a_1\in (0,a^*)$, $a_2\in (ka^*,+\infty)$ for some  $k\geq 0$.  Suppose further that $m\geq 1$.
Then  the following equation (with respect to $\beta$)
\begin{equation}\label{ad2-25}
\frac{ka^*}{a_2}+\frac{ma^* \big(2\beta -(a_1+a_2)\big)}{\beta^2-a_1a_2}=1
\end{equation}
has no solution in the range $\beta\in (0,\min\{a_1,a_2\})$.
\end{lemma}
\begin{proof}
As  in Lemma \ref{lem2-3},  \eqref{ad2-25} is equivalent to
\[
\beta^2-2c_{m,k}\beta+(a_1+a_2)c_{m,k}-a_1a_2=0,\hbox{ or }(\beta-c_{m,k})^2=(c_{m,k}-a_1)(c_{m,k}-a_2),
\]
where
\[
c_{m,k}=\frac{ma^*}{1-\frac{ka^*}{a_2}}\geq ma^*>a_1.
\]
Thus \eqref{ad2-25} has solutions if and only $c_{m,k}\geq a_2$. Moreover, if $c_{m,k}=a_2$, then
\[
\beta=c_{m,k}=a_2=(k+m)a^*>a_1.
\]
If $c_{m,k}>a_2$, then \eqref{ad2-25} admits solutions
\[
\beta_m^+=c_{m,k}+\sqrt{(c_{m,k}-a_1)(c_{m,k}-a_2)}>\max\{a_1,a_2\},
\]
and
\[
\beta_m^-=c_{m,k}-\sqrt{(c_{m,k}-a_1)(c_{m,k}-a_2)}>\min\{a_1,a_2\}.
\]
Therefore, \eqref{ad2-25} has no solution in the range $\beta\in (0,\min\{a_1,a_2\})$.
\end{proof}

\begin{proof}[Proof of Theorem  \ref{adth2-2}]

Since $a_1\in (0, a^*)$,  as in Lemma \ref{lem2-4},
we can prove  that  $\rho_n\to+\infty$ cannot occur.

On the other hand,
\begin{equation}\label{ad2-33}
(\beta-ma^*)^2=(ma^*-a_1)(ma^*-a_2)
\end{equation}
has no solution for $\beta<a_1$. Combining this result and Lemma \ref{adlem2-3},
we can prove  that  $\rho_n\to \rho_0>0$ cannot occur either. See Steps~3--6
in the proof of Lemma \ref{lem2-5}.

Thus $\rho_n\to 0$.  We can proceed  as in the proof of Theorem \ref{th2-2}
to  obtain the desired result.
\end{proof}

Now we are ready to prove Theorem \ref{thm1-1}.

\begin{proof}[Proof of Theorem \ref{thm1-1}]
From Theorems~\ref{th2-1}, \ref{th2-2} and \ref{adth2-2},
for any solution
  $\left(u_1,u_2,\mu\right)\in \mathcal{H}\times \mathbb{R}$ of \eqref{1-9}, we have

\[
\|u_1\|_{L^\infty(\Omega)}+\|u_2\|_{L^\infty(\Omega)}\leq C_\delta.
\]
By \eqref{1-9}, it holds
\[
\begin{split}
\int_\Omega \big( |\nabla u_1|^2+|\nabla u_2|^2+V_1u_1^2+V_2u_2^2
\big)=&\int_\Omega \big(
a_1u_1^4+a_2u_2^4+2\beta u_1^2u_2^2+2\varepsilon u_1u_2
\big)+\mu
\\
\leq& a_1 \|u_1\|_{L^\infty(\Omega)}^2+a_2 \|u_2\|_{L^\infty(\Omega)}^2+2\beta \|u_1\|_{L^\infty(\Omega)}^2 +2\varepsilon +\mu\leq C_\delta^2.
\end{split}
\]
Thus, $\|(u_1,u_2)\|_{\mathcal{H}}\leq C_\delta$.  Moreover, from
\[
\mu=\int_\Omega \big( |\nabla u_1|^2+|\nabla u_2|^2+V_1u_1^2+V_2u_2^2
\big)-\int_\Omega \big(
a_1u_1^4+a_2u_2^4+2\beta u_1^2u_2^2+2\varepsilon u_1u_2
\big),
\]
we get $|\mu|\leq C_\delta. $
\end{proof}


\section{An eigenvalue problem}\label{4}

By Theorem~\ref{thm1-1} and the homotopy invariance of degree,  to  calculate the degree $d_{a_1,a_2,\beta}^\varepsilon$ for any $a_1,a_2\in (0,a^*)$, $\varepsilon\in (0,\varepsilon_0)$ and $\beta \in (\beta_1^*,\beta_2^*)$,
we may take $V_1=V_2=V$,   $\varepsilon=0$ and $a_1=a_2=a_0$ for some $a_0\in (0,a^*)$, then $\beta \in (\beta_1^*, \beta_2^*)= (2a^*-a_0,4a^*-a_0)$ . Moreover, we can further assume that  $V$ satisfies $(H_1)-(H_2)$ and
\begin{equation}\label{eq4-1}
\frac{a^*}{2}<a_0<a^*.
\end{equation}
Thus  $\eqref{1-9}$ becomes
\begin{equation}\label{eq4-2}
\left\{
\begin{array}{ll}
-\Delta u_{1}+V(x)u_{1}=a_{0}u_{1}^3+\beta u_{1}u_{2}^2+\mu u_{1}& \hbox{ in }\Omega,\\
-\Delta u_{2}+V(x)u_{2}=a_{0}u_{2}^3+\beta u_{2}u_{1}^2+\mu u_{2}&\hbox{ in }\Omega,\\
  0\leq u_{1},u_{2}\in H_0^1(\Omega),
 \\
 \int_\Omega (u_1^2+u_2^2)=1.
 \end{array}\right.
\end{equation}

In this section,  we study the Morse index of bubbling solutions of \eqref{eq4-2} for $\beta$ close to $2a^*-a_0$. By \cite[Theorem 4.2]{wy}, for any $\beta>a_0$, all nontrivial solutions of \eqref{eq4-2} are necessarily of the form 
\begin{equation}\label{eq4-3}
(u_1,u_2)=\left(\frac{\sqrt{2}}{2}u,\frac{\sqrt{2}}{2}u\right).
\end{equation}
Moreover,  $u$ solves
\begin{equation}\label{eq4-4}
\left\{
\begin{array}{ll}
-\Delta u+V(x)u=au^3+\mu u& \hbox{ in }\Omega,\\
  0\leq u\in H_0^1(\Omega),
 \\
 \int_\Omega u^2=1,
 \end{array}\right.
\end{equation}
where $a=(a_0+\beta)/2$. Note that $\beta$ is close to $2a^*-a_0$ if and only if $a$ is close to $a^*$. In what follows, we always assume that $a>a_0$, which gives $\beta>a_0$. Then  \eqref{eq4-3} gives a one-to-one map  between  nontrivial solutions of \eqref{eq4-2} and solutions of \eqref{eq4-4}.
On the other hand, since $a_0<a^*$, it follows from Theorem~A in the introduction that
\eqref{eq4-2} has no blow-up semi-trivial solution. Thus,
  as $\beta\to 2a^*-a_0$, all blow-up solutions of \eqref{eq4-2} must be nontrivial and of the form \eqref{eq4-3}.

All solutions of \eqref{eq4-4} blowing up as $a\to a^*$ are
 classified  as follows.

\begin{theoremB}[\cite{lpwy}] Suppose that $V$ satisfies $(V_1)$--$(V_4)$ and $(H_1)$--$(H_2)$.
Then

\begin{itemize}
\item If $(u_a,\mu_a)$ is a blow-up solution of \eqref{eq4-4} as $a\to a^*$, then
$u_a$ blows up at a single point $x_0$, which is a critical point of $V$.
Moreover, $\mu_a\to -\infty$ and

\begin{equation}\label{1-7-8}
(a-a^*)\mu_a^2=\frac12 \Delta V(x_0)\int_{\mathbb R^2} |x|^2 Q^2(x)+ o(1).
\end{equation}

\item
If $x_0\in\Omega$ is a non-degenerate critical point of $V$ with $\Delta V(x_0)\neq 0$, then there is a small constant $\theta>0$ such that \eqref{eq4-4}  has exactly one  solution $(u_a,\mu_a)$ of the following form
\begin{equation}\label{4-5}
u_a=\frac{1}{\sqrt{a}}PQ_{-\mu_a,x_a}+\omega_a,
\end{equation}
provided
\[
(i)\, a\in [a^*-\theta,a^*),\,\, \text{if}\,\, \Delta V(x_0)>0;\quad (ii)\, a\in (a^*,a^*+\theta],\,\, \text{if}\,\, \Delta V(x_0)<0,
\]
where  $PQ_{-\mu_a,x_a}$ is the unique solution
\begin{equation}\label{4-6}
\begin{cases}
-\Delta PQ_{-\mu_a,x_a}-\mu_a PQ_{-\mu_a,x_a}=Q_{-\mu_a,x_a}^3\;\;\; &\text{in}\;\;\Omega,
\\
PQ_{-\mu_a,x_a}=0\;\;\; &\text{on}\;\;\partial \Omega,
\end{cases}
\end{equation}
and $Q_{-\mu_a,x_a}(y)=\sqrt{-\mu_a} Q(\sqrt{-\mu_a} |y-x_a|)$.

Moreover, as $a \to a^*$, we have $\mu_a \to -\infty$, $x_a\to x_0$, and
\begin{equation}\label{4-7}
\|\omega_a\|_{L^\infty(\Omega)}=o(\sqrt{-\mu_a}),\quad \int_\Omega \Bigl(\frac{1}{-\mu_a}|\nabla \omega_a|^2+\Big(1-\frac{V(x)}{\mu_a}\Big)\omega_a^2\Bigr)=O\Big(\frac{1}{\mu_a^2}\Big).
\end{equation}
\end{itemize}
\end{theoremB}

 From
Theorem B, the blow-up solutions of \eqref{eq4-2} as  $\beta\to 2a^*-a_0$  are given by
\begin{equation}\label{eq4-8}
(u_{1\beta},u_{2\beta})=\left(\frac{\sqrt{2}}{2}u_a,\frac{\sqrt{2}}{2}u_a\right),
\end{equation}
where $u_a$ is a blow-up solution of \eqref{eq4-4}.

To compute the Leray-Schauder degree of the blow-up solutions of \eqref{eq4-2} as  $\beta\to 2a^*-a_0$,
in this section, we study an eigenvalue problem associated to the linearized
operator of a blow-up solution of \eqref{eq4-2}.
Recall that the operator $\mathbb{T}_{a_0,a_0,\beta}=\mathbb{T}_{a_0,a_0,\beta}^0$ is
defined in \eqref{1-6-8}.
 Formally, we can linearize $\mathbb{T}_{a_0,a_0,\beta}(u_1,u_2)$ at $(u_{1\beta},u_{2\beta})$ as follows.  For any $(\xi_1,\xi_2)\in \mathcal{H}$, we write $\mathbb{T}_{a_0,a_0,\beta}(u_{1\beta}+t\xi_1, u_{2\beta}+t\xi_2)=\mathbb{T}_{a_0,a_0,\beta}(u_{1\beta},u_{2\beta})+tK(\xi_1,\xi_2)+o(t)$ for small $t$.
 Then,  $(u_1,u_2)=K(\xi_1,\xi_2)$ satisfies
\begin{equation}\label{eq4-9}
\left\{\begin{array}{ll}
-\Delta u_1+V(x)u_1-\mu_a u_1=au_a^2(2\xi_2+\xi_1)+a_0u_a^2(\xi_1-\xi_2)+\mu u_{a}&\hbox{ in }\Omega,\\
-\Delta u_2+V(x)u_2-\mu_a u_2=au_a^2(2\xi_1+\xi_2)+a_0u_a^2(\xi_2-\xi_1)+\mu u_{a}&\hbox{ in }\Omega,\\
\int_{\Omega}\left(u_1+u_2\right)u_a=0,
\\
u_1=u_2=0 \hbox{ on }\partial\Omega,
\end{array}\right.
\end{equation}
where $\mu$ is a Lagrange multiplier. Let
\[
\mathcal{N}=\left\{(\xi_1,\xi_2)\in\mathcal{H}:\int_\Omega (\xi_1+\xi_2) u_a=0\right\}.
\]
We consider the following eigenvalue problem
\begin{equation}\label{eq4-10}
\lambda K(\xi_1,\xi_2)=(\xi_1,\xi_2),\;\;\; (\xi_1,\xi_2)\in \mathcal{N}.
\end{equation}
The Morse index of $(u_{1\beta},u_{2\beta})$ is the number of eigenvalues of \eqref{eq4-10} strictly less than one. To compute this number, we introduce the following matrix
\begin{equation}\label{eq4-11}
G(x)=\left(\begin{array}{ccc}
\frac{\partial^2 V(x)}{\partial x_1^2}&\frac{\partial^2 V(x)}{\partial x_1\partial x_2}&0\\
\frac{\partial^2 V(x)}{\partial x_2\partial x_1}&\frac{\partial^2 V(x)}{\partial x_2^2}&0\\
0&0&\Delta V(x)
\end{array}\right).
\end{equation}

The main result of this section is as follows.
\begin{theorem}\label{thm4-1}
$\lambda=1$ is not an eigenvalue of \eqref{eq4-10}.  The number of eigenvalues of \eqref{eq4-10} less than one equals the number of negative eigenvalues of $G(x_0)$.
\end{theorem}

As an application of Theorem~\ref{thm4-1}, we can use a similar argument as in \cite[Corollary 4.2]{cyy}  to show the following result. We omit the proof here.
\begin{corollary}\label{cor4-2}
Let $i(x_0)$ be the number of negative eigenvalues of $G(x_0)$.  Then we have
\[
deg(I-K, B_\varepsilon(0), 0)=(-1)^{i(x_0)}.
\]
\end{corollary}

Since all blow-up solutions for  \eqref{eq4-2} come from
those of \eqref{eq4-4}, we investigate the connections between the eigenvalues
 of \eqref{eq4-10} and those
 of
 the following eigenvalue problem

\begin{equation}\label{eq4-12}
\lambda \tilde{K} (\xi)=\xi ,\;\;\; \xi\in H_1(\Omega),\;\;\; \int_\Omega \xi u_a=0.
\end{equation}
Here, $u=\tilde{K} (\xi)$ is the unique  solution of
\begin{equation}\label{eq4-13}
\left\{\begin{array}{ll}
-\Delta u+V(x)u-\mu_a u=3au_a^2\xi+\mu u_{a}&\hbox{ in }\Omega,\\
\int_{\Omega}u u_a=0,
\\
u=0\;\;\;  \hbox{ on }\partial\Omega.
\end{array}\right.
\end{equation}

We recall the following result proved in \cite[Theorem 4.1]{cyy}.

\begin{theoremC}[\cite{cyy}] $\tilde{\lambda}=1$ is not an eigenvalue of \eqref{eq4-12}.  The number of eigenvalues of \eqref{eq4-12} less than one equals $i(x_0)$.
\end{theoremC}

Our main idea in the proof of  Theorem~\ref{thm4-1} is to
show that  \eqref{eq4-10} and
\eqref{eq4-12} have the same eigenvalues less than one.
 For this purpose, we will establish some  lemmas.

\begin{lemma}\label{lem4-4}
If $\lambda$ is an eigenvalue of \eqref{eq4-12} with eigenfunction
 $\xi$. Then $\lambda$ is  an eigenvalue of \eqref{eq4-10} with  eigenfunction
 $(\xi,\xi)$. Conversely, if $\lambda$ is  an eigenvalue of \eqref{eq4-10} with  eigenfunction
 $(\xi_1,\xi_2)$ and $\xi_1+\xi_2\ne 0$, then $\lambda$ is also an eigenvalue of  \eqref{eq4-12}.
\end{lemma}

\begin{proof}

Suppose that $(\xi, \lambda)$ satisfies \eqref{eq4-12}. Direct computation shows
that $(\xi, \xi, \lambda)$ satisfies  \eqref{eq4-10}.

On the other hand, if $(\xi_1,\xi_2,\lambda)$ satisfies  \eqref{eq4-10}, then
by adding the two equations in  \eqref{eq4-10}, we find that
$(\xi_1+\xi_2,\lambda)$ satisfies \eqref{eq4-12}. This shows that if
$\xi_1+\xi_2\ne 0$, then $\lambda$ is also an eigenvalue of  \eqref{eq4-12}.
\end{proof}

\begin{lemma}\label{lem4-3}
There exist $\delta,\theta>0$ small enough such that for $|a-a^*|<\theta$, if $(\xi_{1a},\xi_{2a};\lambda_a)$ is a nonzero solution of \eqref{eq4-10} satisfying $\xi_{1a}+\xi_{2a}=0$, then
\[
\lambda_a\geq 1+\delta.
\]
\end{lemma}

\begin{proof} Since $\xi_{1a}=-\xi_{2a}$, we deduce from \eqref{eq4-9} and \eqref{eq4-10} that
\[
\begin{cases}
-\Delta \xi_{1a}+V(x)\xi_{1a}-\mu_a \xi_{1a}=\lambda_a(2a_0-a) u_a^2 \xi_{1a}+\lambda_a\mu u_a,
\\
-\Delta \xi_{1a}+V(x)\xi_{1a}-\mu_a \xi_{1a}=\lambda_a(2a_0-a) u_a^2 \xi_{1a}-\lambda_a\mu u_a.
\end{cases}
\]
Thus, $\lambda_a\mu=0$ and $\xi_{1a}\neq 0$ satisfies
\begin{equation}\label{eq4-14}
\begin{cases}
-\Delta \xi_{1a}+V(x)\xi_{1a}-\mu_a \xi_{1a}=\lambda_a(2a_0-a) u_a^2 \xi_{1a}\;\;\;\text{in}\;\; \Omega,
\\
\xi_{1a}=0\;\;\;\text{on}\;\; \partial \Omega.
\end{cases}
\end{equation}
So for $a<2a_0$,
\[
\lambda_a=\frac{\int_\Omega (|\nabla \xi_{1a}|^2+V(x)|\xi_{1a}|^2-\mu_a|\xi_{1a}|^2)}{(2a_0-a) \int_\Omega u_a^2 |\xi_{1a}|^2}>0.
\]

We suppose by contradiction that up to a subsequence, $\lambda_a\to \lambda_0\in [0,1]$ as $a\to a^*$. Since for any $c\in \mathbb{R}$, $c\xi_{1a}$ is also a solution of \eqref{eq4-14}, we may assume that
\[
\int_\Omega -\frac{1}{\mu_a}|\nabla \xi_{1a}|^2+\left(1-\frac{V(x)}{\mu_a}\right)| \xi_{1a}|^2=1.
\]
Define
\[
\tilde{\xi}_{1a}(y)=\frac{1}{\sqrt{-\mu_a}} \xi_{1a} \left(\frac{1}{\sqrt{-\mu_a}} y+ x_a \right).
\]
Then by Theorem B, we have
\begin{equation}\label{eq4-15}
-\Delta \tilde{\xi}_{1a} +\left(1-\frac{1}{\mu_a} V\left(\frac{1}{\sqrt{-\mu_a}} y+ x_a \right)\right)  \tilde{\xi}_{1a} =\lambda_a \frac{2a_0-a}{a} \big(Q^2(y)+o(1)\big)  \tilde{\xi}_{1a} \;\;\; \text{in}\;\; \Omega_a,
\end{equation}
where $\Omega_a=\Big\{y\in \mathbb{R}^2:\;\frac{1}{\sqrt{-\mu_a}} y+ x_a\in \Omega \Big\}.$ Since $V\geq 0$ and
\[
\int_{\Omega_a} \Bigl[|\nabla \tilde{\xi}_{1a}|^2+\left(1-\frac{1}{\mu_a} V\left(\frac{1}{\sqrt{-\mu_a}} y+ x_a \right)\right)|\tilde{\xi}_{1a}|^2 \Bigr]\,dy=\int_\Omega -\frac{1}{\mu_a}|\nabla \xi_{1a}|^2+\left(1-\frac{V(x)}{\mu_a}\right)| \xi_{1a}|^2dx = 1,
\]
we have that $\{\tilde{\xi}_{1a}\}$ is bounded in $H^1_{loc}(\mathbb{R}^2)$.
So, there exists $\tilde{\xi}\in H^1(\mathbb{R}^2)$ such that up to a subsequence,
\[
\begin{split}
&\tilde{\xi}_{1a}\rightharpoonup \tilde{\xi} \;\;\text{in} \;\; H^1_{loc}(\mathbb{R}^2),
\\
&\tilde{\xi}_{1a}\to \tilde{\xi} \;\;\text{in} \;\;  L^2_{loc}(\mathbb{R}^2),
\end{split}
\]
and  $\tilde{\xi}$ satisfies
\[
-\Delta \tilde{\xi} +\tilde{\xi}=\lambda_0 \frac{2a_0-a^*}{a^*} Q^2 \tilde{\xi}  \;\;\;\text{in}\;\; \mathbb{R}^2.
\]

Note that $\xi=Q>0$ satisfies $-\Delta \xi +\xi = Q^2 \xi$ in $\mathbb R^2$. Thus $1$ is the first
eigenvalue of $$-\Delta \xi +\xi = \lambda Q^2 \xi\hbox{ in }\mathbb R^2.$$ This gives

\[
\frac{\int_{\mathbb{R}^2} |\nabla\tilde{\xi}|^2+ |\tilde{\xi}|^2 dy  }{\int_{\mathbb{R}^2} Q^2 \tilde{\xi}^2dy}\geq 1,\quad \forall\; \xi\ne 0.
\]

Suppose now that  $\tilde{\xi} \neq 0$, then we have
\[
\lambda_0 \frac{2a_0-a^*}{a^*}=\frac{\int_{\mathbb{R}^2} |\nabla\tilde{\xi}|^2+ |\tilde{\xi}|^2 dy  }{\int_{\mathbb{R}^2} Q^2 \tilde{\xi}^2dy}\geq 1.
\]
But for $a_0$ satisfying \eqref{eq4-1}, $\lambda_0 \frac{2a_0-a^*}{a^*}\leq  \frac{2a_0-a^*}{a^*}<1$, which is a contradiction. As a result, $\tilde{\xi} =0$, namely, $\tilde{\xi}_{1a}\to 0$ in   $L^2_{loc}(\mathbb{R}^2)$. Then  \eqref{eq4-15}, together with
$ \int_{\Omega_a} |\tilde{\xi}_{1a}|^2\le 1  $, gives

\[
\begin{split}
1=&\int_{\Omega_a} |\nabla \tilde{\xi}_{1a}|^2+\left(1-\frac{1}{\mu_a} V\left(\frac{1}{\sqrt{-\mu_a}} y+ x_a \right)\right)|\tilde{\xi}_{1a}|^2 dy
\\
=&\lambda_a \frac{2a_0-a}{a} \int_{\Omega_a} \big(Q^2(y)+o(1)\big)  \tilde{\xi}_{1a}^2dy
\\
\to& 0, \;\;\;\text{as}\;\; a\to a^*.
\end{split}
\]
This is impossible, and thus
 the proof is  completed.
\end{proof}

\begin{proof}[Proof of Theorem~\ref{thm4-1}]
Let $\lambda$ be an  eigenvalue of \eqref{eq4-10} with $\lambda\le 1$.
It follows from Lemma~\ref{lem4-3} that $\lambda$ is also an  eigenvalue of \eqref{eq4-12}.
Then Theorem~C gives $\lambda<1$.

On the other hand, Lemmas~\ref{lem4-4} and \ref{lem4-3} imply that the eigenvalues $\lambda<1$
for \eqref{eq4-10} and  \eqref{eq4-12} coincide. Thus, to prove Theorem~\ref{thm4-1},
it remains to prove that any eigenvalue $\lambda<1$ has the same multiplicity for both
\eqref{eq4-10} and  \eqref{eq4-12}.

Suppose that the  eigenfunctions for \eqref{eq4-12} corresponding to $\lambda<1$ are
$\{\xi_1, \cdots, \xi_m\}$ for some $m\ge 1$, satisfying

\[
\int_\Omega  \xi_i \xi_j=0\;\;\;\text{for}\;\; i\neq j.
\]
Then $(\xi_i, \xi_i)$, $i=1,\cdots, m$ are eigenfunctions for \eqref{eq4-10}
corresponding to $\lambda$.

Suppose now  \eqref{eq4-10} has another eigenfunction $(\hat \xi_1, \hat \xi_2)$ for $\lambda$. Then $\hat \xi_1+ \hat \xi_2$ is an eigenfunction of \eqref{eq4-12}
for  $\lambda$, satisfying

\begin{equation}\label{4-18}
\int_\Omega (\hat \xi_1 \xi_j+ \hat \xi_2 \xi_j)=0,\quad j=1\cdots, m.
\end{equation}
This gives

\[
\hat \xi_1+ \hat \xi_2=\sum_{i=1}^m c_i \xi_i.
\]
From \eqref{4-18}, $c_i=0$ for $i=1,\cdots,m$. So, $\xi_1^{(j_0)}+\xi_2^{(j_0)}=0$
and we obtain  a contradiction.
\end{proof}


\section{The degree of the blow-up solutions}\label{5}

In this section, we assume that $a_1=a_2=a_0$, $\beta \in (\beta_1^*, \beta_2^*)= (2a^*-a_0,4a^*-a_0)$ for some $a_0\in (0,a^*)$ with $\frac{a^*}{2}<a_0<a^*$, and  $V_1=V_2=V$ satisfies $(H_1)-(H_2)$.

The discussion in Section~\ref{4} shows   that all the blow-up solutions of \eqref{eq4-2}
as $\beta\to \beta^*_1$ are of the form $(u_{1\beta},u_{2\beta})=\frac{\sqrt{2}}{2}(u_a,u_a)$
where $\beta=2a-a_0$, and $u_a$ is a blow-up solution of \eqref{eq4-4} as $a\to a^*$.
 Theorem B provides all blow-up solutions of \eqref{eq4-4}, which  are isolated for $\beta$ close to $\beta_1^*$. Thus the following Leray-Schauder degree is well-defined:
\begin{equation}\label{5-1}
d(u_{1\beta},u_{2\beta}):=deg\big(I-\mathbb{T}_{a_0,a_0,\beta}, B_\epsilon^a(u_{1\beta},u_{2\beta}), 0\big),
\end{equation}
where $\epsilon>0$ is a small constant, depending on $a$,
\[
B_\epsilon^a(u_{1\beta},u_{2\beta})=\left\{(u_1,u_2):\|(u_1-u_{1\beta},u_2-u_{2\beta})\|_{-\mu_a}<\epsilon\right\}
\]
and
\[
\|(u,v)\|_{-\mu_a}^2=\int_\Omega  \left(\frac{1}{-\mu_a} \big(|\nabla u|^2+|\nabla v|^2\big)+\Big(1-\frac{V(x)}{\mu_a}\Big) (u^2+v^2) \right).
\]
We recall that fixed points of $\mathbb{T}_{a_0,a_0,\beta}$ correspond to solutions of \eqref{eq4-2}.

In this section, we establish the following result.
\begin{theorem}\label{thm5-1}
We have
\[
deg \big(I-\mathbb{T}_{a_0,a_0,\beta}, B_\epsilon^a(u_{1\beta},u_{2\beta}), 0\big)=deg \big(I-K, B_\epsilon^a(0,0), 0\big),
\]
where the operator $K$ is defined in \eqref{eq4-9}.

\end{theorem}

Recall that formally  $I-K$ is the linearized operator of
$(I-\mathbb{T}_{a_0,a_0,\beta})(u_{1\beta}+\xi_1,u_{2\beta}+\xi_2)$. So it
is natural to expect that for $t\in [0,1]$ and $(\xi_1,\xi_2)\in \partial B_\epsilon^a(0,0)  $, it holds
\begin{equation}\label{20-7-8}
 t [(u_{1\beta}+\xi_1,u_{2\beta}+\xi_2)-\mathbb{T}_{a_0,a_0,\beta}(u_{1\beta}+\xi_1,u_{2\beta}+\xi_2)]
+(1-t)[ (\xi_1,\xi_2) -K(\xi_1,\xi_2)]\ne 0.
\end{equation}
Once \eqref{20-7-8} is proved, Theorem~\ref{thm5-1} follows.

To prove \eqref{20-7-8}, we need to   estimate the following term
\begin{equation}\label{5-2}
\mathcal{N}(\xi_1,\xi_2):=(I-\mathbb{T}_{a_0,a_0,\beta})(u_{1\beta}+\xi_1,u_{2\beta}+\xi_2)-(I-K)(\xi_1,\xi_2).
\end{equation}

The estimate of \eqref{5-2} consists of the following steps.

{\bf Step~1}. We begin by showing that for any $(\xi_1,\xi_2)\in B_\epsilon^a(0,0)$, there exists $(\eta_1(\xi_1,\xi_2), \eta_2(\xi_1,\xi_2))$ such that
\begin{equation}\label{5-3}
\mathbb{T}_{a_0,a_0,\beta}(u_{1\beta}+\xi_1,u_{2\beta}+\xi_2)=\left(u_{1\beta}+\eta_1(\xi_1,\xi_2), u_{2\beta}+\eta_2(\xi_1,\xi_2)\right).
\end{equation}

{\bf Step~2}. The pair $(\eta_1(\xi), \eta_2(\xi))=:\eta$ in Step~1 satisfies
$\|\eta(\xi)\|_{-\mu_a}\leq C\|\xi\|_{-\mu_a}$.

{\bf Step~3}. We prove that $\|\mathcal{N}(\xi)\|_{-\mu_a}\leq C\|\xi\|_{-\mu_a}^2$.

Once we obtain the above estimates, we see that  $I-K$ is indeed the linearized operator of
$(I-\mathbb{T}_{a_0,a_0,\beta})(u_{1\beta}+\xi_1,u_{2\beta}+\xi_2)$. Hence \eqref{20-7-8}
follows from the invertibility of $I-K$.

\bigskip
Now we proceed to Step~1.
Note that \eqref{5-3} is equivalent to the existence of $\eta$ such that
\begin{equation}\label{30-7-8}
F(\xi,\eta)=0 \quad \hbox{and} \quad  \mu\neq 0,
\end{equation}
where $F=(F_1,F_2)$ is defined as

\begin{align*}
F_1(\xi,\eta)=&-\Delta (u_{1a}+\eta_1)+V(x)(u_{1a}+\eta_1)-a_0|u_{1a}+\xi_1|^3\\
&-\beta|u_{1a}+\xi_1|(u_{2a}+\xi_2)^2-\mu(u_{1a}+\eta_1),\\
F_2(\xi,\eta)=&-\Delta (u_{2a}+\eta_2)+V(x)(u_{2a}+\eta_2)-a_0|u_{2a}+\xi_2|^3\\
&-\beta|u_{2a}+\xi_2|(u_{1a}+\xi_1)^2-\mu(u_{2a}+\eta_2),
\end{align*}
 $\xi=(\xi_1,\xi_2)$, $\eta=(\eta_1,\eta_2)$, and
\begin{equation}\label{0-9-8}
\begin{split}
\mu=&\int_{\Omega}\Bigl(|\nabla (u_{1\beta}+\eta_1)|^2+V(x)(u_{1\beta}+\eta_1)^2-a_0|u_{1\beta}+\xi_1|^3(u_{1\beta}+\eta_1)\Bigr)\\
&+\int_{\Omega}\Bigl(|\nabla (u_{2\beta}+\eta_2)|^2+V(x)(u_{2\beta}+\eta_2)^2-a_0|u_{2\beta}+\xi_2|(u_{2\beta}+\eta_2)\Bigr)\\
&-\int_{\Omega}\Bigl(\beta|u_{1\beta}+\xi_1|(u_{2\beta}+\xi_2)^2(u_{1\beta}+\eta_1)+\beta|u_{2\beta}
+\xi_2|(u_{1\beta}+\xi_1)^2(u_{2\beta}+\eta_2)\Bigr).
\end{split}
\end{equation}

  By direct calculations, we find that
\begin{align*}
(F_1)_\eta(0,0)[\varphi]=&-\Delta \varphi_1+V(x)\varphi_1-\mu_a\varphi_1\\
&-u_{1\beta}\left[2\int_\Omega\Bigl(\nabla u_{1\beta}\nabla \varphi_1+V(x)u_{1\beta}\varphi_1+\nabla u_{2\beta}\nabla \varphi_2+V(x)u_{2\beta}\varphi_2\Bigr)\right.\\
&\left.-\int_{\Omega}\Bigl(a_0u_{1\beta}^3\varphi_1+a_0u_{2\beta}^3\varphi_2 +\beta u_{1\beta}u_{2\beta}^2\varphi_1+\beta u_{2\beta}u_{1\beta}^2\varphi_2\Bigr)\right]
\\
=&-\Delta \varphi_1+V(x)\varphi_1-\mu_a\varphi_1-u_a\int_\Omega \Big(\frac{a}{2}u_a^3+\mu_au_a\Big)(\varphi_1+\varphi_2),
\end{align*}
and
\begin{align*}
(F_2)_\eta(0,0)[\varphi]=&-\Delta \varphi_2+V(x)\varphi_2-\mu_a\varphi_2\\
&-u_{2\beta}\left[2\int_\Omega\Bigl(\nabla u_{1\beta}\nabla \varphi_1+V(x)u_{1\beta}\varphi_1+\nabla u_{2\beta}\nabla \varphi_2+V(x)u_{2\beta}\varphi_2\Bigr)\right.\\
&\left.-\int_{\Omega}\Bigl(a_0u_{1\beta}^3\varphi_1+a_0u_{2\beta}^3\varphi_2 +\beta u_{1\beta}u_{2\beta}^2\varphi_1+\beta u_{2\beta}u_{1\beta}^2\varphi_2\Bigr)\right]
\\
=&-\Delta \varphi_2+V(x)\varphi_2-\mu_a\varphi_2-u_a\int_\Omega \Big(\frac{a}{2}u_a^3+\mu_au_a\Big)(\varphi_1+\varphi_2),
\end{align*}
where $\varphi=(\varphi_1,\varphi_2)$ and
$
\mu_a=\int_{\Omega}\left(|\nabla u_{a}|^2+V(x)u_{a}^2-au_{a}^4\right).
$

In order to use the implicit function theorem to solve \eqref{30-7-8},
we need the following result.

\begin{proposition}\label{pro5-2}
There exist $\sigma_0$, $\theta_0>0$ such that for $a\in(a^*-\theta_0,a^*+\theta_0)$ and $\varphi=(\varphi_1,\varphi_2)\in\mathcal{H}$,
\[
\|i\circ F_\eta(0,0)[(\varphi_{1},\varphi_{2})]\|_{-\mu_{a}}\geq \sigma_0\|(\varphi_1,\varphi_2)\|_{-\mu_{a}},
\]
where $i$ is the map from $\mathcal{H}^*$ (the dual of $\mathcal{H}$)  to $\mathcal{H}$, satisfying
\begin{align*}
&\langle i\circ  F_\eta(0,0)[(\varphi_{1},\varphi_{2})],(h_1,h_2)\rangle_{-\mu_\beta}\\
=&\int_{\Omega}\Bigl[-\frac{1}{\mu_\beta}\nabla \varphi_{1}\nabla h_1+\left(1-\frac{1}{\mu_\beta}V(x)\right)\varphi_{1}h_1\Bigr]+\int_{\Omega}
\Bigl[-\frac{1}{\mu_\beta}\nabla \varphi_{2}\nabla h_2
+\left(1-\frac{1}{\mu_\beta}V(x)\right)\varphi_{2}h_2\Bigr]\\
&+\frac{1}{\mu_a}\left(\int_\Omega u_{a}(h_1+h_2)\right)\int_\Omega \Big(\frac{a}{2}u_a^3+\mu_au_a\Big)(\varphi_1+\varphi_2).
\end{align*}
\end{proposition}

\begin{proof}
Suppose for contradiction that there exist $a_n\to a^*$, $(\varphi_{1n},\varphi_{2n})\in \mathcal{H}$ with $\|(\varphi_{1n},\varphi_{2n})\|_{-\mu_{a_n}}=1$ such that
\begin{equation}\label{4-3}
\|i\circ F_\eta(0,0)[(\varphi_{1n},\varphi_{2n})]\|_{-\mu_{a_n}}<\frac{1}{n}.
\end{equation}
Denote $\mu_n=\mu_{a_n}$ and $u_n=u_{a_n}$.  Then for any $(h_1,h_2)\in \mathcal{H}$, we have
\begin{equation}\label{5-5}
\begin{split}
\int_{\Omega}\Bigl[-\frac{1}{\mu_n}&\nabla \varphi_{1n}\nabla h_1+\left(1-\frac{1}{\mu_n}V(x)\right)\varphi_{1n}h_1\Bigr]+\int_{\Omega}\Bigl[-\frac{1}{\mu_n}\nabla \varphi_{2n}\nabla h_2+\left(1-\frac{1}{\mu_n}V(x)\right)\varphi_{2n}h_2  \Bigr]
\\
&+\frac{1}{\mu_n}\left(\int_\Omega u_{n}(h_1+h_2)\right)\int_\Omega \Big(\frac{a_n}{2}u_n^3+\mu_nu_n\Big)(\varphi_{1n}+\varphi_{2n})=O\left(\frac{1}{n}\right)\|(h_1,h_2)\|_{-\mu_n}.
\end{split}
\end{equation}

Let $x_n$ be the maximum point of $u_{n}$. Set
\[
\begin{split}
&v_n=\sqrt{\frac{a_n}{-\mu_n}}u_{n}\left(\frac{1}{\sqrt{-\mu_n}}y+x_n\right),\;\; \tilde{\varphi}_{in}=\sqrt{\frac{a_n}{-\mu_n}}\varphi_{in}\left(\frac{1}{\sqrt{-\mu_n}}y+x_n\right),
\\
& h_{1n}=\sqrt{\frac{-\mu_n}{a_n}} h(\sqrt{-\mu_n}(x-x_n)), \;\;\; h \in C_0^\infty(\mathbb{R}^2).
\end{split}
\]
Then by \eqref{5-5}, we have
\begin{equation}\label{5-6}
\begin{split}
O\left(\frac{1}{n}\right)\|(h_{1n},0)\|_{-\mu_n}
=&\int_{\Omega}\Bigl[-\frac{1}{\mu_n}\nabla \varphi_{1n}\nabla h_{1n}+\left(1-\frac{1}{\mu_n}V(x)\right)\varphi_{1n}h_{1n} \Bigr]
\\
&+\frac{1}{\mu_n}\left(\int_\Omega u_{n}h_{1n}\right)\int_\Omega \Big(\frac{a_n}{2}u_n^3+\mu_nu_n\Big)(\varphi_{1n}+\varphi_{2n})
\\
=&\frac{1}{a_n}\int_{\Omega_n}\left( \nabla \tilde{\varphi}_{1n} \nabla h+\left(1-\frac{1}{\mu_n}V(\frac{1}{\sqrt{-\mu_n}}y+x_n)\right)\tilde{\varphi}_{1n} h\right)
\\
&+\frac{1}{2a_n^2} \left(\int_{\Omega_n} v_n h\right)\int_{\Omega_n} (2v_n-v_n^3)(\tilde{\varphi}_{1n}+\tilde{\varphi}_{1n}),
\end{split}
\end{equation}
where  $\Omega_n=\big\{y\in\mathbb{R}^2:\frac{1}{\sqrt{-\mu_n}}y+x_n\in \Omega\big\}$.
Note that
\[
\int_{\Omega}\Bigl[-\frac{1}{\mu_n}|\nabla h_{1n}|^2+\left(1-\frac{1}{\mu_n}V(x)\right)h_{1n}^2
\Bigr]=\frac{1}{a_n} \int_{\Omega_n}\bigg( |\nabla h|^2+\left(1-\frac{1}{\mu_n}V(\frac{1}{\sqrt{-\mu_n}}y+x_n)\right)h^2 \bigg)\leq C.
\]
Moreover, since  $\|(\varphi_{1n},\varphi_{2n})\|_{-\mu_{a_n}}=1$, we have that
 $(\tilde{\varphi}_{1n},\tilde{\varphi}_{2n})$ is bounded in $H_{loc}^1(\mathbb{R}^2)\times H_{loc}^1(\mathbb{R}^2)$. Then,  we can assume that
\begin{align*}
&\tilde{\varphi}_{in}\rightharpoonup \varphi_i\hbox{ weakly in }H_{loc}^1(\mathbb{R}^2),\\
&\tilde{\varphi}_{in}\to \varphi_i\hbox{ strongly in }L_{loc}^p(\mathbb{R}^2), 1\leq p<\infty.
\end{align*}
It follows from \eqref{5-6} that
\begin{equation}\label{5-7}
-\Delta\varphi_1+\varphi_1=\frac{1}{2a^*}Q\int_{\mathbb{R}^2} (Q^3-2Q)(\varphi_1+\varphi_2).
\end{equation}
Similarly, we can show that
\begin{equation}\label{5-8}
-\Delta\varphi_2+\varphi_2=\frac{1}{2a^*}Q\int_{\mathbb{R}^2} (Q^3-2Q)(\varphi_1+\varphi_2).
\end{equation}
Testing \eqref{5-7} and \eqref{5-8} by $Q$ and adding together,
we get
\[
\begin{split}
\int_{\mathbb{R}^2} Q^3(\varphi_1+\varphi_2)=&\int_{\mathbb{R}^2}  (-\Delta Q+Q)(\varphi_1+\varphi_2)
\\
=&\int_{\mathbb{R}^2} \big( -\Delta(\varphi_1+\varphi_2)+(\varphi_1+\varphi_2)  \big)Q
\\
=&  \int_{\mathbb{R}^2} (Q^3-2Q)(\varphi_1+\varphi_2),
\end{split}
\]
which gives
\[
\int_{\mathbb{R}^2}Q(\varphi_1+\varphi_2)=0.
\]
Multiplying \eqref{5-7}  by $\varphi_1$, \eqref{5-8} by $\varphi_2$, integrating by parts over $\mathbb{R}^2$ and adding together, we obtain that
\[
\int_{\mathbb{R}^2}\left(|\nabla \varphi_1|^2+|\nabla \varphi_2|^2+\varphi_1^2+\varphi_2^2\right)=0.
\]
Thus, $(\varphi_1,\varphi_2)=(0,0)$.
Then taking $(h_1,h_2)=(\varphi_{1n}, \varphi_{2n})$    in \eqref{5-5},  we get   $\|(\varphi_{1n},\varphi_{2n})\|_{-\mu_n}\to 0$ This is  a contradiction to $\|(\varphi_{1n},\varphi_{2n})\|_{-\mu_n}=1$.
\end{proof}

Since  $F(0,0)=0$, by Proposition \ref{pro5-2} and the Implicit Function Theorem, we know that for any $\xi:=(\xi_1,\xi_2)\in B_{\epsilon}^a(0)$, there exists a unique $\eta(\xi):=(\eta_1(\xi),\eta_2(\xi))\in C\left(B_{\epsilon}^a(0),B_{\epsilon_1}^a(0)\right)$ such that
\[
F\left(\xi,\eta(\xi)\right)=0,  \hbox{ for any }\xi\in B_{\epsilon}^a(0).
\]
Here,  $\epsilon,\epsilon_1$ are two small positive constants.
Furthermore, we have the following estimate for $\eta(\xi)$.
\begin{lemma}\label{lem5-3}
For any $a$ close to $a^*$ and $\xi\in B_{\epsilon}^a(0)$, it holds that
\[
\|\eta(\xi)\|_{-\mu_a}\leq C\|\xi\|_{-\mu_a}.
\]
\end{lemma}

\begin{proof}
Note that $F(\xi,\eta)=0$. We have
\begin{align*}
-F_\eta(0,0)[\eta]=&F(\xi,\eta)-F_\eta(0,0)[\eta]\\
=&F(\xi,0)+\left(F_\eta(\xi,0)[\eta]-F_\eta(0,0)[\eta]\right)+R(\xi,\eta),
\end{align*}
where $R=(R_1,R_2)=F(\xi,\eta)-F(\xi,0)-F_\eta(\xi,0)[\eta]$.
Direct calculation shows that
\begin{align*}
R_i(\xi,\eta)=&-(u_{i\beta}+\eta_i)\left(\int_\Omega|\nabla \eta_1|^2+V(x)\eta_1^2+|\nabla \eta_2|^2+V(x)\eta_2^2\right)\\
&-\eta_i\left(\int_\Omega 2\nabla u_{1\beta}\nabla \eta_1+2V(x)u_{1\beta}\eta_1-a_0|u_{1\beta}+\xi_1|^3\eta_1-\beta|u_{1\beta}+\xi_1|(u_{2\beta}
+\xi_2)^2\eta_1\right.\\
&\left.+\int_\Omega 2\nabla u_{2\beta}\nabla \eta_2+2V(x)u_{2\beta}\eta_2-a_0|u_{2\beta}+\xi_2|^3\eta_2-\beta|u_{2\beta}
+\xi_2|(u_{1\beta}+\xi_1)^2\eta_2\right).
\end{align*}
It follows from Proposition \ref{pro5-2} that
\begin{equation}\label{5-9}
\begin{aligned}
\|\eta\|_{-\mu_a}\leq& C\|i\circ F_\eta(0,0)[\eta]\|_{-\mu_a}\\
\leq& \|i\circ F(\xi,0)\|_{-\mu_a}+\|i\circ\left(F_\eta(\xi,0)[\eta]-F_\eta(0,0)[\eta]\right)\|_{-\mu_a}+ \|i\circ R(\xi,\eta))\|_{-\mu_a}.
\end{aligned}
\end{equation}

On the other hand,  for any $h\in\mathcal{H}$,
\begin{equation}\label{1-9-8}
\begin{split}
&\Bigl|\big\langle i\circ F(\xi,0),h\big\rangle_{-\mu_a}\Bigr|\\
 =  &\frac1{-\mu_a}   \Bigl|
\int_{\Omega}   \Bigl( a_0 \bigl(|u_{1\beta}+\xi_1|^3 -u_{1\beta}^3   \bigr)h_1  +
a_0 \bigl(|u_{2\beta}+\xi_2|^3 -u_{2\beta}^3   \bigr)h_2 \\ &\qquad+
\beta \bigl(|u_{1\beta}+\xi_1|(u_{2\beta}+\xi_2)^2- u_{1\beta}u_{2\beta}^2  \bigr)h_1
 +  \beta
\bigl(|u_{2\beta}+\xi_2|(u_{1\beta}+\xi_1)^2- u_{2\beta}u_{1\beta}^2  \bigr)h_2\Bigr)\\
&+ \mu\int_{\Omega} ( u_{1\beta} h_1+ u_{2\beta}h_2)\Bigr| \\
 \le &C\|\xi\|_{-\mu_a}\|h\|_{-\mu_a}.
\end{split}
\end{equation}
Similarly, we can also prove that 
\begin{equation}\label{2-9-8}
\begin{aligned}
\left| \big\langle i\circ \big(F_{\eta}(\xi,0)[\eta]-F_{\eta}(0,0)[\eta]\big),h\big\rangle_{-\mu_a}\right|
\leq C\|h\|_{-\mu_a}\|\eta\|_{-\mu_a}\|\xi\|_{-\mu_a},
\end{aligned}
\end{equation}
and

\begin{equation}\label{3-9-8}
\begin{aligned}
\left|\langle i\circ R(\xi,\eta),h\rangle_{-\mu_a}\right|
\leq& C\|h\|_{-\mu_a}\|\eta\|^2_{-\mu_a}.
\end{aligned}
\end{equation}

Combining \eqref{5-9}--\eqref{3-9-8}, we conclude that
\[
\|\eta(\xi)\|_{-\mu_a}\leq C\left(\|\xi\|_{-\mu_a}+\|\eta(\xi)\|_{-\mu_a}\|\xi\|_{-\mu_a}+\|\eta(\xi)\|^2_{-\mu_a}\right),
\]
which, together with $\|\xi\|_{-\mu_a}\leq \epsilon$ and $\|\eta(\xi)\|_{-\mu_a}\leq \epsilon_1$,  gives
\[
\big(1-C(\epsilon+\epsilon_1)\big) \|\eta(\xi)\|_{-\mu_a} \leq C\|\xi\|_{-\mu_a}.
\]
Since $\epsilon,\epsilon_1>0$ are small, the result follows.
\end{proof}

From Lemma \ref{lem5-3}, we find that
\begin{align*}
\mu=&\int_{\Omega}\Bigl(|\nabla (u_{1\beta}+\eta_1)|^2+V(x)(u_{1\beta}+\eta_1)^2-a_0|u_{1\beta}+\xi_1|^3(u_{1\beta}+\eta_1)\Bigr)\\
&+\int_{\Omega}\Bigl(|\nabla (u_{2\beta}+\eta_2)|^2+V(x)(u_{2\beta}+\eta_2)^2-a_0|u_{2\beta}+\xi_2|(u_{2\beta}+\eta_2)\Bigr)\\
&-\int_{\Omega}\Bigl(\beta|u_{1\beta}+\xi_1|(u_{2\beta}+\xi_2)^2(u_{1\beta}+\eta_1)+\beta|u_{2\beta}
+\xi_2|(u_{1\beta}+\xi_1)^2(u_{2\beta}+\eta_2)\Bigr)\\
=&\int_{\Omega}\Bigl(|\nabla u_{1\beta}|^2+V(x)u_{1\beta}^2-a_0u_{1\beta}^4-\beta u_{1\beta}^2u_{2\beta}^2
\Bigr)\\
&+\int_{\Omega}\Bigl(|\nabla u_{2\beta}|^2+V(x)u_{2\beta}^2-a_0u_{2\beta}^4-\beta u_{1\beta}^2u_{2\beta}^2\Bigr)+O\left(-\mu_a\|\xi\|_{-\mu_a}\right)\\
=&\mu_a+O\left(-\mu_a\|\xi\|_{-\mu_a}\right)\neq 0.
\end{align*}
Thus, we have proved \eqref{5-3}.

\begin{lemma}\label{lem5-4}
For $a$ close to $a^*$ and $\xi\in B_{\epsilon}^a(0)$, we have
\[
\|\mathcal{N}(\xi)\|_{-\mu_a}\leq C\|\xi\|_{-\mu_a}^2.
\]
\end{lemma}

\begin{proof}
By definition of $\mathcal{N}$, we have
\[
\mathcal{N}(\xi)=-\eta(\xi)+K(\xi),
\]
where $\eta(\xi)=(\eta_1(\xi_1,\xi_2),\eta_2(\xi_1,\xi_2))$ satisfies $\int_{\Omega}\bigl(|u_{1\beta}+\eta_1|^2+|u_{2\beta}+\eta_2|^2\bigr)=1$  and
\begin{equation}\label{5-10}
\left\{\begin{aligned}
-\Delta \eta_1+V(x)\eta_1-\mu_a\eta_1=&a_0\big(|u_{1\beta}+\xi_1|^3-u_{1\beta}^3\big)+\beta\big(|u_{1\beta}
+\xi_1|(u_{2\beta}+\xi_2)^2-u_{1\beta}^3\big) \\
&+(\mu-\mu_a)(u_{1\beta}+\eta_1),\\
-\Delta \eta_2+V(x)\eta_2-\mu_a\eta_2=&a_0\big( |u_{2\beta}+\xi_2|^3-u_{2\beta}^3\big)+\beta\big(|u_{2\beta}+\xi_2|(u_{1\beta}+\xi_1)^2-u_{2\beta}^3\big)\\
&+(\mu-\mu_a)(u_{2\beta}+\eta_2),
\end{aligned}\right.
\end{equation}
where $\mu$ is defined in \eqref{0-9-8}.

We recall that $K(\xi)=(K_1(\xi_1,\xi_2),K_2(\xi_1,\xi_2))$ satisfies $\int_\Omega (K_1+K_2)u_a=0$ and
\begin{equation}\label{5-11}
\left\{\begin{aligned}
-\Delta K_1+V(x)K_1-\mu_aK_1=&3a_0u_{1\beta}^2\xi_1+\beta u_{2\beta}^2\xi_1+2\beta u_{1\beta}u_{2\beta}\xi_2+\lambda u_{1\beta},  \\
-\Delta K_2+V(x)K_2-\mu_aK_2=&3a_0u_{2\beta}^2\xi_2+\beta u_{1\beta}^2\xi_2+2\beta u_{1\beta}u_{2\beta}\xi_1+\lambda u_{2\beta},
\end{aligned}\right.
\end{equation}
where $\lambda$ is a Lagrange multiplier.
See \eqref{eq4-9}.    Multiplying the first equation  by $u_{1\beta}$ and the second  equation  by $u_{2\beta}$  and integrating by parts,
using \eqref{eq4-2} and   $\int_\Omega (u_{1\beta}^2+u_{2\beta}^2)=1$, we get
\begin{equation}\label{60-9-8}
\begin{split}
\lambda=& (\beta+a_0)\int_\Omega \big(u_{1\beta}^3 K_1+u_{2\beta}^3 K_2\big)
-3a_0\int_\Omega (u_{1\beta}^3\xi_1+u_{2\beta}^3\xi_2)-3\beta \int_\Omega
\big(u_{1\beta}^2 u_{2\beta}\xi_2+u_{2\beta}^2 u_{1\beta}\xi_1\big).
\end{split}
\end{equation}

On the other hand, it follows from \eqref{5-10} and \eqref{5-11}  that $\mathcal{N}(\xi)=(\mathcal{N}_{1}(\xi_1,\xi_2),\mathcal{N}_{2}(\xi_1,\xi_2))$ satisfies
\begin{equation}\label{5-12}
\left\{\begin{aligned}
-\Delta \mathcal{N}_{1}+V(x)\mathcal{N}_{1}-\mu_a\mathcal{N}_{1}=&-a_0\big(|u_{1\beta}+\xi_1|^3-u_{1\beta}^3-
3u_{1\beta}^2\xi_1\big)-\beta\big(|u_{1\beta}+\xi_1|(u_{2\beta}+\xi_2)^2
\\
-u_{1\beta}^3&-u_{2\beta}^2\xi_1-2u_{1\beta}u_{2\beta}\xi_2\big)
+\big(\lambda-(\mu-\mu_a)\big)u_{1\beta}+(\mu_a-\mu)\eta_1
\\
-\Delta \mathcal{N}_{2}+V(x)\mathcal{N}_{2}-\mu_a\mathcal{N}_{2}=&-a_0\big(|u_{2\beta}+\xi_2|^3-u_{2\beta}^3
-3u_{2\beta}^2\xi_2\big)-\beta\big(|u_{2\beta}+\xi_2|(u_{1\beta}+\xi_1)^2
\\
-u_{2\beta}^3&-u_{1\beta}^2\xi_2-2u_{1\beta}u_{2\beta}\xi_1\big)+\big(\lambda-(\mu-\mu_a)\big)
u_{2\beta}+(\mu_a-\mu)\eta_2,
\end{aligned}\right.
\end{equation}
from which, we deduce
 \begin{equation}\label{30-9-8}
\begin{split}
\int_\Omega \Bigl(|\nabla \mathcal{N}_{1}|^2+&V(x)\mathcal{N}_{1}^2-\mu_a|\mathcal{N}_{1}|^2\Bigr)+
\int_\Omega \Bigl(|\nabla \mathcal{N}_{2}|^2+V(x)\mathcal{N}_{2}^2-\mu_a|\mathcal{N}_{2}|^2\Bigr)\\
=&-\beta\int_{\Omega}\left(|u_{1\beta}+\xi_1|(u_{2\beta}+\xi_2)^2
-u_{1\beta}^3-u_{2\beta}^2\xi_1-2u_{1\beta}u_{2\beta}\xi_2\right)\mathcal{N}_{1}\\
&-\beta\int_{\Omega}\left(|u_{2\beta}+\xi_2|(u_{1\beta}+\xi_1)^2
-u_{2\beta}^3-u_{1\beta}^2\xi_2-2u_{1\beta}u_{2\beta}\xi_1\right)\mathcal{N}_{2}\\
&-a_0\int_{\Omega}\big(|u_{1\beta}+\xi_1|^3-u_{1\beta}^3-3u_{1\beta}^2\xi_1\big)\mathcal{N}_{1}
-a_0\int_{\Omega}\big(|u_{2\beta}+\xi_2|^3-u_{2\beta}^3-3u_{2\beta}^2\xi_2\big)\mathcal{N}_{2}\\
&+\left(\lambda-(\mu-\mu_a)\right)\left(\int_{\Omega}u_{1\beta}\mathcal{N}_{1}
+\int_{\Omega}u_{2\beta}\mathcal{N}_{2}\right)+(\mu_a-\mu)\left(\int_\Omega \eta_1\mathcal{N}_{1}
+\int_\Omega \eta_2\mathcal{N}_{2}\right).
\end{split}
 \end{equation}

  First, we know    that

 \begin{equation}\label{31-9-8}
 \text{ LHS of \eqref{30-9-8}  }  \ge c'  \|\mathcal{N}\|_{-\mu_a}^2
\end{equation}
for some $c'>0$. Next, we estimate each term in the right hand side of
\eqref{30-9-8}.  By the Gagliardo-Nirenberg inequality, we find
\[
\int_{\Omega}|\xi_i|^4\leq C\int_{\Omega}|\nabla \xi_i|^2\int_{\Omega}\xi_i^2\leq C(-\mu_a)\|\xi_i\|_{-\mu_a}^4, \]
from which, we deduce

\begin{equation}\label{5-13}
\begin{aligned}
\Bigl|\int_{\Omega}\left(|u_{i\beta}+\xi_i|^3-u_{i\beta}^3-3u_{i\beta}^2\xi_i\right)\mathcal{N}_{i}
\Bigr|\leq C|\mu_a|\|\xi_i\|_{-\mu_a}^2\|\mathcal{N}_{i}\|_{-\mu_a},
\end{aligned}
\end{equation}
and
\begin{equation}\label{5-14}
\begin{aligned}
\Bigl|\int_{\Omega}\left(|u_{i\beta}+\xi_i|(u_{(3-i)\beta}+\xi_{3-i})^2-u_{i\beta}^3-u_{(3-i)\beta}^2\xi_i
-2u_{1\beta}u_{2\beta}\xi_{3-i}\right)\mathcal{N}_{i}\Bigr|\\
\leq C|\mu_a|\|\mathcal{N}_{i}\|_{-\mu_a}(\|\xi_{3-i}\|^2_{-\mu_a}+\|\xi_1\|_{-\mu_a}\|\xi_2\|_{-\mu_a}).
\end{aligned}
\end{equation}

Using $2a=a_0+\beta$ and
\[
\begin{split}
\mu_a=&\int_{\Omega}(|\nabla u_a|^2+V(x)u_a^2-au_a^4)\; dx
\\
=&\int_{\Omega}(|\nabla u_{1\beta}|^2+V(x)u_{1\beta}^2-2au_{1\beta}^4)\; dx+\int_{\Omega}(|\nabla u_{2\beta}|^2+V(x)u_{2\beta}^2-2au_{2\beta}^4)\; dx,
\end{split}
\]
 by direct calculations, we find from \eqref{0-9-8}   that
\begin{align*}
\mu-\mu_a=&2\mu_a\int_{\Omega}u_{1\beta}\eta_1+2a\int_\Omega u_{1\beta}^3\eta_1+\int_\Omega(|\nabla \eta_1|^2+V(x)\eta_1^2)\\
&-a_0\int_\Omega (|u_{1\beta}+\xi_1|^3(u_{1\beta}+\eta_1)-u_{1\beta}^4)
-\beta \int_\Omega (|u_{1\beta}+\xi_1|(u_{2\beta}+\xi_2)^2(u_{1\beta}+\eta_1)-u_{1\beta}^4)\\
&+2\mu_a\int_{\Omega}u_{2\beta}\eta_2+2a\int_\Omega u_{2\beta}^3\eta_2+\int_\Omega(|\nabla \eta_2|^2+V(x)\eta_2^2)\\
&-a_0\int_\Omega (|u_{2\beta}+\xi_2|^3(u_{2\beta}+\eta_2)-u_{2a}^4)
-\beta \int_\Omega (|u_{2\beta}+\xi_2|(u_{1\beta}+\xi_1)^2(u_{2\beta}+\eta_2)-u_{2\beta}^4)
\\
=&2\mu_a\int_{\Omega}(u_{1\beta}\eta_1+u_{2\beta}\eta_2)+\int_\Omega(|\nabla \eta_1|^2+V(x)\eta_1^2)+\int_\Omega(|\nabla \eta_2|^2+V(x)\eta_2^2)\\
&-a_0\int_\Omega (|u_{1\beta}+\xi_1|^3-u_{1\beta}^3)(u_{1\beta}+\eta_1)
-\beta \int_\Omega (|u_{1\beta}+\xi_1|(u_{2\beta}+\xi_2)^2-u_{1\beta}^3)(u_{1\beta}+\eta_1)\\
&-a_0\int_\Omega (|u_{2\beta}+\xi_2|^3-u_{2\beta}^3)(u_{2\beta}+\eta_2)
-\beta \int_\Omega (|u_{2\beta}+\xi_2|(u_{1\beta}+\xi_1)^2-u_{2\beta}^3)(u_{2\beta}+\eta_2).
\end{align*}
Thus,
\begin{equation}\label{5-15}
\begin{aligned}
|\mu-\mu_a|\leq C|\mu_a|\sum\limits_{i=1}^2\left(\|\eta_i\|_{-\mu_a}+\|\xi_i\|_{-\mu_a}\right).
\end{aligned}
\end{equation}
Since $\int_{\Omega}\bigl(|u_{1\beta}+\eta_1|^2+|u_{2\beta}+\eta_2|^2\bigr)=1$ and $\int_{\Omega}(u_{1\beta}^2+u_{2\beta}^2)=1$, we have
\[
2\int_{\Omega}(u_{1\beta}\eta_1+u_{2\beta}\eta_2)=-\int_{\Omega}(\eta_1^2+\eta_2^2).
\]
Then, using \eqref{60-9-8}, $\mathcal{N}(\xi)=-\eta(\xi)+K(\xi)$ and  $u_{1a}=u_{2a}$, we find  that
\begin{align*}
\lambda&-(\mu-\mu_a)=(\beta+a_0)\int_\Omega \big(u_{1\beta}^3 \mathcal{N}_1+u_{2\beta}^3 \mathcal{N}_2\big)+\mu_a\int_{\Omega}(\eta_1^2+\eta_2^2)
\\
&-\int_\Omega(|\nabla \eta_1|^2+V(x)\eta_1^2)-\int_\Omega(|\nabla \eta_2|^2+V(x)\eta_2^2)
\\
&+a_0\int_\Omega (|u_{1\beta}+\xi_1|^3-u_{1\beta}^3-3u_{1\beta}^2\xi_1)u_{1\beta}+a_0\int_\Omega (|u_{1\beta}+\xi_1|^3-u_{1\beta}^3)\eta_1
\\
&+\beta \int_\Omega (|u_{1\beta}+\xi_1|(u_{1\beta}+\xi_2)^2-u_{1\beta}^3-u_{1\beta}^2\xi_1-2
u_{1\beta}^2\xi_2)u_{1\beta}+\beta \int_\Omega (|u_{1\beta}+\xi_1|(u_{1\beta}+\xi_2)^2-u_{1\beta}^3)
\eta_1
\\
&+a_0\int_\Omega (|u_{2\beta}+\xi_2|^3-u_{2\beta}^3-3u_{2\beta}^2\xi_2)u_{2\beta}+a_0\int_\Omega (|u_{2\beta}+\xi_2|^3-u_{2\beta}^3)\eta_2
\\
&+\beta \int_\Omega (|u_{2\beta}+\xi_2|(u_{2\beta}+\xi_1)^2-u_{2\beta}^3-u_{2\beta}^2\xi_2
-2u_{2\beta}^2\xi_1)u_{2\beta}+\beta \int_\Omega (|u_{2\beta}+\xi_2|(u_{2\beta}+\xi_1)^2-u_{2\beta}^3)\eta_2.
\end{align*}
As a result,
\begin{equation}\label{5-16}
\begin{aligned}
|\lambda-(\mu-\mu_a)|\leq C|\mu_a|\sum\limits_{i=1}^2\big(\|\mathcal{N}_{i}\|_{-\mu_a}+\|\eta_i\|_{-\mu_a}^2+\|\xi_i\|_{-\mu_a}^2\big).
\end{aligned}
\end{equation}

 From  $\int_\Omega (u_{1\beta}K_1+u_{2\beta}K_2)=0$, we find \begin{equation}\label{5-17}
\left|\int_\Omega (u_{1\beta}\mathcal{N}_{1}+u_{2\beta}\mathcal{N}_{2})\right|=\left|\int_\Omega (u_{1\beta}\eta_1+u_{2\beta}\eta_2)\right|=\frac{1}{2}\int_{\Omega}(\eta_1^2+\eta_2^2)\leq C\big(\|\eta_1\|_{-\mu_a}^2+\|\eta_2\|_{-\mu_a}^2\big).
\end{equation}
We also have \begin{equation}\label{5-18}
\left|\int_\Omega ( \eta_1\mathcal{N}_{1}+\eta_2\mathcal{N}_{2})\right|\leq C\big(\|\eta_1\|_{-\mu_a}\|\mathcal{N}_1\|_{-\mu_a}+\|\eta_2\|_{-\mu_a}\|\mathcal{N}_2\|_{-\mu_a} \big).
\end{equation}

Combining  \eqref{31-9-8}-\eqref{5-18}  and Lemma \ref{lem5-3},  we conclude   that
\[
\|\mathcal{N}\|_{-\mu_a}^2\leq C(\|\xi\|_{-\mu_a}^2\|\mathcal{N}\|_{-\mu_a}+\|\xi\|_{-\mu_a}^4),
\]
which gives
\[
\|\mathcal{N}\|_{-\mu_a}\leq C\|\xi\|_{-\mu_a}^2.
\]

\end{proof}

\begin{proof}[Proof of Theorem \ref{thm5-1}]

Let
\[
H(\xi,t)=\xi-K(\xi)+t\mathcal{N} (\xi).
\]
We claim  there exists $\epsilon>0$ small enough such that for $\xi\in \partial B_\epsilon^a (0)$ and $t\in [0,1]$, $H(\xi,t)\neq 0$. For contradiction, suppose there exist $t_n\in [0,1]$ and $\xi_n\in \mathcal{H}$ with $\|\xi_n\|_{-\mu_a}\to 0$ such that $H(\xi_n,t_n)=0$. Then by Theorem \ref{thm4-1} and Lemma \ref{lem5-4}, we have
\[
C_a\|\xi_n\|_{-\mu_a}\leq \|\xi_n-K(\xi_n)\|_{-\mu_a}=t_n\|\mathcal{N}\|_{-\mu_a}\leq C\|\xi_n\|_{-\mu_a}^2,
\]
which is impossible.  Then the result  follows from the homotopy invariance of the degree.
\end{proof}


\section{Existence  results for $\beta>\max\{a_1, a_2\}$ }\label{6}

In this section, we prove Theorems~\ref{thm1-2} and \ref{thm1-3}. Recall that
\[
d_{a_1,a_2,\beta}^{\varepsilon}=deg (I-\mathbb{T}_{a_1,a_2,\beta}^\varepsilon,\bar{B}_R\setminus B_\delta,0),
\]
and $d_{a_1,a_2,\beta}=d_{a_1,a_2,\beta}^0$, $\mathbb{T}_{a_1,a_2,\beta}=\mathbb{T}_{a_1,a_2,\beta}^0$.

\begin{proof}[Proof of Theorem \ref{thm1-2} ]

First, by Theorem~\ref{thm1-1} and the homotopy invariance of degree,
to prove Theorem \ref{thm1-2}, it suffices to  prove that the results hold  for $V_1=V_2=V$.

$(i)$   By Theorem~\ref{thm1-1} and the homotopy invariance of degree, it remains to prove
that for $\varepsilon=0$ and small $a_1,a_2,\beta>0$,
\begin{equation}\label{eq6-1}
deg(I-\mathbb{T}_{a_1,a_2,\beta},\bar{B}_R\setminus B_\delta,0)=1.
\end{equation}

Define
\[
H_t(u_1,u_2)=T\big(t(a_1|u_1|^3+\beta u_2^2|u_1|)+(1-t)\lambda_1 e_1, t(a_2|u_2|^3+\beta u_1^2|u_2|)\big)
\]
where $T$ is defined in Section \ref{1},   $\lambda_1$ is  the first eigenvalue of $-\Delta+V$ in $\Omega$ with zero Dirichlet boundary condition, and $e_1$ is  the unique positive eigenfunction corresponding to $\lambda_1$ with $\int_\Omega e_1^2=1$.
Note that  $(I-H_t)(u_1,u_2)=0$ if and only if
\begin{equation}\label{eq6-2}
\left\{\begin{array}{ll}
-\Delta u_1+V(x)u_1=ta_1u_1^3+t\beta u_1u_2^2+(1-t)\lambda_1 e_1+\mu u_1&\hbox{ in }\Omega,\\
-\Delta u_2+V(x)u_2=ta_2u_2^3+t\beta u_2u_1^2+\mu u_2&\hbox{ in }\Omega,\\
0\leq u_1,u_2\in H_0^1(\Omega),
\\
\int_\Omega (u_1^2+u_2^2)=1,
\end{array}\right.
\end{equation}
for some $\mu<\lambda_1$. Testing the first equation  by $u_1$ and the second equation  by $u_2$, we derive from the Gagliardo-Nirenberg inequality that for any $t\in [0,1]$,
\[
\begin{split}
\|(u_1,u_2)\|_{\mathcal{H}}^2=&\int_\Omega \big(|\nabla u_1|^2+|\nabla u_2|^2+V(x)u_1^2+V(x)u_2^2
\big)
\\
=&t\int_\Omega \big( a_1 u_1^4+a_2 u_2^4+2\beta  u_1^2u_2^2\big)+(1-t)\lambda_1 \int_\Omega u_1e_1+\mu
\\
\leq& C\max\{a_1,a_2,\beta\} \int_\Omega \big(  u_1^4+u_2^4\big)+(1-t)\lambda_1 \int_\Omega u_1e_1+\lambda_1
\\
\leq& C\max\{a_1,a_2,\beta\}\left(  \int_\Omega ( |\nabla u_1|^2+ |\nabla u_2|^2)\right)  \left( \int_\Omega (
u_1^2+ u_2^2)
\right)+\lambda_1\Big(\int_\Omega  u_1^2\Big)^{\frac{1}{2}}\Big(\int_\Omega  e_1^2\Big)^{\frac{1}{2}}+\lambda_1
\\
\leq& C \max\{a_1,a_2,\beta\} \int_\Omega  (|\nabla u_1|^2+|\nabla u_2|^2)+C
\\
\leq & C \max\{a_1,a_2,\beta\}  \|(u_1,u_2)\|_{\mathcal{H}}^2+C,
\end{split}
\]
which shows that for sufficiently small $a_1,a_2,\beta >0$,  $\|(u_1,u_2)\|_{\mathcal{H}}\leq C$. Then, for large $R>0$ and small $\delta>0$, $I-H_t\neq 0$ on $\partial (\bar{B}_R\setminus B_\delta)$. By the homotopy invariance of degree, we have
\[
\begin{split}
deg(I-\mathbb{T}_{a_1,a_2,\beta},\bar{B}_R\setminus B_\delta,0)=&deg(I-H_1,\bar{B}_R\setminus B_\delta,0)
=deg(I-H_t,\bar{B}_R\setminus B_\delta,0)
\\
=&deg(I-H_0,\bar{B}_R\setminus B_\delta,0)
=deg(I-(e_1,0),\bar{B}_R\setminus B_\delta,0)
\\
=&1.
\end{split}
\]
Thus, \eqref{eq6-1} is proved.

 $(ii)$
By Theorem~\ref{thm1-1} and the homotopy invariance of degree,  to  compute the degree $d_{a_1,a_2,\beta}^\varepsilon$ for any $a_1,a_2\in (0,a^*)$, $\varepsilon\in (0,\varepsilon_0)$ and $\beta \in (\beta_1^*,\beta_2^*)$,
we may take   $\varepsilon=0$ and $a_1=a_2=a_0$  for some $\frac{a^*}{2}<a_0<a^*$. Moreover, we can further assume that  $V$ satisfies $(H_1)-(H_2)$.

Let  $P_1,\cdots,P_m$ be all critical points of $V$. By Theorem~B  in Section~4, there exist exactly $m$ peak solutions $u_{a}^{(1)},\cdots, u_{a}^{(m)}$ of \eqref{eq4-4} as $a\to a^*$,  where $a=(\beta+a_0)/2$.  From the discussion at the beginning of Section~4, all the blow-up
solutions for $\beta  \to \beta_1^*$ are given by

\[
(u_{1\beta}^{(j)},u_{2\beta}^{(j)})=\bigl(\frac{\sqrt{2}}{2} u_{a}^{(j)}, \frac{\sqrt{2}}{2} u_{a}^{(j)}
\bigr),
\quad j=1,\cdots,m.
\]
 In what follows,  we assume that $\beta>a_0$ is close to $\beta_1^*$ and $a=(\beta+a_0)/2$.

Denote
\[
B_\epsilon ^a(P_j)=\big\{(u_1,u_2)\in\mathcal{H}:\|(u_1,u_2)-(u_{1\beta}^{(j)},u_{2\beta}^{(j)})\|_{-\mu_a}<\epsilon\big\},\;\; j=1,\cdots, m.
\]

By \eqref{1-7-8} in Theorem~B, we deduce that all zeros of
$I-\mathbb{T}_{a_0,a_0,\beta_1^*}$ are uniformly bounded. Moreover,
 for $\beta_1^*<\beta<\beta_1^*+\theta$, namely $a^*<a<a^*+\theta/2$,  with $\theta>0$ small,
the degree can be counted by
\begin{equation}\label{6-3}
\begin{aligned}
d^+:=&deg(I-\mathbb{T}_{a_0,a_0,\beta}, \bar{B}_{C_a}\setminus B_\delta,0)\\
=&\sum\limits_{\Delta V(P_j)<0}deg (I-\mathbb{T}_{a_0,a_0,\beta},B_\epsilon^a(P_j),0)+deg(I-\mathbb{T}_{a_0,a_0,\beta},\bar{B}_{C^*}\setminus B_\delta,0),
\end{aligned}
\end{equation}
where $C^*$ is a constant independent of $\beta$.
Also for $\beta_1^*-\theta<\beta<\beta_1^*$, namely $a^*-\theta/2<a<a^*$, we have
\begin{equation}\label{6-4}
\begin{aligned}
d^-:=\sum\limits_{\Delta V(P_j)>0}deg (I-\mathbb{T}_{a_0,a_0,\beta},B_\epsilon^a(P_j),0)+deg(I-\mathbb{T}_{a_0,a_0,\beta},\bar{B}_{C^*}\setminus B_\delta,0).
\end{aligned}
\end{equation}

By the homotopy invariance of degree,  the degree of all bounded solutions is unchanged as  $\beta$ crosses  $\beta_1^*$. This implies that $deg(I-\mathbb{T}_{a_0,a_0,\beta},\bar{B}_{C^*}\setminus B_\delta,0)$ is equal to the same  integer for all $\beta_1^*-\theta<\beta<\beta_1^*+\theta$. Thus by
Corollary \ref{cor4-2} and Theorem \ref{thm5-1},
\begin{equation}\label{6-5}
\begin{aligned}
d^+-d^-=&\sum\limits_{\Delta V(P_j)<0}deg (I-\mathbb{T}_{a_0,a_0,\beta},B_\epsilon^a(P_j),0)-\sum\limits_{\Delta V(P_j)>0}deg (I-\mathbb{T}_{a_0,a_0,\beta},B_\epsilon^a(P_j),0)\\
=&\sum\limits_{\Delta V(P_j)<0}(-1)^{1+ind~P_j}-\sum\limits_{\Delta V(P_j)>0}(-1)^{ind~P_j}\\
=&-\sum\limits_{j=1}^m(-1)^{ind~P_j},
\end{aligned}
\end{equation}
where $ind~P_j$ is the number of negative eigenvalues of the Hessian of $V$ at $P_j$.

On the other hand, we have (see the proof of Theorem 1.2 in  \cite{cyy})
\[
\sum\limits_{j=1}^m(-1)^{ind~P_j}=1-k,
\]
which together with \eqref{6-5} gives  $d^+-d^-=k-1$. By (i), we have $d^-=1$.  As a result, $d^+=k$ and $(ii)$  follows.

 $(iii)$  By Theorem~\ref{thm1-1} and the homotopy invariance of degree,  to  compute the degree $d_{a_1,a_2,\beta}^\varepsilon$ for any $a_1\in (0,a^*)$, $a_2\in (a^*, 2a^*)$, $\varepsilon\in (0,\varepsilon_0)$ and $\beta \in [\min \{a_1,a_2\},\beta_2^*)$,
we may take  $\beta=a_1$.
Define
\[
H_t^\varepsilon(u_1,u_2)=T\big(t(a_1|u_1|^3+\varepsilon |u_2|+ a_1 u_2^2|u_1|), a_2|u_2|^3+ta_1 u_1^2|u_2|+\varepsilon |u_1|\big)
\]
Since $ta_1\in [0,a^*)$ for any $t\in [0,1]$,  similar to the proof of $(b)$
in Theorem~\ref{thm1-1}, we can
 show that there exists a constant $R>0$, independent of $t$, such that each solution $(u_1,u_2)$ of the following equation satisfies $\|(u_1,u_2)\|_{\mathcal{H}}<R$,
\[
\left\{
\begin{array}{ll}
-\Delta u_{1}+V(x)u_{1}=t(a_1|u_1|^3+\varepsilon |u_2|+ a_1 u_2^2|u_1|)+\mu u_{1}& \hbox{ in }\Omega,\\
-\Delta u_{2}+V(x)u_{2}=a_2|u_2|^3+ta_1 u_1^2|u_2|+\varepsilon |u_1|+\mu u_{2}&\hbox{ in }\Omega,\\
  u_{1},u_{2}\geq 0 &\hbox{ in }\Omega,
 \\
 u_1=u_2=0 &\hbox{ on }\partial\Omega,
 \\
 \int_{\Omega}(u_{1}^2+u_{2}^2)=1.
 \end{array}\right.
\]
Then the homotopy invariance of degree yields
\begin{equation}\label{6-6Y}
\begin{split}
deg(I-\mathbb{T}_{a_1,a_2,a_2}^\varepsilon,\bar{B}_R\setminus B_\delta,0)=&deg(I-H_1^\varepsilon,\bar{B}_R\setminus B_\delta,0)
\\
=&deg(I-H_0^\varepsilon,\bar{B}_R\setminus B_\delta,0).
\end{split}
\end{equation}
Note that $H_0^\varepsilon(u_1,u_2)=T\big(0,a_2|u_2|^3+\varepsilon |u_1|\big)$. Denote $(\xi_1,\xi_2)=H_0^\varepsilon(u_1,u_2)$. Then
\begin{equation}\label{6-7Y}
\left\{
\begin{array}{ll}
-\Delta \xi_1+V(x)\xi_1=\mu \xi_1& \hbox{ in }\Omega,\\
-\Delta \xi_2+V(x)\xi_2=a_2|u_2|^3+\varepsilon |u_1|+\mu \xi_2&\hbox{ in }\Omega,\\
 \xi_1,\xi_2\geq 0 &\hbox{ in }\Omega,
 \\
 \xi_1=\xi_2=0 &\hbox{ on }\partial\Omega,
 \\
 \int_{\Omega}(\xi_{1}^2+\xi_{2}^2)=1.
 \end{array}\right.
\end{equation}
If $(u_1,u_2)\neq (0,0)$, then $a_{2}|u_{2}|^3+\varepsilon |u_1|\neq 0$. By Lemma \ref{lem3-2}, $\mu<\lambda_1(V)$, and hence the first equation in \eqref{6-7Y} gives $\xi_1=0$. Thus, $H^\varepsilon_0$ is a map from $\mathcal{H}\setminus B_\delta$ to $\{0\}\times H_2(\Omega)\setminus B_\delta$.
Using the Leray-Schauder degree reduction theorem (see for
example Theorem~8.17 in \cite{D}), we have
\[
deg(I-H_0^\varepsilon,\bar{B}_R\setminus B_\delta,0)=deg\big( (I-H_0^\varepsilon)|_{H_1(\Omega)},(\bar{B}_R\setminus B_\delta)\cap H_2(\Omega),0\big),
\]
where
\[
(\bar{B}_R\setminus B_\delta)\cap H_2(\Omega):=\{u_2: \; (u_1,u_2)\in \bar{B}_R\setminus B_\delta\}.
\]

Note that $(I-H_0^\varepsilon)|_{H_2(\Omega)}=(I-H_0^0)|_{H_2(\Omega)}$,  and
$(I-H_0^0)|_{H_2(\Omega)}(u_1)=0$ if and only if $u_2$ is a positive solution of
\[
\begin{cases}
-\Delta u_2+V(x)u_2=a_2u_2^3+\mu u_1\;\;\;\text{in}\;\;\Omega,
\\
u_2=0\;\;\; \text{on}\;\;\Omega,
\\
\int_\Omega u_2^2=1.
\end{cases}
\]
By \cite[Theorem 1.2]{cyy},  we have
\begin{equation}\label{6-8Y}
deg(I-H_0^\varepsilon,\bar{B}_R\setminus B_\delta,0)=deg\big( (I-H_0^0)|_{H_2(\Omega)},(\bar{B}_R\setminus B_\delta)\cap H_2(\Omega),0\big)=k.
\end{equation}
Combining \eqref{6-6Y} and \eqref{6-8Y}, we  obtain that $d_{a_1,a_2,a_2}^\varepsilon=k$.
\end{proof}

We now prove Theorem \ref{thm1-3}.

\begin{proof}[Proof of Theorem \ref{thm1-3}]
  Under the assumptions  in Theorem \ref{thm1-3},  Theorem  \ref{thm1-2} gives
   $d_{a_1,a_2,\beta}^\varepsilon\neq 0$ for $\varepsilon\in (0,\varepsilon_0)$.
 Then  there exists a solution $(u_{1\varepsilon},u_{2\varepsilon};\mu_\varepsilon)$ of \eqref{1-9}. It is clear that $u_{1\varepsilon}\not\equiv 0$ and $u_{2\varepsilon}\not\equiv 0$.

By Theorem \ref{thm1-1}, $u_{i\varepsilon}$ is bounded in $H_i(\Omega)$ and $L^\infty(\Omega)$, $i=1,2$,  and $\mu_\varepsilon$ is bounded.  So up to a subsequence, we can assume that as $\varepsilon\to0$,  $u_{i\varepsilon}\rightharpoonup u_i$ in $H_i(\Omega)$,  and $u_{i\varepsilon}\to u_i$ in $L^p(\Omega)$  for $p\geq 2$, and $\mu_\varepsilon \to \mu$. It is clear that $(u_1,u_2;\mu)$ is a solution of \eqref{1-1}-\eqref{1-2}. By Proposition~\ref{p1-5-8}, we have
\begin{equation}\label{6-15}
\|u_{1}\|_{L^2(\Omega)}\geq \sigma,\quad \|u_{2}\|_{L^2(\Omega)}\geq \sigma.
\end{equation}
Hence, $(u_1,u_2;\mu)$ is a nontrivial solution of \eqref{1-1}-\eqref{1-2}.
\end{proof}

\section{Existence  results for $0\le \beta<\min\{a_1, a_2\}$ }\label{100}

This section is devoted to proving Theorem \ref{thm1-4}. Throughout, we  assume that $a_1,a_2,V_1,V_2$ satisfy the assumptions in Theorem \ref{thm1-4}. We begin by showing that there exists $\sigma>0$ such  that the following degree for all nontrivial solutions is well-defined:
\begin{equation}\label{6-17}
\begin{split}
d_{a_1,a_2,\beta}^{non-trivial}:=&\text{deg}(I-\mathbb{T}_{a_1,a_2,\beta}, 
\bar{B}_{R}\setminus (B_\delta \cup W_1\cup W_2), 0),
\end{split}
\end{equation}
where
\begin{equation}\label{6-18}
W_1=\{  (u_1,u_2)\in  \bar{B}_{R}\setminus B_\delta :\; \|u_2\|_{L^2(\Omega)}<\sigma  \},\;\;\; W_2=\{  (u_1,u_2)\in  \bar{B}_{R}\setminus B_\delta :\; \|u_1\|_{L^2(\Omega)}<\sigma  \},
\end{equation}
and $\sigma\in (0, \delta)$ is a  constant.

 Firstly, we know that though $\mathbb{T}_{a_1,a_2,\beta}$ may not
be well defined in $\bar{B}_{R}\setminus B_\delta$, it is well defined in
$\bar{B}_{R}\setminus (B_\delta \cup W_1\cup W_2)$. Secondly, 
we need to prove that $I-\mathbb{T}_{a_1,a_2,\beta}\ne 0$
on the boundary of $ W_1\cup W_2$.
To this end, we will show that
 if
$-\tilde c_2<\inf \{V_1-V_2\} \leq \sup \{V_1-V_2\} <\tilde  c_1 $ for some constants $\tilde c_1,\;
\tilde c_2>0$, there exists
a constant $\sigma>0$, such that any nontrivial solution $(u_1,u_2)$  of \eqref{1-1}--\eqref{1-2}
satisfies
\[
\|u_1\|_{L^2(\Omega)}>\sigma,\quad \|u_2\|_{L^2(\Omega)}>\sigma.
\]

Let $\lambda_1(V)$ be the first eigenvalue of $-\Delta +V$. To define the constants $\tilde c_1,\;
\tilde c_2>0$, we consider the following problem

\begin{equation}\label{6-20}
\begin{cases}
-\Delta u+V(x)u=au^3+\mu u\;\;\;\text{in}\;\;\Omega,
\\
u>0\;\;\; \text{in} \;\; \Omega; \quad u=0\;\;\; \text{on}\;\;\partial\Omega,
\\
\int_\Omega u^2=1.
\end{cases}
\end{equation}
Define
\begin{equation}\label{20-14-8}
\begin{split}
\mu^*(a, V)=\sup \Big\{\mu:\; \text{$(u;\mu)$ solves \eqref{6-20}}      \Big\}.
\end{split}
\end{equation}
It follows from Theorem~A in the introduction that if  \eqref{6-20}
has a solution, then there exists
$C>0$ such that for any  solution $u$ of \eqref{6-20}, it holds
$\int_{\Omega}(|\nabla u|^2 +V(x)u^2)\le C$ and $\|u\|_{L^\infty(\Omega)}\le C$
and $|\mu|\le C$.
 From

\[
\mu=\int_{\Omega}(|\nabla u|^2 +V(x)u^2) - a\int_{\Omega} u^4,
\]
we can prove that $\mu^*(a, V)$ is attained and $\mu^*(a, V)\in (-\infty, \lambda_1)$.
On the other hand,  if \eqref{6-20} has no solution,  we set $\mu^*(a, V)=-\infty$.

If $a\in (0, a^*-\delta]$, \eqref{6-20} always has a solution  and $\mu^*(a, V)<\lambda_1(V)$.
We define
\begin{equation}\label{n6-22}
c_1(a_1):=\frac{a_1-\beta}{a_1} ( \lambda_{1}(V_1)-\mu^*(a_1, V_1)).
\end{equation}
Then $c_1(a_1)>0$.

Similarly, for $a_2\in [m a^*+\delta, (m+1)a^*-\delta]$, we define

\begin{equation}\label{6-22}
 c_2(a_2):=\frac{a_2-\beta}{a_2} \inf_{V\in\mathcal E}( \lambda_{1}(V)-\mu^*(a_2,V)).
\end{equation}
Then $c_2(a_2)>0$.

Now we consider the following problem

\begin{equation}\label{6-19}
\left\{
\begin{array}{ll}
-\Delta u_{1}+V_{1}(x)u_{1}=a_{1}u_{1}^3+\beta u_{1}u_{2}^2+\mu u_{1}& \hbox{ in }\Omega,\\
-\Delta u_{2}+V_{2s}(x)u_{2}=a_{2}u_{2}^3+\beta u_{2}u_{1}^2+\mu u_{2}&\hbox{ in }\Omega,\\
  u_{1},u_{2}\geq 0 &\hbox{ in }\Omega,
 \\
 u_1=u_2=0 &\hbox{ on }\partial\Omega,
 \\
 \int_{\Omega}(u_{1}^2+u_{2}^2)=1,
 \end{array}\right.
\end{equation}
where  $V_{2s}=(1-s)V_2+sV_1$.

\begin{lemma}\label{lem6-2} Suppose that $-c_2(a_2)<\inf \{V_1-V_2\} \leq \sup \{V_1-V_2\} <c_1(a_1)$. Then there exists $\sigma>0$ such that for any  $s\in [0,1]$ and any nontrivial solution $(u_1,u_2;\mu)$ of \eqref{6-19},  it holds
\[
\|u_1\|_{L^2(\Omega)}>\sigma,\quad \|u_2\|_{L^2(\Omega)}>\sigma.
\]
\end{lemma}

\begin{proof}

We argue by contradiction.
Suppose  that there exists a family of nontrivial solutions $(u_{1n},u_{2n};\mu_n)$ of \eqref{6-19} with $s=s_n\in [0,1]$ such that as $n\to \infty$,
\[
\int_\Omega u_{1n}^2\to 0 \;\;\text{or}\;\; \int_\Omega u_{2n}^2\to 0,\;\;\; s_n\to s_0\in[0,1].
\]
First, we assume that $ \int_\Omega u_{2n}^2\to 0$. So,
$\int_\Omega u_{1n}^2\to 1$. Taking $\varepsilon_n= \|u_{2n}\|_{L^2(\Omega)}$, we have  $\varepsilon_n>0$ and $\varepsilon_n\to0$.  By Theorem \ref{thm1-1},
\begin{equation}\label{6-23}
\|u_{1n}\|_{H_1(\Omega)},\;\; \|u_{1n}\|_{L^\infty(\Omega)},\;\; |\mu_n|\leq C,
\end{equation}
where $C>0$ is a constant independent of $n$.

Let $v_n=u_{2n}/\varepsilon_n$. Then $\|v_n\|_{L^2(\Omega)}=1$. Moreover, $(u_{1n},v_n)$ satisfies
\begin{equation}\label{6-24}
\left\{
\begin{array}{ll}
-\Delta u_{1n}+V_{1}(x)u_{1n}=a_{1}u_{1n}^3+\beta \varepsilon_n^2 u_{1n}v_{n}^2+\mu_n u_{1n}& \hbox{ in }\Omega,\\
-\Delta v_n+V_{2s_n}(x)v_n=a_{2}\varepsilon_n^2 v_n^3+\beta  v_n u_{1n}^2+\mu_n v_{n}&\hbox{ in }\Omega.
 \end{array}\right.
\end{equation}
Multiplying the second equation of \eqref{6-24} by $v_n$ and integrating by parts, we get from \eqref{6-23} and the Gagliardo-Nirenberg inequality that $\{v_n\}$ is bounded in $H_2(\Omega)$. Then we may assume that up to a subsequence,
\[
\begin{split}
&u_{1n}\rightharpoonup u\;\;\text{in} \;\; H_1(\Omega), \quad u_{1n}\to u\;\;\text{in} \;\; L^2(\Omega),
\\
&v_{n}\rightharpoonup v\;\;\text{in} \;\; H_2(\Omega), \quad v_{n}\to v\;\;\text{in} \;\; L^2(\Omega),
\\
&\mu_n\to \mu,\quad s_n\to s_0.
\end{split}
\]
Moreover, $(u,v)$ is a solution of
\[
\begin{cases}
-\Delta u+V_{1}(x)u=a_1u^3+\mu u\;\;\;\text{in}\;\;\Omega,
\\
-\Delta v+V_{2s_0}(x)v=\beta vu^2+\mu v\;\;\;\text{in}\;\;\Omega,
\\
u,v>0\;\;\; \text{in} \;\; \Omega; \quad u,v=0\;\;\; \text{on}\;\;\partial\Omega,
\\
\int_\Omega u^2=\int_\Omega v^2=1,
\end{cases}
\]
As a result,
\[
\begin{split}
\mu=\int_{\Omega}\big( |\nabla u|^2+V_{1}(x)u^2-a_1u^4\big)=\int_{\Omega}\big( |\nabla v|^2+V_{2s_0}(x)v^2-\beta u^2 v^2\big).
\end{split}
\]
Since $u>0$, we get
\[
\begin{split}
\mu
=&\inf_{w\in H_1(\Omega),\; \|w\|_{L^2(\Omega)}=1} \int_{\Omega}\big( |\nabla w|^2+V_{1}(x)w^2-a_1u^2w^2\big)
\\
\leq&  \int_{\Omega}\big( |\nabla v|^2+V_{1}(x)v^2-a_1u^2v^2\big)
\\
=&\mu  +\int_{\Omega} (V_1-V_{2s_0}) v^2+(\beta-a_1)\int_{\Omega} u^2v^2,
\end{split}
\]
which gives
\begin{equation}\label{Y1}
(a_1-\beta)\int_{\Omega} u^2v^2\leq \int_{\Omega} (V_1-V_{2s_0}) v^2=(1-s_0)\int_{\Omega} (V_1-V_2) v^2
\end{equation}
Moreover,
\begin{equation}\label{Y2}
\begin{split}
\beta \int_{\Omega} u^2v^2=&\int_{\Omega}\big( |\nabla v|^2+V_{2s_0}(x)v^2\big)-\mu
=\int_{\Omega}\big( |\nabla v|^2+V_{1}v^2\big)+(1-s_0)\int_{\Omega} (V_2-V_1) v^2-\mu
\\
\geq &\lambda_1(V_1)\|v\|^2_{L^2(\mathbb{R}^2)}-\mu+(1-s_0)\int_{\Omega} (V_2-V_1) v^2
\\
\geq& \lambda_1(V_1)-\mu^*(a_1, V_1)+(1-s_0)\int_{\Omega} (V_2-V_1) v^2.
\end{split}
\end{equation}
By \eqref{Y1} and \eqref{Y2}, we get
\begin{equation}\label{Y3}
(1-s_0)\int_{\Omega} (V_1-V_2) v^2\geq \frac{a_1-\beta}{a_1}(\lambda_1(V_1)-\mu^*(a_1, V_1))= c_1(a_1),
\end{equation}
thus
\[
\sup\limits_{x\in\Omega}\{V_1-V_2\}\geq c_1(a_1).
\]
 This is a contradiction to $\sup_{x\in \Omega}\{V_1-V_{2}\}< c_1(a_1)$.

We can use a similar argument  to show that $\int_\Omega u_{1n}^2\to 0$  cannot occur either.
\end{proof}

By Lemma \ref{lem6-2}, the degree $d_{a_1,a_2,\beta}^{non-trivial}$ in \eqref{6-17} is well-defined. Furthermore, using Lemma \ref{lem6-2} and
the homotopy invariance of the degree,  we can  take $V_2=V_1=V$ to compute it.

Now, in view of $V_2=V_1=V$, $\mathbb{T}_{a_1,a_2,\beta}$ is well defined in 
$ W_1\cup W_2$.
By Lemma \ref{lem6-2} again, the following  degrees for semi-trivial solutions are well-defined:
\begin{equation}\label{eq6-25}
d_{a_1,a_2,\beta}^{1,semi-trivial}:=\text{deg}(I-\mathbb{T}_{a_1,a_2,\beta}, W_1, 0),\quad d_{a_1,a_2,\beta}^{2,semi-trivial}:=\text{deg}(I-\mathbb{T}_{a_1,a_2,\beta}, W_2, 0).
\end{equation}

Let  $d_{a_1,a_2,\beta}$  be the total degree. Then, we have
\begin{equation}\label{eq6-26}
d_{a_1,a_2,\beta}=d_{a_1,a_2,\beta}^{non-trivial}+d_{a_1,a_2,\beta}^{1,semi-trivial}+d_{a_1,a_2,\beta}^{2,semi-trivial}.
\end{equation}
From \eqref{eq6-26}, to compute $d_{a_1,a_2,\beta}^{non-trivial}$, it suffices to compute $d_{a_1,a_2,\beta}^{1,semi-trivial}, d_{a_1,a_2,\beta}^{2,semi-trivial}$ and $d_{a_1,a_2,\beta}$.

For the calculation of $d_{a_1,a_2,\beta}^{1,semi-trivial}$, we  consider the following system
\begin{equation}\label{6-27}
\left\{
\begin{array}{ll}
-\Delta u_{1}+V(x)u_{1}=a_{1}u_{1}^3+\beta h(t)u_{1}u_{2}^2+\mu u_{1}& \hbox{ in }\Omega,\\
-\Delta u_{2}+V(x)u_{2}=a_{2}g(t)u_{2}^3+\beta h(t) u_{2}u_{1}^2+\mu u_{2}&\hbox{ in }\Omega,\\
  u_{1},u_{2}\geq 0 &\hbox{ in }\Omega,
 \\
 u_1=u_2=0 &\hbox{ on }\partial\Omega,
 \\
 \int_{\Omega}(u_{1}^2+u_{2}^2)=1,
 \end{array}\right.
\end{equation}
where
\begin{equation}\label{6-28}
h(t)=\max\{2t-1,0\},\;\;\; g(t)=\min\{1,2t\}.
\end{equation}

\begin{lemma}\label{lem6-3} There exists $\sigma>0$ such that for any $t\in [0,1]$ and any nontrivial solution $(u_1,u_2;\mu)$ of \eqref{6-27}, we have
\[
\|u_2\|_{L^2(\Omega)}>\sigma.
\]
\end{lemma}

\begin{proof}We suppose for contradiction that there exists a family of nontrivial solutions $(u_{1n},u_{2n};\mu_n)$ of \eqref{6-27} with $t=t_n\in [0,1]$ such that as $n\to \infty$,
\[
\int_\Omega u_{1n}^2\to 1,\;\;\; \int_\Omega u_{2n}^2\to 0,\;\;\; t_n\to t_0\in[0,1].
\]
Taking $\varepsilon_n= \|u_{2n}\|_{L^2(\Omega)}$, then $\varepsilon_n>0$ and $\varepsilon_n\to0$.

Note that $0\leq \beta h(t_n)\leq \beta <\min \{a_1,a_2\}$. Moreover,
if $t_n\geq \frac{1}{2}$, then we have $a_2g(t_n)=a_2$. By Theorem \ref{thm1-1},
\begin{equation}\label{6-29}
\|u_{1n}\|_{H_1(\Omega)},\;\; \|u_{1n}\|_{L^\infty(\Omega)},\;\; |\mu_n|\leq C
\end{equation}
where $C>0$ is a constant independent of $n$. If $t_n<\frac{1}{2}$, then $\beta h(t_n)=0$. By Theorem~A
in the introduction, \eqref{6-29} holds.

Let $v_n=u_{2n}/\varepsilon_n$. Then $\|v_n\|_{L^2(\Omega)}=1$. Moreover, $(u_{1n},v_n)$ satisfies
\begin{equation}\label{6-30}
\left\{
\begin{array}{ll}
-\Delta u_{1n}+V(x)u_{1n}=a_{1}u_{1n}^3+\beta \varepsilon_n^2 h(t_n)u_{1n}v_{n}^2+\mu_n u_{1n}& \hbox{ in }\Omega,\\
-\Delta v_n+V(x)v_n=a_{2}\varepsilon_n^2 g(t_n)v_n^3+\beta h(t_n) v_n u_{1n}^2+\mu_n v_{n}&\hbox{ in }\Omega.
 \end{array}\right.
\end{equation}
Multiplying the second equation of \eqref{6-30} by $v_n$ and integrating by parts, we obtain from \eqref{6-29} and the Gagliardo-Nirenberg inequality that $\{v_n\}$ is bounded in $H_2(\Omega)$. Then we may assume that up to a subsequence,
\[
\begin{split}
&u_{1n}\rightharpoonup u\;\;\text{in} \;\; H_1(\Omega), \quad u_{1n}\to u\;\;\text{in} \;\; L^2(\Omega),
\\
&v_{n}\rightharpoonup v\;\;\text{in} \;\; H_2(\Omega), \quad v_{n}\to v\;\;\text{in} \;\; L^2(\Omega),
\\
&\mu_n\to \mu, \quad t_n\to t_0.
\end{split}
\]
Then,  $(u,v)$ is a positive solution of
\[
\begin{cases}
-\Delta u+V(x)u=a_1u^3+\mu u\;\;\;\text{in}\;\;\Omega,
\\
-\Delta v+V(x)v=\beta h(t_0)vu^2+\mu v\;\;\;\text{in}\;\;\Omega,
\\
u,v=0\;\;\; \text{on}\;\;\partial\Omega,
\\
\int_\Omega u^2=\int_\Omega v^2=1.
\end{cases}
\]
Hence
\[
\begin{split}
\mu=&\int_{\Omega}\big( |\nabla u|^2+V(x)u^2-a_1u^4\big)
\\
=&\inf_{w\in H_1(\Omega),\; \|w\|_{L^2(\Omega)}=1} \int_{\Omega}\big( |\nabla w|^2+V(x)w^2-a_1u^2w^2\big)
\\
\leq&  \int_{\Omega}\big( |\nabla v|^2+V(x)v^2-a_1u^2v^2\big)
\\
=&\mu  +(\beta h(t_0)-a_1)\int_{\Omega} u^2v^2,
\end{split}
\]
which gives $\beta h(t_0)-a_1\geq 0$. But we have $\beta h(t_0)\leq \beta < a_1$, contradiction.
\end{proof}

For the calculation of $d_{a_1,a_2,\beta}^{2,semi-trivial}$, we  consider the following problem
\begin{equation}\label{6-31}
\left\{
\begin{array}{ll}
-\Delta u_{1}+V(x)u_{1}=a_{1}g(t)u_{1}^3+\beta h(t)u_{1}u_{2}^2+\mu u_{1}& \hbox{ in }\Omega,\\
-\Delta u_{2}+V(x)u_{2}=a_{2}u_{2}^3+\beta h(t) u_{2}u_{1}^2+\mu u_{2}&\hbox{ in }\Omega,\\
  u_{1},u_{2}\geq 0 &\hbox{ in }\Omega,
 \\
 u_1=u_2=0 &\hbox{ on }\partial\Omega,
 \\
 \int_{\Omega}(u_{1}^2+u_{2}^2)=1,
 \end{array}\right.
\end{equation}
where $h(t),g(t)$ are defined in \eqref{6-28}.

\begin{lemma}\label{lem6-4} There exists $\sigma>0$ such that for any $t\in [0,1]$ and any nontrivial solution $(u_1,u_2;\mu)$ of \eqref{6-31}, we have
\[
\|u_1\|_{L^2(\Omega)}>\sigma.
\]
\end{lemma}

\begin{proof}

Since the proof is
 similar to that in Lemma \ref{lem6-3}, we omit it.

\end{proof}

Now we are able to prove Theorem \ref{thm1-4}.

\begin{proof}[Proof of Theorem \ref{thm1-4}] We only consider the case $a_1\in (0,a^*)$ and $a_2\in ((m-1)a^*,ma^*)$.  By Lemma \ref{lem6-2} and  the homotopy invariance of degree, we further assume that $V_1=V_2=V$.

First, we define
\[
\tilde{H}_t(u_1,u_2)=T\big(a_{1}|u_{1}|^3+\beta h(t)|u_{1}|u_{2}^2,  a_{2}g(t)|u_{2}|^3+\beta h(t) |u_{2}|u_{1}^2\big).
\]
By Lemma \ref{lem6-3} and  the homotopy invariance of degree, we have
\begin{equation}\label{6-32}
\begin{split}
d_{a_1,a_2,\beta}^{1,semi-trivial}=&\text{deg}(I-\mathbb{T}_{a_1,a_2,\beta}, W_1, 0)
\\
=&\text{deg}(I-\tilde{H}_1, W_1, 0)
\\
=&\text{deg}(I-\tilde{H}_0, W_1, 0)
=\text{deg}\big((I-\tilde{H}_0)|_{H_1(\Omega)}, (\bar{B}_{R}\setminus B_\delta)\cap H_1(\Omega), 0\big)
=1.
\end{split}
\end{equation}
In this last equality, we have used (1.16) in  \cite{cyy}.

Second, we define
\[
H_t(u_1,u_2)=T\big(a_{1}g(t)|u_{1}|^3+\beta h(t)|u_{1}|u_{2}^2,  a_{2}|u_{2}|^3+\beta h(t) |u_{2}|u_{1}^2\big).
\]
By Lemma \ref{lem6-4} and  the homotopy invariance of degree,
\begin{equation}\label{6-33}
\begin{split}
d_{a_1,a_2,\beta}^{2,semi-trivial}=&\text{deg}(I-\mathbb{T}_{a_1,a_2,\beta}, W_2, 0)
\\
=&\text{deg}(I-H_1, W_2, 0)
\\
=&\text{deg}(I-H_0, W_2, 0)
\\
=&\text{deg}((I-H_0)|_{H_2(\Omega)}, (\bar{B}_{R}\setminus B_\delta)\cap H_2(\Omega), 0).
\end{split}
\end{equation}

On the other hand, we define
\[
H_t^\varepsilon(u_1,u_2)=T\big(ta_1|u_1|^3+t \beta u_2^2|u_1|, a_2|u_2|^3+(1-t)\varepsilon |u_1| + t \beta u_1^2|u_2|)\big),
\]
where $\varepsilon>0$ is a small constant.
Since $ta_1\in [0,a^*)$ and  $t\beta\leq \min\{a_1, a_2\}$ for any $t\in [0,1]$, we can use a similar blow-up argument  as in Theorem \ref{thm1-1} to show that there exists a constant $R>0$, independent of $t$, such that each solution $(u_1,u_2)$ of the following equation satisfies $\|(u_1,u_2)\|_{\mathcal{H}}<R$,
\[
\left\{
\begin{array}{ll}
-\Delta u_{1}+V(x)u_{1}=ta_{1}u_{1}^3+t\beta u_{1}u_{2}^2+\mu u_{1}& \hbox{ in }\Omega,\\
-\Delta u_{2}+V(x)u_{2}=a_{2}u_{2}^3+t\beta u_{2}u_{1}^2+(1-t)\varepsilon u_1+\mu u_{2}&\hbox{ in }\Omega,\\
  u_{1},u_{2}\geq 0 &\hbox{ in }\Omega,
 \\
 u_1=u_2=0 &\hbox{ on }\partial\Omega,
 \\
 \int_{\Omega}(u_{1}^2+u_{2}^2)=1.
 \end{array}\right.
\]
The homotopy invariance of degree yields
\begin{equation}\label{6-34}
\begin{split}
d_{a_1,a_2,\beta}=&\text{deg}(I-\mathbb{T}_{a_1,a_2,\beta},\bar{B}_R\setminus B_\delta,0)
\\
=&\text{deg}(I-H_1^\varepsilon,\bar{B}_R\setminus B_\delta,0)
=\text{deg}(I-H_0^\varepsilon,\bar{B}_R\setminus B_\delta,0).
\end{split}
\end{equation}
Note that $H_0^\varepsilon(u_1,u_2)=T\big(0,a_2|u_2|^3+\varepsilon |u_1|\big)$. Let $(\xi_1,\xi_2)=H_0^\varepsilon(u_1,u_2)$. Then
\begin{equation}\label{6-35}
\left\{
\begin{array}{ll}
-\Delta \xi_1+V(x)\xi_1=\mu \xi_1&\hbox{ in }\Omega,
\\
-\Delta \xi_2+V(x)\xi_2=a_{2}|u_{2}|^3+\varepsilon |u_1|+\mu \xi_2& \hbox{ in }\Omega,
\\
 \xi_1,\xi_2\geq 0 &\hbox{ in }\Omega,
 \\
 \xi_1=\xi_2=0 &\hbox{ on }\partial\Omega,
 \\
 \int_{\Omega}(\xi_{1}^2+\xi_{2}^2)=1.
 \end{array}\right.
\end{equation}
If $(u_1,u_2)\neq (0,0)$, then $a_2|u_2|^3+\varepsilon |u_1|\neq 0$. By Lemma \ref{lem3-2}, $\mu<\lambda_1$. It follows from the first  equation in \eqref{6-35} that $\xi_1=0$. Thus, $H_0^\varepsilon$ is a map from $\mathcal{H}\setminus B_\delta$ to $H_2(\Omega)$. Since $(I-H_0^\varepsilon)|_{H_2(\Omega)}=(I-H_0)|_{H_2(\Omega)}$, by the Leray-Schauder degree reduction theorem, it can be reduced as follows
\[
\begin{split}
\text{deg}(I-H_0^\varepsilon,\bar{B}_R\setminus B_\delta,0)=&\text{deg}\big( (I-H_0^\varepsilon)|_{H_2(\Omega)},(\bar{B}_R\setminus B_\delta)\cap H_2(\Omega),0\big)
\\
=&\text{deg}\big( (I-H_0)|_{H_2(\Omega)},(\bar{B}_R\setminus B_\delta)\cap H_2(\Omega),0\big),
\end{split}
\]
which, together with \eqref{6-34}, gives
\begin{equation}\label{6-36}
d_{a_1,a_2,\beta}=\text{deg}\big( (I-H_0)|_{H_2(\Omega)},(\bar{B}_R\setminus B_\delta)\cap H_2(\Omega),0\big).
\end{equation}

By \eqref{eq6-26}, \eqref{6-32}, \eqref{6-33} and \eqref{6-36}, we finally obtain that
\begin{equation}\label{6-37}
d_{a_1,a_2,\beta}^{non-trivial}=-1.
\end{equation}
As a result, \eqref{1-1}-\eqref{1-2} has a nontrivial solution.
\end{proof}

\section{Some further results}\label{8}

In this section, we briefly discuss other conditions ensuing the existence of
nontrivial solutions for \eqref{1-1}--\eqref{1-2}. Throughout this section, we assume
that $\lambda_1(V_1)=\lambda_1(V_2)$, where $\lambda_1(V)$ is the first eigenvalue of $-\Delta +V$.
With this assumption, by Lemma~\ref{l1-4-8},  $\mathbb{T}_{a_1,a_2,\beta}$ 
is well defined in $B_R\setminus B_\delta $.

\bigskip

First, we consider  the case $0\le \beta<a_1$,
$a_1\in (0,a^*)$ and $a_2\in ((m-1)a^*,ma^*)$.
Recall that $\mu^*(a_1, V_1)$ is defined in 
\eqref{20-14-8}. We also define

\begin{equation}\label{7-4}
 M_1(a_1)=\sup \big\{ \|u\|_{L^\infty(\Omega)} \; \text{$(u;\mu)$ solves \eqref{6-20}
 with $V=V_1$ and
$a=a_1$}\big\}.
\end{equation}
 Then, 
$\mu^*(a_1, V_1)<\lambda_1(V_1)$ and  $M_1(a_1)<\infty$. 

On the other hand, we let

\begin{equation}\label{n7-4}
 M_2(a_2)=\sup \big\{ \|u\|_{L^\infty(\Omega)} \; \text{$(u;\mu)$ solves \eqref{6-20}  with $V=V_2$,
 $a=a_2$}\big\}.
\end{equation}
If \eqref{6-20} with $a=a_2$ and $V=V_2$ has no solution, we regard 
 $M_2(a_2)=0$. Otherwise,  $M_2(a_2)<+\infty$.

We define

\begin{equation}\label{7-5}
\beta_0:=\min \Big\{  \frac{\lambda_1(V_1)-\mu^*(a_1, V_1)}{M_1^2(a_1)},  \frac{\lambda_1(V_2)-\mu^*(a_2, V_2)}{M_2^2(a_2)},  \min\{a_1,a_2\}       \Big\}>0,
\end{equation}
Then, we have

\begin{theorem}\label{pro7-1} Suppose that $V_1$ and $V_2$ satisfy $(V_1)$--$(V_4)$, 
$\lambda_1(V_1)=\lambda_1(V_2)$, and $k\ge 0$.
Assume  that  $a_1\in (0,a^*)$ and $a_2\in ((m-1)a^*,ma^*)$,  where $m\in \mathbb{Z}_+$. Then 
 for any $\beta\in (0,\beta_0)$,  \eqref{1-1}-\eqref{1-2} has a non-trivial solution. 
\end{theorem}

To prove Theorem~\ref{pro7-1}, we need to derive some results similar to those
in Lemmas~\ref{lem6-3} and \ref{lem6-4}. 
Firstly, we consider the following problem 
\begin{equation}\label{7-6}
\left\{
\begin{array}{ll}
-\Delta u_{1}+V_1(x)u_{1}=a_{1}u_{1}^3+\beta h(t)u_{1}u_{2}^2+\mu u_{1}& \hbox{ in }\Omega,\\
-\Delta u_{2}+V_2(x)u_{2}=a_{2}g(t)u_{2}^3+\beta h(t) u_{2}u_{1}^2+\mu u_{2}&\hbox{ in }\Omega,\\
  u_{1},u_{2}\geq 0 &\hbox{ in }\Omega,
 \\
 u_1=u_2=0 &\hbox{ on }\partial\Omega,
 \\
 \int_{\Omega}(u_{1}^2+u_{2}^2)=1, 
 \end{array}\right.
\end{equation}
where 
\begin{equation}\label{7-7}
h(t)=\max\{2t-1,0\},\;\;\; g(t)=\min\{1,2t\}.
\end{equation}

\begin{lemma}\label{lem7-2} Suppose that $\beta\in (0,\beta_0)$. Then there exists $\sigma>0$ such that for any $t\in [0,1]$ and any non-trivial solution $(u_1,u_2;\mu)$ of \eqref{7-6}, we have 
\[
\|u_2\|_{L^2(\Omega)}>\sigma. 
\]
\end{lemma}

\begin{proof}We suppose  that there exists a family of non-trivial solutions $(u_{1n},u_{2n};\mu_n)$ of \eqref{7-6} with $t=t_n\in [0,1]$ such that as $n\to \infty$, 
\[
\int_\Omega u_{1n}^2\to 1,\;\;\; \int_\Omega u_{2n}^2\to 0,\;\;\; t_n\to t_0\in[0,1]. 
\]
Taking $\varepsilon_n= \|u_{2n}\|_{L^2(\Omega)}$, then $\varepsilon_n>0$ and $\varepsilon_n\to0$. 
Let $v_n=u_{2n}/\varepsilon_n$. Then $\|v_n\|_{L^2(\Omega)}=1$.
Similar to Lemma~\ref{lem6-3}, we can prove that up to a subsequence,

\[
\begin{split}
&u_{1n}\rightharpoonup u\;\;\text{in} \;\; H_1(\Omega), \quad u_{1n}\to u\;\;\text{in} \;\; L^2(\Omega),
\\
&v_{n}\rightharpoonup v\;\;\text{in} \;\; H_2(\Omega), \quad v_{n}\to v\;\;\text{in} \;\; L^2(\Omega),
\\
&\mu_n\to \mu,\quad t_n\to t_0.
\end{split}
\]
and   $u$ is a positive solution of  \eqref{6-20} with $V=V_1$ and $a=a_1$. Thus, we have
\begin{equation}\label{7-10}
\mu\leq \mu^*(a_1, V_1), \quad \|u\|_{L^\infty(\Omega)}\leq M_1(a_1). 
\end{equation}
Moreover,  $v$ is a positive solution of 
\[
\begin{cases}
-\Delta v+V_2(x)v=\beta h(t_0)vu^2+\mu v\;\;\;\text{in}\;\;\Omega,
\\
v=0\;\;\; \text{on}\;\;\partial\Omega,
\\
\int_\Omega v^2=1,
\end{cases}
\]
which gives 
\[
\mu=\int_\Omega \big( |\nabla v|^2+V_2(x)v^2-\beta h(t_0)v^2u^2\big). 
\]
Then by \eqref{7-10} we have 
\[
\begin{split}
\lambda_1(V_1)=\lambda_1(V_2)=&\inf_{w\in H_2(\Omega),\; \|w\|_{L^(\Omega)}=1}\int_\Omega \big( |\nabla w|^2+V_2(x)w^2\big)
\\
\leq& \int_\Omega \big( |\nabla v|^2+V_2(x)v^2\big)
= \mu+\beta h(t_0)\int_\Omega v^2u^2
\\
\leq&  \mu +\beta \|u\|_{L^\infty(\Omega)}^2  \|v\|_{L^2(\Omega)}^2  \leq  \mu^*(a_1,V_1)+M_1^2(a_1)\beta. 
\end{split}
\]
This gives $\beta \geq \frac{\lambda_1(V_1)-\mu^*(a_1,V_1)}{M_1^2(a_1)}$ and it  contradicts to  $\beta<\beta_0\leq \frac{\lambda_1(V_1)-\mu^*(a_1,V_1)}{M_1^2(a_1)}$. 
\end{proof}

Now we consider the following problem 
\begin{equation}\label{7-11}
\left\{
\begin{array}{ll}
-\Delta u_{1}+V_1(x)u_{1}=a_{1}g(t)u_{1}^3+\beta h(t)u_{1}u_{2}^2+\mu u_{1}& \hbox{ in }\Omega,\\
-\Delta u_{2}+V_2(x)u_{2}=a_{2}u_{2}^3+\beta h(t) u_{2}u_{1}^2+\mu u_{2}&\hbox{ in }\Omega,\\
  u_{1},u_{2}\geq 0 &\hbox{ in }\Omega,
 \\
 u_1=u_2=0 &\hbox{ on }\partial\Omega,
 \\
 \int_{\Omega}(u_{1}^2+u_{2}^2)=1, 
 \end{array}\right.
\end{equation}
where $h(t),g(t)$ are defined in \eqref{7-7}. Using a similar argument as in Lemma \ref{lem7-2}, we can prove the following result. 
\begin{lemma}\label{lem7-3} Suppose that $\beta\in (0,\beta_0)$. Then there exists $\sigma>0$ such that for any $t\in [0,1]$ and any non-trivial solution $(u_1,u_2;\mu)$ of \eqref{7-11}, we have 
\[
\|u_1\|_{L^2(\Omega)}>\sigma. 
\]
\end{lemma}

\begin{proof}[Proof of Theorem~\ref{pro7-1}] 
From Lemmas \ref{lem7-2} and  \ref{lem7-3}, we can use a similar argument as in the proof of Theorem \ref{thm1-4} to prove that 
\begin{equation}\label{7-14}
d_{a_1,a_2,\beta}^{non-trivial}=-1. 
\end{equation}
\end{proof}

Now we turn to the discussion for  the case $a_1,a_2\in (0,a^*)$,  $\beta\in  (\beta_1^*,\beta_2^*)$.

  Let $\lambda_1^*(u,V)$ be the first eigenvalue of
$-\Delta +V(x)-\beta  u^2$. Then for each fixed $u\ne 0$, it holds
$\lambda_1^*(u, V)<\lambda_1(V)$. Define

\begin{equation}\label{ap-3}
\lambda^*_1:=\inf_{a\in[0,a^*/2]}\inf\{\lambda_1^*(u, V_2):\;\; \text{$(u;\mu)$ solves \eqref{6-20} with $V=V_1$ and $a\in[0,a^*/2]$}\},
\end{equation}
  and
 \begin{equation}\label{nap-3}
\lambda^*_2:=\inf_{a\in[0,a^*/2]}\inf\{\lambda_1^*(u, V_1):\;\; \text{$(u;\mu)$ solves \eqref{6-20} with $V=V_2$ and $a\in[0,a^*/2]$}\},
\end{equation}

By Theorem~A in the introduction,  we can prove that $\lambda^*_i$ can be achieved by some $u_i>0$ and $a_i\in[0,a^*/2]$. Hence $\lambda^*_i< \lambda_1(V_i) = \lambda_1(V_1) $ by our assumption. 

Similar to Proposition~\ref{p1-5-8}, we have the following result.

 \begin{proposition}\label{ap-pro-1}
 Assume that $a_1$, $a_2$ and $\beta$ satisfy the following conditions 
 \begin{equation}\label{ap-4}
 0<a_1,\; a_2<\min \{\frac{ a^*(\lambda_1(V_1)- \lambda^*_1) }{2 \lambda_1(V_1) }, 
 \frac{ a^*(\lambda_1(V_1)- \lambda^*_2) }{2 \lambda_1(V_1) },\frac{a^*}2\},\quad \beta_1^*<\beta<\beta_2^*.
 \end{equation}
Then there exists $\sigma>0$ such that for any sufficiently small $\varepsilon\in (0,\varepsilon_0)$ and any solution $(u_{1\varepsilon}, u_{2\varepsilon})$
  of \eqref{1-9}, it holds
\begin{equation}\label{ap-5}
\|u_{1\varepsilon}\|_{L^2(\Omega)}> \sigma,\quad \|u_{2\varepsilon}\|_{L^2(\Omega)}> \sigma.
\end{equation}

\end{proposition}

\begin{proof}
The proof of this proposition is similar to that of Proposition~\ref{p1-5-8}. Thus,
we just sketch it.

Suppose that up to a subsequence, we have
\[
\int_\Omega u_{1\varepsilon}^2\to 1, \;\;\; \int_\Omega u_{2\varepsilon}^2\to 0, \;\;\; \text{as}\;\; \varepsilon \to 0.
\]
Define $v_{\varepsilon}=u_{2\varepsilon}/ \|u_{2\varepsilon}\|_{L^2(\Omega)}$. 
As in  the proof of Proposition \ref{p1-5-8}, we can prove that 
$u_{1\varepsilon}$ is bounded in $H_1(\Omega)$ and $v_\varepsilon$  is bounded in $H_2(\Omega)$. So up to a subsequence, we can assume that $u_{1\varepsilon}\rightharpoonup u_1$ in $H_1(\Omega)$, and $v_{\varepsilon}\rightharpoonup v$ in $H_2(\Omega)$. Similar to \eqref{6-9} and \eqref{6-12}, we have 
\begin{equation}\label{ap-6}
\begin{cases}
-\Delta u_1+V_1(x)u_1=a_1u^3_1+\mu u_1\;\;\;\text{in}\;\;\Omega,
\\
-\Delta v+V_2(x)v=\beta v u_1^2+\sigma_0 u_1+\mu v\;\;\;\text{in}\;\;\Omega,
\\
u_1>0,v>0\;\; \text{in}\;\; \Omega,\quad u_1=v=0\;\;\; \text{on}\;\;\partial\Omega,
\\
\int_\Omega u_1^2=\int_\Omega v^2=1,
\end{cases}
\end{equation}
where $\sigma_0\geq 0$ is a constant. 
This shows that 
\[
\mu=\int_\Omega \big(|\nabla u_1|^2+V_1(x) u_1^2- a_1u_1^4\big)=\int_\Omega \big(|\nabla v|^2+V_2(x) v^2- \beta u_1^2v^2-\sigma_0 u_1v\big). 
\]

 From the first equation of \eqref{ap-6}, we have
 
 \begin{equation}\label{1-17-8}
\begin{split}
\Bigl( 1-\frac {2a_1}{a^*}\Bigr)\lambda_1(V_1)
\leq & \Bigl( 1-\frac {2a_1}{a^*}\Bigr)\int_\Omega \big(|\nabla u_1|^2+V_1(x) u_1^2\big)\\
\le &\int_\Omega \big(|\nabla u_1|^2+V_1(x) u_1^2\big)-a_1 \int_\Omega u_1^4=
\mu,
\end{split}
\end{equation}
which gives

\begin{equation}\label{2-17-8}
\begin{split}
a_1\ge \frac{a^*(\lambda_1(V_1)-\mu)}{2\lambda_1(V_1)}.
\end{split}
\end{equation}

On the other hand, in view of  $v>0$, from the second equation of \eqref{ap-6},  we have 
$
\mu\le \lambda^*_1.
$
Thus,  

\[
a_1\ge \frac{a^*(\lambda_1(V_1)-\mu) }{2\lambda_1(V_1)}\ge \frac{a^*(\lambda_1(V_1)-\lambda^*_1)}{2\lambda_1(V_1)}.
\]
This is a contradiction.

Hence, we have proved that $\|u_{2\varepsilon}\|_{L^2(\Omega)}> \sigma$.
We can also prove $\|u_{1\varepsilon}\|_{L^2(\Omega)}> \sigma$ in a similar way.
\end{proof}

With Proposition~\ref{ap-pro-1}   and Theorem~\ref{thm1-2}, we obtain
the following existence result.

\begin{theorem}\label{th8-1} Suppose that $V_1$ and $V_2$ satisfy $(V_1)$--$(V_4)$, 
$\lambda_1(V_1)=\lambda_1(V_2)$, and $k>0$.
Assume  that   

\[
0<a_1,\; a_2<\min \{\frac{ a^*(\lambda_1(V_1)- \lambda^*_1) }{2 \lambda_1(V_1) }, 
 \frac{ a^*(\lambda_1(V_1)- \lambda^*_2) }{2 \lambda_1(V_1) },\frac{a^*}2\},\quad \beta_1^*<\beta<\beta_2^*.
\]
Then 
   \eqref{1-1}-\eqref{1-2} has a non-trivial solution. 
\end{theorem}

 \appendix \section{A lower bound of $L^2$-norm of NLS system}\label{7}

In this section, we consider the lower bound of $L^2$-norm of  solutions to the following nonlinear Schr\"odinger equations
\begin{equation}\label{A-1}
\begin{cases}
-\Delta u+c_1u=a_1u^3+\beta uv^2\quad \text{in}\;\; \mathbb{R}^2,
\\
-\Delta v+c_2v=a_2v^3+\beta vu^2\quad \text{in}\;\; \mathbb{R}^2,
\end{cases}
\end{equation}
where $a_1,a_2,c_1,c_2,\beta$ are positive constants.

\begin{lemma}\label{lemA-1} Let  $a_1,a_2,c_1,c_2,\beta$ be positive constants. If $(u,v)\in H^1(\mathbb{R}^2)\times H^1(\mathbb{R}^2)$ is a nontrivial  solution of \eqref{A-1}, then we have
\begin{equation}\label{A-2}
\int_{\mathbb{R}^2} (u^2+v^2)\geq \frac{a^*}{\max\{a_1,a_2,\beta\}}.
\end{equation}
\end{lemma}

\begin{proof}Multiplying the first equation of \eqref{A-1} by $u$ and the second equation by $v$ and integrating by parts, we get
\[
\int_{\mathbb{R}^2} (|\nabla u|^2+ |\nabla v|^2)+\int_{\mathbb{R}^2}(c_1u^2+c_2v^2)=\int_{\mathbb{R}^2}(a_1u^4+a_2v^4+2\beta u^2v^2).
\]
Multiplying the first equation of \eqref{A-1} by $x\cdot \nabla u$ and the second equation by $x\cdot \nabla v$ and integrating by parts, we obtain the following Pohozaev identity
\[
2\int_{\mathbb{R}^2}(c_1u^2+c_2v^2)=\int_{\mathbb{R}^2}(a_1u^4+a_2v^4+2\beta u^2v^2).
\]
So we obtain
\[
\int_{\mathbb{R}^2} (|\nabla u|^2+ |\nabla v|^2)=\frac{1}{2}\int_{\mathbb{R}^2}(a_1u^4+a_2v^4+2\beta u^2v^2).
\]

Note that
\[
a_1u^4+a_2v^4+2\beta u^2v^2\leq \max\{a_1,a_2,\beta\} (u^4+v^4+2u^2v^2)\leq \max\{a_1,a_2,\beta\} (u^2+v^2)^2.
\]
Let $U=\sqrt{u^2+v^2}$. We have
\[
|\nabla U|^2=\Big|\frac{u}{\sqrt{u^2+v^2}} \nabla u+\frac{v}{\sqrt{u^2+v^2}} \nabla v\Big|^2\leq |\nabla u|^2+|\nabla v|^2.
\]
By the Gagliardo-Nirenberg inequality,
\[
\begin{split}
\int_{\mathbb{R}^2}(u^2+v^2)^2=\int_{\mathbb{R}^2}U^4\leq \frac{2}{a^*} \int_{\mathbb{R}^2} |\nabla U|^2\int_{\mathbb{R}^2} U^2
\leq \frac{2}{a^*} \left(\int_{\mathbb{R}^2}(|\nabla u|^2+|\nabla v|^2) \right) \left(\int_{\mathbb{R}^2}(u^2+v^2)\right).
\end{split}
\]
So we get
\[
\begin{split}
\int_{\mathbb{R}^2} (|\nabla u|^2+ |\nabla v|^2)=&\frac{1}{2}\int_{\mathbb{R}^2}(a_1u^4+a_2v^4+2\beta u^2v^2)
\\
\leq& \frac{1}{2} \max\{a_1,a_2,\beta\}\int_{\mathbb{R}^2}  (u^2+v^2)^2
\\
\leq& \frac{\max\{a_1,a_2,\beta\}}{a^*} \left(\int_{\mathbb{R}^2}(|\nabla u|^2+|\nabla v|^2) \right) \left(\int_{\mathbb{R}^2}(u^2+v^2)\right),
\end{split}
\]
which gives \eqref{A-2}.
\end{proof}

\section{Classification of positive solutions of  NLS system}\label{7}

In this section, we recall  the uniqueness of   solutions of the following nonlinear Schr\"odinger 
system \cite{clwy,wy}:

\begin{equation}\label{B-1}
\begin{cases}
-\Delta u_1+u_1=a_1u_1^3+\beta u_1u_2^2\quad \text{in}\;\; \mathbb{R}^2,
\\
-\Delta u_2+u_2=a_2u_2^3+\beta u_2u_1^2\quad \text{in}\;\; \mathbb{R}^2,
\\
u_1,u_2\ge 0,
\end{cases}
\end{equation}
where $a_1,a_2,\beta$ are positive constants.

\begin{theorem}\label{lemB-1} Suppose that either $\beta>\max(a_1, a_2)$, or $0<\beta<\min(a_1, a_2)$.
Then
up to a translation, any solution $(u_1,u_2)$ of \eqref{B-1} with $u_1>0$ and
$u_2>0$ has  the form
\[
(u_1,u_2)=\left(\sqrt{\frac{\beta-a_2}{\beta^2-a_1a_2}}Q(x),\sqrt{\frac{\beta-a_1}{\beta^2-a_1a_2}}Q(x)\right).
\]
\end{theorem}

 For the existence of least energy solutions for nonlinear Schr\"odinger 
system, the readers can  refer to \cite{s}. Early uniqueness of   solutions
for \eqref{B-1} in the case $\beta<\min(a_1, a_2)$ can be found in \cite{cz}.
We also mention the result in \cite{bw}. If $0\le \min(a_1, a_2)\le \beta \le \max(a_1, a_2)$,
\eqref{B-1} has a solution $(u_1,u_2)$  with $u_1>0$ and
$u_2>0$ if and only if $a_1=a_2=\beta$.

\end{document}